\documentclass[11pt]{article}

\usepackage{scalerel}
\usepackage{subcaption}
\usepackage{adjustbox}

\usepackage{graphicx}

\usepackage[a4paper]{geometry}
\usepackage{rotating}
\usepackage{pdflscape}
\usepackage[T1]{fontenc}
\usepackage{tabularx}
\usepackage{amsfonts}
\usepackage{amsmath}% or letter or a5paper or ... etc
\usepackage{mathabx}
\usepackage{hyperref}
\usepackage{xcolor}
\usepackage{amssymb}
\usepackage{color}
\usepackage{tikz}
\usepackage{mathrsfs}
\usepackage{stmaryrd}
\usepackage{bussproofs}
\usepackage[numbers]{natbib}
\usepackage{url}
\usepackage{mathtools}

\newcommand{\curc}{\mathord{\circ}}
\newcommand{\dsarrow}{\mathord{\text{\rotatebox[origin=c]{270}{$\rightsquigarrow$}}}}

\usepackage{amsthm}% or letter or a5paper or ... etc

\theoremstyle{plain}
\newtheorem{definition}{Definition}[section]

\newtheorem{lemma}[definition]{Lemma}
\newtheorem{proposition}{Proposition}
\newtheorem{corollary}[definition]{Corollary}
\newtheorem{fact}[definition]{Fact}

\theoremstyle{definition}
\newtheorem{remark}[definition]{Remark}

\numberwithin{equation}{section}



\newcommand{\B}[1]{\mathbf{#1}}
\newcommand{\C}[1]{\mathcal{#1}}
\newcommand{\D}[1]{\mathscr{#1}}
\newcommand{\OV}[1]{\mathbin{\overline{#1}}}

\newcommand{\NI}{\mathtt{NI}}
\newcommand{\Nd}{\mathtt{NI}^2}

\newcommand{\at}[1]{\mathsf{at}(#1)}

\newcommand{\ho}[1]{\{ #1\}}

\newcommand{\Np}{\mathtt{NI}^p}
\newcommand{\Ndp}{\mathtt{NI}^{2p}}

\newcommand{\Nb}{\mathtt{NI}^\bullet}
\newcommand{\Ndb}{\mathtt{NI}^{2\bullet}}

\newcommand{\Nv}{\mathtt{NI}^{\!\vee}}
\newcommand{\Ndv}{\mathtt{NI}^{2\vee}}

\newcommand{\Nda}{\mathtt{NI}^2_{at}}

\newcommand{\Ndba}{\mathtt{NI}^{2\bullet}_{at}}
\newcommand{\Ndva}{\mathtt{NI}^{2\vee}_{at}}
\newcommand{\Ndpa}{\mathtt{NI}^{2p}_{at}}

\newcommand{\Ndrp}{\mathtt{NI}^2_{RP}}

\newcommand{\Ndrpc}{\mathtt{NI}^2_{\curc}}

\newcommand{\Ldrp}{\mathcal{L}^2_{RP}}

\newcommand{\Ld}{\mathcal{L}^2}

\newcommand{\Lp}{\mathcal{L}^p}
\newcommand{\Lb}{\mathcal{L}^\bullet}
\newcommand{\Lv}{\mathcal{L}^\vee}

\newcommand{\Ldv}{\mathcal{L}^{2\vee}}
\newcommand{\Ldb}{\mathcal{L}^{2\bullet}}
\newcommand{\Ldp}{\mathcal{L}^{2p}}

\newcommand{\lost}[3]{\left\langle{#1}\right\rangle%_{\scalerel*{{#2}\in{#3}}{X}}
} 
\newcommand{\lust}[3]{\left\langle{#1}\right\rangle_{{#2}\in{#3}}}

\newcommand{\EXP}[3]{\langle{#3}\rangle_{#2}\! \supset\! {#1}} 

\newcommand{\EXPO}[3]{\langle{#3}\rangle\! \supset%_{#2}
  \! {#1}}

\usepackage{varwidth}
\newsavebox{\mypti}
\newsavebox{\mypto}
\newsavebox{\mypta}
\newsavebox{\myptu}
\newsavebox{\mypte}

\usepackage{tikz}
\usetikzlibrary{positioning,shapes,shapes.multipart}
\usetikzlibrary{cd,arrows,decorations,backgrounds}

\tikzset{
  symbol/.style={
    draw=none,
    every to/.append style={
      edge node={node [sloped, allow upside down, auto=false]{$#1$}}}
  }
}

\newcommand{\bb}[1]{{#1}}
\newcommand{\rr}[1]{{#1}}
\newcommand{\TT}[1]{\mathtt{#1}}

\newcommand{\imp}{\mathbin{\supset}}

\newcommand{\rop}{*}

\newcommand*\SF[3]{%
\begin{tikzpicture}[baseline=(C.base)]
    \node[rectangle split, rectangle split horizontal, rectangle split parts=2, rectangle split part fill={black!10,black!10} , inner sep=1pt,
](C) {${#2}$ \nodepart{two}  \colorbox
{white}{$%\left\ldbrack
    {#3}
    $}};
  \end{tikzpicture}
}

\newcommand*\cexpD[3]{%
\begin{tikzpicture}[baseline=(C.base)]
    \node[rectangle split, rectangle split horizontal, rectangle split parts=2, rectangle split part fill={black!10,black!10} , inner sep=1pt
    ](C) {${#2}$ \nodepart{two}   \colorbox{white}{$%\left\ldbrack\!\!\!
    {#3}%\!\!\!
    $}};
  \end{tikzpicture}
}

\newcommand*\Fun[3]{%
\begin{tikzpicture}[baseline=(C.base)]
    \node[rectangle split, rectangle split horizontal, rectangle split parts=2, rectangle split part fill={black!10,black!10} , inner sep=1pt,
](C) {${#2}^{\scaleto{#1}{3.5pt}}$ \nodepart{two}  \colorbox
{white}{$%\left\ldbrack
    {#3}
    $}};
  \end{tikzpicture}
}

\usepackage{authblk}

\title{The naturality of natural deduction (II)\\ \Large Some remarks on atomic polymorphism}
\author[1]{Paolo Pistone}
\author[2]{Luca Tranchini}
\author[3]{Mattia Petrolo}
\affil[1]{\small\emph{Dipartimento di Informatica-Scienza e Ingegneria, Universit\`a di Bologna}}
\affil[2]{\small\emph{Wilhelm-Schickard-Institut, Universit\"at T\"ubingen}}
\affil[3]{\small\emph{Centro de Ci\^encias Naturais e Humanas, Universidade Federal do ABC}}

\date{}

\begin{document}

\EnableBpAbbreviations

\maketitle

\begin{abstract}
In a previous paper (of which this is a prosecution) we investigated the extraction of proof-theoretic properties of natural deduction derivations from their  impredicative translation into System F. Our key idea was to introduce an extended equational theory for System F codifying at a syntactic level some properties found in parametric models.

In a recent series of papers a different approach to extract proof-theoretic properties of natural deduction derivations was proposed by defining predicative variants of the usual translation, embedding intuitionistic propositional logic into the atomic fragment of System F.

In this paper we show that this approach  finds a general explanation within our equational study of second-order natural deduction, and a clear semantic justification provided by parametricity.

\end{abstract}

\paragraph{Keywords} second-order logic, propositional quantification, identity of proofs,  Russell-Prawitz translation, atomic polymorphism, naturality condition, instantiation overflow.

\section{Introduction}\label{intro}
%\subsection{The Russell-Prawitz translation}

%
%Russell is usually referred to as the first to observe that propositional connectives like conjunction and disjunction can be defined using only implication and propositional quantification. 
%This observation was later adapted by Prawitz into a derivability-preserving mapping (sometimes referred to as the \emph{Russell-Prawitz translation}, henceforth RP-translation) from natural deduction systems for propositional logic into natural deduction systems for second-order logic (or, in more precise terms, into System F). 
%% Prawitz is usually referred to as the first to observe this fact, and the translation is often referred to as the 
%  \emph{Russell-Prawitz translation} (henceforth RP-translation). 
%
 % 
%
%
%$\NI$ into System F (also referred to as $\Nd$, e.g.~\cite{TS96}) is usually called the \emph{Russell-Prawitz translation} (henceforth RP-translation). 

Russell was the first to observe that propositional connectives like disjunction and conjunction can be defined using only implication and propositional quantification and %in his monograph on natural deduction, Prawitz spelled out the definitions in the natural deduction formalism. Using Prawitz's definitions, it is easy to define a derivability-preserving embedding of the natural deduction system for intuitionistic propositional logic (henceforth $\NI$)  into the implicational fragment of second-order propositional intuitionistic logic (nowadays known as {\em System $F$}, here $\Nd$), sometimes referred to as the \emph{Russell-Prawitz translation} (shortly, RP-translation).
%
% In
in his monograph on
natural deduction, Prawitz showed how the natural deduction system for intuitionistic
propositional logic (henceforth $\NI$) can be embedded into the implicational fragment of
second-order propositional intuitionistic logic (also known as {\em System F}, and here referred to as $\Nd$). We will refer to this embedding  as the Russell-Prawitz translation
(shortly RP-translation).% is the

%Although

Taking inspiration from this embedding, in recent work  in proof-theoretic semantics (see for instance in \cite{OSxx} and \cite{Ferreira2013}) $\Nd$ has been suggested as a suitable setting to investigate the proof theory of propositional connectives. This way of looking at $\Nd$ faces however two  kinds of difficulties.

%some recent directions of research on proof-theoretic semantics have investigated $\Nd$ and related systems as a convenient framework \PP{un altro termine?} to investigate propositional connectives and their provability conditions

%have investigated the possibility of considering $\Nd$ as a convenient 
% 
%  $\Nd$ has been taken as a vantage point for investigating propositional connectives  (see for instance  \cite{OSHxx}), two kinds of difficulties seem to undermine the adoption of $\Nd$ as the natural environment  for such investigations.
%
%point   recently The existence of a ``predicative'' encoding of propositional connectives has been welcomed with enthusiasm in the proof-theoretic semantics community, since it provides further evidence in favour of using $\Nd$ as a vantage point for investigating propositional connectives (as done for instance in \cite{}).

\paragraph{The equivalence-preservation problem}

In proof-theoretic semantics, one is not only concerned with \emph{provability}---i.e.~with whether there is a derivation of a given formula in a certain system---but also with \emph{identity of proofs}---i.e.~with whether  %nature of \emph{derivations} themselves. 
% In that respect, a typical problem is to be able to guess whether
two distinct derivations of the same formula can be viewed as different syntactic representations of \emph{the same} proof (understood as an abstract object).

%In proof-theoretic semantics, one is not only concerned with the  question  of \emph{provability}---i.e.~of whether there is a derivation of a given formula in a certain system---but also with the one of \emph{identity of proofs}---i.e.~of whether  %nature of \emph{derivations} themselves. 
% In that respect, a typical problem is to be able to guess whether
%two distinct derivations of the same formula can be viewed as different syntactic representations of \emph{the same} proof (understood as an abstract object). 
 
 A common way to characterize identity of proofs is 
  %An answer to the latter question can be attained by first defining 
%  through 
%  an equivalence relation on derivations---
  by declaring two derivations \emph{equivalent}  when they converge, under the usual conversions used for normalization, to the same normal derivation. Equivalent derivations are then taken to represent the same proof. 
  This intuition is made precise by the categorical semantics for natural deduction systems. For instance
  $\NI$ can be interpreted in any \emph{bi-cartesian closed category}, 
 with equivalent derivations being mapped onto the same morphism.
  %categorical semantics of $\NI$ this idea is made  fully precise in that
%   equivalent derivations are be interpreted as morphism in any \emph{bi-cartesian closed category}.\PP{il seguito lo toglierei}, i.e.~any category with finite products/co-products and with an internal Hom-object (of which the simplest example is the category of sets). %For this reason this equivalence relation is sometimes referred to as \emph{identity of proofs} \cite{Dosen2003}.
%  \PP{non esiste l'equivalente delle bi-ccc per System F, giusto?}

%  
%  
%  In the usual natural deduction system for intuitionistic propositional logic (henceforth $\NI$), there are three families of such conversions: $\beta$-conversions (those eliminating intro/elim patterns), $\eta$-conversions (referred to as ``immediate expansions'' in \cite{Prawitz1971}) and $\gamma$-conversions (called ``permutative'' or ``commuting'' conversions).
%  
%  

If not only provability but also identity of proofs is considered, then 
the RP-translation might not seem entirely satisfactory
% seems, at first glance, to be wanting%. %to fail to provide a faithful translation of propositional proofs inside System F.
% In fact
as  equivalent  derivations in $\NI$ need not translate into equivalent derivations in $\Nd$. Although the translation works for the equivalence induced by  $\beta$-conversions only, it fails for the one induced by $\eta$-conversions and permutations, here referred to as $\gamma$-conversions. 
In categorical terms, %this means that the second-order formula corresponding to
the RP-translation of, say, a disjunction, is not interpreted as a {co-product} in every categorical model of $\Nd$, but only as a ``weak'' variant of it. 
We will refer to this fact as the  \emph{equivalence-preservation} problem of the RP-translation.

%

%
%The equivalence relation induced by these conversions plays an essential role in the categorical  semantics of $\NI$, as equivalent derivations are interpreted by the same entity, namely a morphism in a  \emph{bi-cartesian closed category}. For this reason this equivalence relation is sometimes referred to as \emph{identity of proofs} \cite{Dosen2003}.
%

% In categorical terms, the interpretation  of conjunction and disjunction via the  RP-translation does not satisfy all the laws of the product and coproduct in a bi-cartesian closed category (but only ``weak'' versions of some of them).

In a  previous paper, of which the present one is a follow-up, we explored a solution to the equivalence-preservation problem based on the  fact that the RP-translation of conjunctions and disjunctions \emph{does} yield categorical products/co-products in the class of \emph{parametric models} of $\Nd$  \cite{Reynolds1983,Bainbridge1990, Girard1992}. 
%A well-known semantic solution to the equivalence-preservation problem arises from the categorical semantics of System F (see \cite{Bainbridge1990, Girard1992}), in which derivations in
%$\Nd$ are interpreted as \emph{dinatural transformations}.
With the goal of making this result accessible to the proof theory community at large, in \cite{StudiaLogica} we provided a purely syntactic reconstruction of it: we described an equational theory extending the one arising from the usual  $\beta$- and $\eta$-conversions   for $\Nd$-derivations using a new class of conversions---that we called $\varepsilon$-conversions---expressing a {\em naturality} condition for $\Nd$-derivations that holds in all {\em parametric} models of $\Nd$, and we showed that the RP-translation does preserve the full 
equivalence of $\NI$-derivations as soon as $\Nd$-derivations are considered under this stronger equivalence.\footnote{In the functorial semantics of System F \citep{Bainbridge1990,Girard1992}, $\Nd$-derivations are actually interpreted as {\em di}natural transformation. The reason to focus on naturality, rather than on the more general notion of dinaturality, is briefly discussed in Section~\ref{conc-section}.}

%\footnote{This equivalence strictly extends the equivalence induced by $\beta$- and $\eta$-
%  conversions. In fact, while $\beta\eta$-equivalence is the maximum consistent equivalence for the
%  implicational fragment of $\NI$,   the equivalence induced by
%  $\beta$- and $\eta$- conversions in $\Nd$ is by no means the maximum
%  consistent equivalence definable on $\Nd$-derivations.}. 
% 
% 
% 
% for System F derivations  gave a
%syntactic formulation of a naturality condition satisfied by some of these
%transformations in the form of a particular class of equations (called
%$\varepsilon$-equations) among derivations of $\Nd$. Then we
%showed that, when the equivalence on $\Nd$-derivation is extended with
%the $\varepsilon$-equations,\footnote{Whereas the equivalence induced by $\beta$- and $\eta$-
%  conversions is the maximum consistent equivalence for the
%  implicational fragment of $\NI$,   the equivalence induced by
%  $\beta$- and $\eta$- conversions in $\Nd$ is by no means the maximum
%  consistent equivalence definable on $\Nd$-derivations.} the RP-translation does preserve the
%equivalence of $\NI$-derivations.

\paragraph{Impredicative vs predicative translations}%
%\PP{Non mi torna qui. L'impredicativita non vedo in cosa sia un problema. O meglio, si puo dire che e un problema da un punto di vista fondazionale, ma per i contenuti dell'articolo che ci importa? Anzi il nostro punto di vista qui e impredicativo fino al midollo...}
%In the community of type theory, 

 A second difficulty is of a foundational nature and stems from the fact that the RP-translation and, more generally, the second-order encoding of \emph{inductive types} (e.g. the types of natural numbers and well-founded trees) inside $\Nd$ are \emph{impredicative}. In fact, the embedding of $\NI$ into $\Nd$
% as they
%, are usually referred to as the {\em impredicative encodings}. In fact, 
%
% is a special case of a more general type-theoretical embedding known as the {\em impredicative encoding} of inductive types (e.g. the types of natural numbers and well-founded trees).
  %The label ``impredicative'' points to the fact that, in general, 
 % such encodings 
  requires the full power of second-order quantification: in the elimination rule for the second-order quantifier $\forall$E:
\begingroup\makeatletter\def\f@size{10}\check@mathfonts
$$
\AXC{$\forall X.A $}
\RL{$\forall\text{E}$}
\UIC{$A\llbracket B/X\rrbracket$}
\DP
$$\endgroup
no restriction can be imposed on the choice of the formula $B$ (called the {\em witness} of the rule application).
%Hence, in sharp contrast with the finitary character of $\NI$, 

%In spite of the extensive literature on  predicative fragments of $\Nd$ and of their expressive power, it 
A solution to this problem can be found in a recent series of papers by Fernando Ferreira and Gilda Ferreira, who proposed a variant of the RP-translation (to which we will refer to as FF-translation) which encodes $\NI$ in \emph{atomic} System F (here referred to as $\Nda$), a weak predicative fragment of $\Nd$ in which the  witnesses of $\forall$E are required to  be atomic formulas. A further refinement of the FF-translation was later proposed by Jos\'e Esp\'irito Santo and Gilda Ferreira in \cite{ESF19} (we will refer to it as the ESF-translation).

% (as done for instance in \cite{}).
Besides being predicative, the FF- and ESF-translations have another significant advantage over the RP-translation:  
%:, the FF-translation and the ESF-translations
they do preserve the equivalence arising  
% Ferreira and Ferreira devised a different strategy to extract proof-theoretic properties of $\NI$-derivations from their second-order translation. They proposed a variant of the RP-translation (that we will refer to as the FF-translation)
% which preserves the equivalence arising
not only from $\beta$-conversions, but also from %permutative conversions  (here called $\gamma$-conversions)  and
$\eta$- and $\gamma$-conversions
  for disjunction and $\bot$ \cite{Ferreira2009,Ferreira2017,ESF19}.\footnote{Actually, preservation of $\eta$-conversions fails for conjunction.} 
 For these reasons the predicative translations  were advocated in \cite{Ferreira2013} 
 as evidence in favor of taking $\Nd$ and its fragments as a convenient framework to investigate propositional connectives.

\paragraph{From impredicative to atomic polymorphism through $\varepsilon$-conversions}
%\PP{Non ci siamo ancora.
%
%Bisogna dire prima che sia noi che FF proponiamo Nd come framework per NI, mostrando soluzioni all'equivalence-preservation problem. Poi che il nostro framework spiega il loro, ma non viceversa.
%}
%There are however several aspects of the approach of  Ferreira and co-authors that stand in need of clarification.

%\PP{
%The results of FF show something interesting, but they are very syntactic and while they seem very relevant for identity of proof, they lack a clear semantic intuition, and a connection with the semantic analysis of second-order logic/polymorphism.
%}

%We believe that the results on atomic polymorphism just recalled are interesting both from the viewpoint of the proof-theorist looking for a satisfactory ``metatheory'' for propositional logic and from the viewpoint of the 
%proof-theorist interested in second-order logic and its well-known connections with the theory of programming languages (through the study of parametric polymorphism and its models).
%However, such results are very syntactic and a clear semantic intuition, as well as a connection with the rich literature on identity of proof and categorical semantics of $\Nd$ and related systems.

%\PP{In this paper we develop an approach to the predicative translations based on the equational theory we investigated in our previous paper}.

In this paper we investigate the predicative translations into System F$_{at}$ using the
 equational framework we developed in our previous paper, and we show that  
 the syntactic results on atomic polymorphism can be given a semantic explanation ultimately relying on parametricity, a well-investigated semantics of (full) polymorphism.
%
% 
%can be used to shed light on the predicative translations and their proof-theoretic properties. 
%Thus, our approach complements the syntactic results on atomic polymorphism with a clear semantic justification based on parametricity, a well-investigated semantics of (full) polymorphism.

%  their proof-theoretic  
%properties  
%%results on atomic polymorphism  
%find a general explanation \PP{che vuol dire che le pred trans sono spiegate dal framework equazionale?} within the
%equational framework for natural deduction we developed in \cite{StudiaLogica}, as these translations can be analyzed through
%

%
%  atomic polymorphism finds a general explanation within our study of generalized permutations in $\Nd$, with a clear semantic justification based on parametricity, a well-investigated semantics of polymorphism.
%

Our first observation is that the results of Ferreira and co-authors do not hold only for $\vee$ and $\bot$, but for the class of connectives that are definable in $\NI$ by arbitrarily composing $\wedge$, $\vee$, $\top$ and $\bot$ (to describe such connectives we borrow another concept from the toolbox of category theory, that of a \emph{finite polynomial functor}  \cite{Kock}).

By extending the RP-translation to a natural deduction system for this class of  propositional connectives (called $\Np$), we are led to consider another fragment of $\Nd$, that we call the \emph{Russell-Prawitz fragment} ($\Ndrp$). Unlike the atomic fragment $\Nda$, the fragment $\Ndrp$ is impredicative, since no restriction is imposed on the witnesses of the applications of $\forall$E.

Nonetheless, we show that every derivation in $\Ndrp$ can be ``atomized'', i.e.~it can be mapped onto a derivation in $\Nda$ with the same conclusion and the same assumptions, by applying instances of $\eta$-expansion and the $\varepsilon$-conversion. By composing the RP-translation from $\Np$ to $\Ndrp$ with the atomization from $\Ndrp$ to $\Nda$ one thereby obtains another predicative translation from $\Np$ to $\Nda$ (we call it the \emph{$\varepsilon$-translation}) which only differs from the FF- and ESF-translation by some $\beta$-reduction steps.

% In particular, the (generalization of the) FF-translation from $\Np$ to $\Nda$ can be analyzed as the result of composing the RP-translation from $\Np$ to $\Ndrp$ with the atomization from $\Ndrp$ to $\Nda$.
%
%
%Nonetheless, we show that every derivation in $\Ndrp$ can be ``atomized'', i.e.~it can be mapped onto a derivation in $\Nda$ with the same conclusion and the same assumptions. In particular, the (generalization of the) FF-translation from $\Np$ to $\Nda$ can be analyzed as the result of composing the RP-translation from $\Np$ to $\Ndrp$ with the atomization from $\Ndrp$ to $\Nda$.
%
%A connection between the predicative translations and 
%our study of generalized permutations in $\Nd$ is obtained by defining an alternative atomization obtained by applying instances of $\eta$-expansion and of the $\varepsilon$-conversion. By composing the RP-translation with this alternative atomization, we obtain a third translation of $\Np$ into $\Nda$, we call it the $\varepsilon$-translation, which is  $\beta$-equivalent both to the  FF-translation and  the ESF-translation (scaled to $\Np$).
%\PP{Cosi mi sembra un po una lista di risultati tecnici. Non si capisce bene perche dovrebbero interessare qualcuno, ne che rilevanza abbiano per i due problemi menzionati sopra. Dovrebbe emergere secondo me che i nostri risultati spiegano qualcosa che nell'approccio FF non e per niente chiaro. }

An immediate consequence of this fact is that the RP-translation and its three predicative variants are all equivalent modulo $\beta$-, $\eta$- and $\varepsilon$-conversions. 
%
%
% the $\varepsilon$-translation is $\eta\varepsilon$-equivalent to the RP-trasnlation, the latter is $\beta\eta\varepsilon$-equivalent both to the FF- and to the ESF-translations.
%
%In this paper we show that the extraction of proof-theoretic properties of $\NI$-derivation using atomic polymorphism finds a general explanation within our study of generalized permutations in $\Nd$, with a clear semantic justification based on parametricity, a well-investigated semantics of polymorphism.
%
%In the previous paper we had shown that the $\varepsilon$-rule plays an essential role to capture the equivalence arising from $\eta$- and $\gamma$-conversions. Our fundamental observation here is that this rule also accounts for the construction of the predicative variants of the RP-translation, as the former can be obtained from the latter by applying instances of these rules (up to some $\beta$ and $\eta$-conversion step)
%A consequence of this fact is that
Thus, under the notion of identity of proofs induced by $\varepsilon$-conversions, all translations of a given propositional derivation are different syntactic descriptions of the same second-order proofs (that is, all these translations interpret an $\Np$-derivation as the same morphism % equivalent %from the viewpoint of naturality conditions for natural deduction
% (and more generally, equivalent
in all parametric models of $\Nd$).

\paragraph{Comparing predicative translations and $\varepsilon$-conversions}
%${\phantom{A}}$
%Our analysis shows that, for what concerns the identity of proofs relation, the results stemming from the approach based on atomic polymorphism can be reframed and further generalized within our study of generalized permutations \PP{non si e' ancora usata la parola permutation. 
%Forse parlerei in semantichese con parametricity?}in $\Nd$. Moreover, 
On the one hand, we highlight two limitations of the approach based on atomic polymorphism: first, 
 the predicative translations do not preserve the full $\eta$-rule needed to interpret disjunction as a categorical co-product, and thus
fail to provide 
a full solution to the equivalence-preservation problem.
%, since both the FF- and ESP-translations fail to preserve the full $\eta$-rule needed to interpret disjunction as a categorical co-product (which is instead preserved modulo the $\varepsilon$-rule). 
Moreover, we show that once $\forall$E is restricted to atomic witnesses it is not possible to prove the logical equivalence between a propositional formula and its second-order translation (with the terminology of \cite{Prawitz1965}, this means that connectives are not \emph{strongly} definable, but only \emph{weakly} definable in $\Nda$).

On the other hand, we observe that while the predicative translations are well-suited for the study of proof reductions, as shown for instance by the results in \cite{ESFarx}, 
the use of $\varepsilon$-conversions comes at the price of a rather involved and still not well-understood reduction behavior.

\paragraph{Goals and plan of the paper} One of the motivations for the previous and present papers is that of making some ideas underlying the categorical semantics of System $F$ accessible to the  proof-theoretic community at large, and to show that these ideas can be fruitfully connected with strands of research arisen within more philosophically-oriented areas of proof theory. This was the reason for reformulating in the first paper categorical notions such as functors and natural transformations in the language of natural deduction, at the expenses of typographic conciseness. 

To keep the presentation compact and readable for the largest audience, we chose to present the main results of the paper using the natural deduction notation and restricting the attention only to the case of a particular ternary connective $\bullet(A,B,C)$. %. Several
 %full proofs in appendix %be postponed in the appendix where we use
Full  proofs for the whole class of  connectives we consider are postponed to a (large) technical appendix written using the drastically more economical $\lambda$-calculus notation.

In Section~\ref{preliminaries}, we introduce the natural deduction calculus $\Np$ for the class of propositional connectives we intend to investigate and a fragment of $\Nd$, that we call the \emph{Russell-Prawitz fragment} (noted  $\Ndrp$). Both $\Np$ and $\Ndrp$ are inspired by the notion of finite polynomial functor from category theory, and we introduce a generalization of the usual RP-translation as a derivability-preserving embedding between these two systems.
In Section~\ref{sec:epsilon} we recall the framework introduced in our previous paper to describe functors and natural transformations within natural deduction, based on the $\varepsilon$-conversion. % and show how it scales to $\Np$.
In Section~\ref{sec-ferris} we generalize the FF- and the ESF-translations to $\Np$  and  we investigate their relationship to the RP-translation. To do this, we  first show how the FF-translation can be analyzed as the composition of the RP-translation and of an embedding from the fragment $\Ndrp$ into  $\Nda$ that we call FF-atomization, and then defining an alternative embedding from $\Ndrp$ into  $\Nda$ atomization using the $\varepsilon$-conversions, the $\varepsilon$-atomization.
%
%show that the first $\beta$-reduces to the second  and the second to the third; then we show that  the $\varepsilon$-translation is %$\beta$-equivalent to the original FF-translation;  finally %. In Section~\ref{rp-ff-section}
%% we show that our modified FF-translation is
%$\eta\epsilon$-equivalent to the original RP-translation, and thus that the FF-translation and its refinement are $\beta\eta\varepsilon$-equivalent to the RP-translation. 
In Section~\ref{invertib-section} we discuss some limitations as well as some advantages of predicative translations for the study of identity of proofs and proof reduction. 
%RP-translation and the translations into $\Nda$ with respect to the preservation of equivalence and reduction.  In 
In Section~\ref{conc-section} we briefly summarize the results of the paper, we draw some connections with related work, and we suggest further directions of investigation. 
Finally, the rich appendix provides full proofs (in  $\lambda$-calculus notation)
of the results discussed or simply sketched in the main text.

\section{Polynomial  connectives and their RP-translation}\label{preliminaries}

%To make the present paper as self-contained as possible, definitions of these notions are given in Appendix~\ref{systemsappendix}-\ref{epsilonappendix}.

%As discussed in the introduction the translation of propositional connectives is by no means restricted to the usual connectives $\land,\lor$, but scales in a rather straightforward way to all connectives definable by harmonious introduction/elimination rules. 

In this section we introduce a formal framework for natural deduction which extends the one from \cite{StudiaLogica} to a more general class of propositional connectives.
%We then shortly recall the main notions and results from \cite{StudiaLogica}, which constitute our starting point.
%In the appendix it is shown in detail how such results scale to the more general framework here considered.

% we refer to that article for a fully detailed treatment of some technical notions here employed. 

%
% as well as other ones that are well-known from the literature, and we introduce some notational conventions.

\subsection{Polynomial connectives}\label{lansys}

As suggested  in the  previous paper (cf.~\cite{StudiaLogica} Section~4.3), the results we are concerned with are not limited to the standard intuitionistic connectives, but scale smoothly to a wider class of connectives investigated in proof-theoretic semantics  (see e.g.~\cite{Pra79, SH1984}). These  are those  connectives that can be defined by composing $\vee$, $\wedge$, $\top$ and $\bot$, such as %the %. 
% An example, that will guide our discussion throughout the text, is
the ternary connective $\bullet(A_1,A_2,A_3)$ definable as $(A_1\wedge A_2)\vee A_3$  %, whose rules and conversion are given in Table~\ref{bul}.
whose introduction and elimination rules are as follows:
%
%\begin{figure}
% \adjustbox{scale=0.8,center}{
% \centering \parbox[c][5cm][c]{1.2\textwidth}{

{\small$$\AXC{$A_1$}\AXC{$A_2$}\RL{$\bullet$I$_1$}\BIC{$\bullet(A_1,A_2,A_3)$}\DP \quad\AXC{$A_3$}\RL{$\bullet$I$_2$}\UIC{$\bullet(A_1,A_2,A_3)$}\DP \qquad\qquad \AXC{$\bullet(A_1,A_2,A_3)$}\AXC{$[A_1][A_2]$}\noLine\UIC{$C$}\AXC{$[A_3]$}\noLine\UIC{$C$}\RL{$\bullet$E}\TIC{$C$}\DP$$}

%
%A uniform and compact presentation of these connectives and their introduction/elimination rules can be obtained using the technique of \emph{finite polynomial functors} as a notation for indeces.
%
%We will be mainly concerned with three natural deduction systems based on three  different languages and some restrictions thereof. % $\mathcal{L}$, its extension with the second-order propositional universal quantier $\Ldve$, and the \{\supset,\forall.
\noindent In general, each such  connective,  %$\dagger$
 is definable in $\NI$ as $\bigvee_{p=1}^n \bigwedge_{q=1}^{m_p} A_{pq}$ for some choice of $n$, $m_p$s and $A_{pq}$s, but here it will be treated as primitive.

Borrowing ideas from the theory of \emph{finite polynomial functors} \cite{Kock}, each such connective can be described as determined by %starting from the following data:
 three finite lists $\C I,\C J, \C K$  (to be thought of as lists of indices) and two functions $f: \C J \mapsto \C I$ and $g: \C J \mapsto \C K$ %. We depict this configuration as follows:
 that we depict in a diagram as follows:\footnote{In the language of category theory this configuration describes a unary {finite polynomial functor}, which is the reason for our terminological choice.} 
$$\begin{tikzcd}
%  \C I  &  \C A \ar{l}[above]{f} \ar{r}{g}  & \C B%\ar{r}{h} & \C J 
%\end{tikzcd}$$
%A polynomial functor is finite when it is a diagram in $\mathsf{FinSet}$.
%\end{definition}
%Let $I$ be a set. A \emph{finite polynomial diagram over $I$}, noted $(f,g)$, is a diagram in $\Set$ of the form
%$\begin{tikzcd}
\C I  &  \C J \ar{l}[above]{f} \ar{r}{g}  & \C K
\end{tikzcd}$$
%where $\C I,\C J,\C K$ are finite linear orders \PP{giusto? Forse $J$ non serve}.
% A finite polynomial diagram $(f,g)$ over $I$ induces a functor $P_{f,g}: \Set^{I} \to \Set$, that we call a \emph{finite polynomial functor},
%by letting $P_{f,g}(X_{i})= \sum_{i\in cod(g)}\prod_{j\in g^{-1}(i)}X_{f(j)}$.
% 
%where, given an arrow $\begin{tikzcd}X_{i} \ar{r}{u_{i}} & Y_{i}\end{tikzcd}$, that is a $I$-indexed family of arrows $u_{i}:X_{i}\to Y_{i}$,  $P(u_{i}): P(X_{i})\to P(Y_{i})$ is given by $\sum_{i\in K}\prod_{j\in g^{-1}(i)}u_{f(j)}$. 
%
%
%
%
%
%
%
%
% \PP{questo non mi pare serva}The diagram $(f,g)$ yields a functor $ \mathsf{Set}^{\C I}\to \mathsf{Set}^{\C J}$ given by 
% $$
%  (X_{i})_{i\in\C I}\quad \mapsto \quad  \sum_{k\in{\C K}}\prod_{j\in \OV{g}(k)}X_{f(j)} 
% $$
%
%
%
%
%
Any such diagram determines  what we will call a {\em polynomial connective}  to be indicated with $\dagger^{(f,g)}$, or simply $\dagger$ when $f,g$ are clear from the context, in the following way:

\begin{itemize}
\item The length $|\C I|$ of $\C I$ measures the arity of $\dagger$, so that  when $\dagger$ is applied to an $\C I$-indexed list of formulas $\langle A_{i}\rangle_{i\in \C I}$\footnote{$\left\langle{A_i}\right\rangle_{\scalerel*{{i}\in{\C I}}{X}}$
abbreviates the sequence of formulas $A_{i_1} \ldots A_{i_n}$ when  $\C I=\langle i_1, \ldots, i_n \rangle$.
} one obtains a new formula $\dagger \langle A_{i}\rangle_{i\in \C I}$.
\item The length $|\C K|$ of $\C K$ is the number  of distinct introduction rules of $\dagger$.
  \item Any element  $k$ in $\C K$ determines a sublist of $\C J$, namely the list of all $j\in \C J$ such that $g(j)=k$, that we indicate with $g^{-1}(k)$ and whose length $|g^{-1}(k)|$ is the number of premises of the $k$-th introduction rule of $\dagger$.\footnote{\label{indfam}Thus $\C J$ can be seen as a family of lists indexed by the elements of $\C K$.}
  \end{itemize}

  Using the  functions $f$ and $g$ we can  describe the introduction and elimination rules for $\dagger$ as follows. Given an $\C I$-indexed list of formulas $\langle A_{i}\rangle_{i\in \C I}$, the  $k$-th introduction rule $\dagger$I$_{k}$ for $\dagger$, allows us to infer  $\dagger \langle A_{i}\rangle_{i\in \C I}$ from the list of premises  $\langle A_{f(j)}\rangle_{j\in g^{-1}(k)}$.
% (where ${g}^{-1}(k)$ is the inverse of $g$, i.e.~${g}^{-1}(k):\C K\mapsto \wp(\C J)$).
Given $\dagger \langle A_{i}\rangle_{i\in \C I}$ and a $\C K$-indexed list of derivations of an arbitrary formula $C$ from (respectively) the premises of the $k$-th introduction rule  $\langle A_{f(j)}\rangle_{j\in g^{-1}(k)}$, we can infer $C$ thereby discharging in the $k$-th derivation of $C$ the assumptions  $\langle A_{f(j)}\rangle_{j\in g^{-1}(k)}$. %(with the  multiple and vacuous discharge following the same convensions as in $\NI$).  %As usual in proof-theoretic semantics, from a list of introduction rules one can deduce a unique elimination $\dagger$E rule as illustrated below right:
%
%
% and  its  rules are the following:
%
%for any
%
%
%its number of introduction rules is the cardinality of $\C K$, and the number of premises of the $k$th-introduction rule  is the cardinality of ${g}^{-1}(k)$ (where ${g}^{-1}(k)$ is the inverse of $g$, i.e.~${g}^{-1}(k):\C K\mapsto \wp(\C J)$) and  its  rules are the following:
We depict the rules as follows:

\begin{lrbox}{\mypti}% Store prooftree in \mypti
  \raisebox{2ex}[-0ex][0ex]{
\begin{varwidth}{\linewidth}
  $\left\langle\AXC{$%\phantom{_{j\in\OV{g}(k)}}
      \left[\langle A_{f(j)}\rangle_{j\in{g}^{-1}(k)}\right]$}\noLine\UIC{$C$}\DP\right\rangle_{k\in\C K}$
\end{varwidth}}
\end{lrbox}

$$\centerAlignProof\left\langle\AXC{$\phantom{_{j\in{g^{-1}}(k)}}\lust{A_{f(j)}}{j}{{g^{-1}}(k)}$}\RL{$\dagger
%^{(f,g)}
$I$_k$}\UIC{$\dagger
% ^{(f,g)}
\lust {A_i}{i}{\C I}$}\DP\right\rangle_{k\in \C K} \qquad\qquad \centerAlignProof\AXC{$\dagger
%^{(f,g)}
\lost{A_i}{i}{\C I}$}\AXC{\usebox{\mypti}}\RL{$\dagger
%^{(f,g)}
$E}\BIC{$C$}\DP$$

\noindent
%where  angle brackets indicate sequences of syntactic entities, so that for instance  $\left\langle{A_i}\right\rangle_{\scalerel*{{i}\in{\C I}}{X}}$
%abbreviates the sequence of formulas $A_{i_1} \ldots A_{i_n}$ when  $\C I=\langle i_1, \ldots, i_n \rangle$.

\begin{remark}\label{ex:fpfsum}
  % For all non-zero positive integer $n$, let $\B n$ indicate the list $\langle 1,\dots, n\rangle$.  Moreover,
  Let $\B 0$ be the empty list, and $\B 1$, $\B 2$, $\B 3$, $\B 4$, $\B 5$ be the lists $\langle 1\rangle$, $\langle 1,2\rangle$, $\langle 1,2,3\rangle$, $\langle 1,2,3,4\rangle$, $\langle 1,2,3,4,5\rangle$ respectively. 
The connective  $\bullet(A_1,A_2,A_3)= (A_{1}\land A_{2})\lor A_{3}$ is shorthand for  $\dagger^{(I_{\B 3},g)}\langle A_1,A_2,A_3\rangle$,   given by $\B 3 \stackrel{I_{\B 3}}{\leftarrow} \B 3 \stackrel{g}{\to} \B 2$, where  $I_{\C A}$ indicates the identity function on the list  $\C A$ and $g:\{1,2\mapsto 1; 3\mapsto 2 \}$. Similarly, the usual connectives $\vee$, $\wedge$, $\top$, $\bot$ are obtained through the configurations 
$
\B 2 \stackrel{I_{\B 2}}{\leftarrow} \B 2 \stackrel{I_{\B 2}}{\to}\B 2
$, 
%where %$\B 2=\{0,1\}$,
%induces disjunction, and
$\B 2 \stackrel{I_{\B 2}}{\leftarrow}  \B 2 \stackrel{1}{\to} \B 1$,  $\B 0 \stackrel{\emptyset}{\leftarrow} \B 0 \stackrel{\emptyset}{\to} \B 1$ and $\B 0 \stackrel{\emptyset}{\leftarrow} \B 0 \stackrel{\emptyset}{\to} \B 0$, respectively   (where $\emptyset$ indicates the  empty function and $1$ the constant function with value $1$.). We observe that some arguments of a connective may play a ``dummy'' role or may be used more than once, such as in the connective $\blacktriangle(A_1,A_2,A_3,A_4,A_5)= (A_2\wedge A_3)\vee (A_4 \wedge A_3)$ given by $\B 5 \stackrel{f}{\leftarrow} \B 4 \stackrel{g}{\to} \B 2$, where $f:\{1\mapsto 2; 2,4\mapsto 3; 3\mapsto 4 \}$ and $g:\{1,2\mapsto 1; 3,4\mapsto 2 \}$. %, respectively.
%corresponds to the product $(X_{i})_{i\in\B 2} \mapsto X_{0}\times X_{1}$.

\end{remark}

\begin{remark}\label{gencon1}
  Treating $\vee$ and $\bot$  as polynomial connectives (see Remark \ref{ex:fpfsum}) yields their usual  introduction and elimination rules. % for $\vee,\top,\bot$ coincide with the generalized rules described above (when we treat such connectives as described in Remark \ref{ex:fpfsum}).
  This is not the case for $\land$ (and $\top$). When treated as a polynomial connective, conjunction has a unique elimination rule $\wedge$E$_{p}$ instead of the usual $\wedge$E$_1$ and $\wedge$E$_2$ (see also Remark~\ref{gencon}  below): 

{\small $$\AXC{$A\wedge B$}\AXC{$[A][B]$}\noLine\UIC{$C$}\RL{$\wedge$E$_p$}\BIC{$C$}\DP%$$%\end{equation}}
%instead of the usual elimination rules
 %
%
%{\small \begin{equation}\label{eq:elimand}
  \qquad\qquad\AXC{$A\wedge B$}\RL{$\wedge$E$_1$}\UIC{$A$}\DP \qquad \qquad \AXC{$A\wedge B$}\RL{$\wedge$E$_2$}\UIC{$B$}\DP$$}%\end{equation}}
  \end{remark}

\begin{remark}\label{index-sets}We will adopt the convention of using $i$ for indices in $\C I$, $k$ for indices in $\C K$ and  $j$ for indices in ${g}^{-1}(k)$. Hence, to enhance readability, we will often omit the indication of the index set, so that  for instance $\left\langle{A_i}\right\rangle_{\scalerel*{{i}\in{\C I}}{X}}$ will be abbreviated as  $\lost{A_i}{i}{\C I}$. 
%We will moreover abbreviate $\dagger^{(f,g)}$, to $\dagger$, thereby leaving $f$ and $g$ implicit. 
%The above rules  are therefore shortened as follows:
%
%\begin{lrbox}{\mypti}% Store prooftree in \mypti
%  \raisebox{2ex}[-0ex][0ex]{
%\begin{varwidth}{\linewidth}
%  $\left\langle\AXC{$%\phantom{_{j\in\OV{g}(k)}}
%      \left[\lost{A_{f(j)}}{j}{\OV{g}(k)}\right]$}\noLine\UIC{$C$}\DP\right\rangle%_{k\in\C K}
%  $
%\end{varwidth}}
%\end{lrbox}
%
%
%$$\centerAlignProof%\left(
%\AXC{$%\phantom{_{j\in\OV{g}(k)}}
%    \lost{A_{f(j)}}{j}{\OV{g}(k)}$}\RL{$\dagger%^{(f,g)}
%    $I$_k$}\UIC{$\dagger%^{(f,g)}
%    \lost {A_i}{i}{\C I}$}\DP%\right)_{k\in \C K}
%  \qquad\qquad \centerAlignProof\AXC{$\dagger%^{(f,g)}
%  \lost{A_i}{i}{\C I}$}\AXC{\usebox{\mypti}}\RL{$\dagger$E}\BIC{$C$}\DP$$
\end{remark}

% \noindent
%That is, .% each introduction rule has . The arity of $\dagger$ is the cardinality of 
%We call the set $\C I$ the \emph{base} of the functor.
%We will restrict attention to finite polynomial functors $\mathsf{FinSet}^{\C I}\to \mathsf{FinSet}$. As $\C J$ is a singleton we can omit the constant arrow $h$. 

%We now introduce the language and inference rules for propositional intuitionistic logic enriched with such connectives.
Given a set of propositional variables $\mathcal{V}$, indicated as $X$, $Y$, $Z$, \dots, the formulas of the language $\Ldp$  will be constructed using implication,  the universal quantifier and the family of polynomial connectives $\dagger^{f,g}$%.%of different arities%to be indicated with $\dagger$
. Besides $\Ldp$, we will mainly be concerned with two restrictions thereof, $\Lp$ and $\Ld$:
\begingroup\makeatletter\def\f@size{10}\check@mathfonts
$$
\begin{array}{rcl}
    \Ldp &::=& X \ | \ A\supset B\ |\ \dagger^{(f,g)}\langle A_i\rangle_{i\in \C I} \ | \ %A\wedge B \ | \ \bot \ |\
             \forall X. A\\
\Lp &::=& X \ | \ A\supset B\ | \ \dagger^{(f,g)}\langle A_i\rangle_{i\in \C I} %\ | \ A\wedge B \ | \ \bot \ %|\ %\forall X. A
\\
\Ld &::=& X \ | \ A\supset B\ | \ %A\vee B \ | \ A\wedge B \ | \ \bot \ |\
\forall X. A\\
\end{array}
$$
\endgroup
We moreover indicate with $\Lv$ (respectively $\Lb$) the restriction of $\Lp$ in which $\vee$ (resp.~$\bullet$) is the only connective besides $\imp$, and similarly for $\Ldv$ and  $\Ldb$.

The natural deduction system $\Ndp$ over the language $\Ldp$ is obtained by adding to the rules of the standard system $\Nd$ for $\Ld$ (recalled in Definition 1.1 and Table~1 on pages 197--198 of \cite{StudiaLogica}),
 all introduction and elimination rules for the connectives $\dagger^{(f,g)}$. The system $\Ndp$, along with its equational theory, is described in detail in $\lambda$-notation in Appendix~\ref{systemsappendix}. Finally, by restricting $\Ndp$ to the languages $\Lp, \Lb$, \ldots we obtain the subsystems $\Np, \Nb$, \ldots
%
%
%
%
% by replacing the rules for $\vee$ in $\Ndv$ with the above rules for $\bullet$ (resp.~$\dagger^{(f,g)}$), while their  equational theories are described in $\lambda$-notation in Appendix~\ref{systemsappendix}.% are summarized in Appendix~\ref{systemsappendix} and \ref{betaetagamma}. 
%

\begin{remark}\label{gencon}
Observe that the  standard intuitionistic natural deduction system  $\NI$ is \emph{not} a fragment of $\Np$ since the elimination rule for conjunction in the latter system is in general form (see Remark~\ref{gencon1}). Although replacing  $\wedge$E$_1$ and $\wedge$E$_2$  with $\wedge$E$_{p}$  does not alter derivability, the two forms of elimination rules behave differently with respect to identity of proofs.
  More on this below in Section~\ref{invertib-section}.
 \end{remark}

% The rules of $\Nb$ and $\Ndb$ (respectively $\Np$ and $\Ndp$) are obtained by replacing those for $\vee$ with the above rules for $\bullet$ (resp.~$\dagger^{(f,g)}$), while their  equational theories are described in $\lambda$-notation in Appendix~\ref{systemsappendix}.% are summarized in Appendix~\ref{systemsappendix} and \ref{betaetagamma}. 

%\begin{remark} 
%In order to simplify the reading of the main text, we will mostly discuss results about the propositional fragment $\Nb$ which contains the unique generalized connective $\bullet(A, B,C)$. 
%Unless otherwise stated, all such results are proved in detail in appendix for the full propositional system $\Np$.
%\end{remark}

The system
$\Nda$ (as well as $\Ndpa$, $\Ndba$, $\Ndva$)
is obtained by replacing the rule $\forall$E with the following:
\begingroup\makeatletter\def\f@size{10}\check@mathfonts
$$\AXC{$\forall X.A $}
\RL{$\forall\text{E}_{at}$}
\UIC{$A\llbracket Y/X\rrbracket$}
\DP
$$
\endgroup
in which the witness must be an atomic formula.

We indicate  derivability in $\Np, \Nd, \ldots$ with $\vdash_{\Np}, \vdash_{\Nd}, \ldots$

% \begin{remark}\label{conbot}As in \cite{StudiaLogica}, we do not consider $\wedge$ and $\bot$, but unless otherwise specified the results presented in the paper apply as well if one takes $\mathcal{L}$ and $\Ldv$ to be the full intuitionistic propositional language and its second-order extension.
% \end{remark}

In the following we will often indicate an arbitrary formula $A\in\Ld$ as:

\begingroup\makeatletter\def\f@size{10}\check@mathfonts
\begin{equation*}\tag{$*$}\label{Fformula}\forall \langle Y_{1}\rangle(F_{1}\supset \forall \langle Y_{2}\rangle(F_{2}\supset \dots \supset \forall \langle Y_{n}\rangle(F_{n}\supset \forall \langle Y_{n+1}\rangle(X))\dots ))
\end{equation*}
\endgroup
where $\forall\langle Y_{i}\rangle$ (for $1\leq i\leq n+1$) indicates a list of consecutive quantifications.
Moreover, we let $\at{A}=X$ indicate the rightmost atom of $A$. Finally, we will often abbreviate $A_1\imp (\ldots \imp (A_{n-1}\imp A_n)\ldots)$ with $A_1\imp \ldots\imp A_{n-1}\imp  A_n$.

\subsection{Polynomial formulas and the RP-fragment of $\Nd$}\label{rp-trans}

\begin{table}
\rule{\textwidth}{.5pt}
\begin{align}
(A\land B)^{*} & =\forall X. (A^{*}\supset B^{*}\supset X)\supset X\\
(A\lor B)^{*}& =\forall X. (A^{*}\supset X)\supset (B^{*}\supset X)\supset X\\
\bot^{*}& = \forall X.X\\ 
\top^{*}&= \forall X.X\supset X
\end{align}
\rule{\textwidth}{.5pt}
\caption{RP-translation of standard connectives.}
\label{tabrp}
\end{table}

The RP-translation of standard connectives $\land,\lor, \bot, \top$ is recalled in Table \ref{tabrp}.
Although the connective $\bullet$ can be translated by composing the translations of $\land$ and $\lor$, 
a natural and  more economical way to encode $\bullet$ in $\Nd$ is given by

$$\bullet(A, B,C)^*= \forall X. (A^*\imp (B^*\supset X)) \supset (C^*\supset X) \supset X$$

The universal formula above shares a common structure with those in Table \ref{tabrp}: all such formulas are of the form $\forall X.A_{1}\supset \dots \supset A_{n}\supset X$, where the $A_{i}$ have a unique occurrence of $X$ in rightmost position. This suggests the following definition:

% Note that $\bullet(A, B,C)^*$  is of the form $\forall X. F_1\supset (F_2 \supset X))$ with $F_1$ and $F_2$ sp-$X$. 
%In general,   $(\dagger\lost{A_i}{}{})^*$  is of the form $\forall X. F_1\supset (F_2 \supset \ldots (F_n \supset X)\ldots)$, where  $F_1,\ldots, F_n$ are sp-$X$. 

\begin{definition}
  A formula $A \in \Ld$ is \emph{strongly positive in $X$} (short \emph{sp-$X$}) 
 when it is of the form 
  $A_{1}\supset \dots \supset A_{n}\supset X$ for some $n\in \mathbb N$ and $X$ does not occur in any of the  $A_{i}$.
  
  A formula $A\in \Ld$ is \emph{polynomial in $X$} when it is of the form 
  $A_{1}\supset \dots \supset A_{n}\supset X$ for some $n\in \mathbb N$ and all $A_{i}$ are sp-$X$.

  A formula $\forall X.A$ is called \emph{universal polynomial} if $A$ is polynomial in $X$.
  
%  
%when all  free occurrences of the variable $X$ in $F$ 
%are in strictly positive position (i.e.~do not appear to the left of any occurrence of $\supset$).
\end{definition}

\begin{remark}
We introduce the following compact notation for $\Nd$-formulas. Given a finite list $\C A=\langle a_{1},\dots, a_{k}\rangle$, an $\C A$-indexed list of formulas $A_{a_{1}},\dots, A_{a_{k}}$ and a formula $B$, we let 
\begin{center}
$\langle A_{a}\rangle_{a\in \C A}  \supset B$ 
\end{center}
\noindent be shorthand for the formula
$A_{a_{1}}\supset \dots \supset A_{a_{k}}\supset B$ if $\C A$ is non-empty, otherwise for $B$. When $\C A$ is clear from the context the index ${a\in \C A}$ will be omitted so we simply  write $\langle A_{a}\rangle \supset B$.
\end{remark}

A universal  polynomial formula $A$ can thus be written as 
$\forall X.\langle B_{b}\rangle_{b\in \C B}\supset X$, where $\C B$ is some list and the $B_{b}$s are sp-$X$. Moreover, any sp-$X$ formula $B$ can in turn  be written as $\langle A_{a}\rangle_{a\in \C A}\supset X$, for some list $\C A$ such that each  $A_{a}$ has no occurrence of $X$. Hence, a universal polynomial formula can be written as $\forall X.\langle \langle A_{a}\rangle_{a\in \C A_b}\supset X\rangle_{b\in \C B}\supset X$ for some list $\C B$ and family of lists $\C A_b$ indexed by $\C B$. Since a diagram $\C I \stackrel{f}{\leftarrow} \C J\stackrel{g}{\to} \C K$ describes a family of lists indexed by the elements of a list (see footnote~\ref{indfam} above), it can be used to associate to each  $\C I$-indexed family of formulas $\langle A_i\rangle_{i\in \C I}$ the universal polynomial formula $\forall X. \EXP{X}{ k\in \C K   }{ \EXP{X}{j\in {g}^{-1}(k)  }{A_{f(j)}^{\phantom{fj}*}}}$, which in turn we propose to take as the RP-translation of the propositional formula $\dagger^{f,g}\langle A_i\rangle_{i\in \C I}$.

The Russell-Prawitz translation can thus be generalized to the whole of  $\Ldp$ 
by translating the formulas of $\Ldp$ whose outermost connective is polynomial with a universal polynomial formula of $\Ld$:

\begin{definition}[RP-translation of formulas]\label{def:rp1}
We define a translation $^{*}$ from formulas of $\C L^{2p}$ to formulas of $\C L^{2}$ as follows:
$$
X^{*}= X \qquad  (A\supset B)^{*}=A^{*}\supset B^{*} \qquad  (\forall X.A )^{*}=\forall X.A ^{*}
$$
$$\left(\dagger\lust{A_i}{i}{\C I}\right)^*=  \forall X. \EXP{X}{ k\in \C K   }{ \EXP{X}{j\in {g}^{-1}(k)  }{A_{f(j)}^{\phantom{fj}*}}} \quad \text{(for $X$ not free in any $A_i$)}
%U_{1}\to\dots \to U_{card(a)}\to X 
$$

\noindent
where $\dagger$ is determined by $\C I \stackrel{f}{\leftarrow} \C J\stackrel{g}{\to} \C K$
\end{definition}

\begin{remark}
  For readability, we will  abbreviate $(A\supset X)\supset ((B\supset X)\supset X)$ as $A\curlyvee B$ (and thus   $(A\vee B)^*$  as  $\forall X. (A^*\curlyvee B^*)$), and similarly $(A\imp (B\supset X) \supset (C\supset X) \supset X$ as $\curc(A,B,C)$ (and thus $\bullet(A, B,C)^*$ as $\forall X.\curc(A^*,B^*,C^*)$).
\end{remark}

The RP-translation scales well from formulas to derivations, yielding an embedding $^{*}:\Ndp \mapsto \Nd$. The embedding for usual connectives is recalled  in
(\cite{StudiaLogica} Def.~2.2). Its extension to $\Np$ 
 is  defined in detail in Appendix~\ref{rp-appendix} in $\lambda$-calculus notation.

\begin{remark}
If $\D D$ is an $\Ndb$-derivation of $A$ from undischarged assumptions $A_1,\ldots, A_n$, then $\D D^*$ is an $\Nd$-derivation of $A^*$ from undischarged assumptions $A_1^*, \ldots, A_n^*$. (For a proof-sketch for the whole $\Ndp$, see Appendix~\ref{rp-appendix}).
\end{remark}

If we restrict the language of $\Nd$ by requiring every universal formula to be polynomial, then we obtain a fragment of $\Nd$ that we will call the \emph{Russell-Prawitz fragment}.

\begin{definition}[Russell-Prawitz fragment]
We let  $\Ldrp$ be the subset of $\Ld$ in which all universal formulas are polynomial, and 
 we let  $\Ndrp$, called the \emph{Russell-Prawitz fragment of $\Nd$}, be the fragment of $\Nd$ obtained by restricting to $\Ldrp$-formulas.
The system $\Ndrpc$ is the subsystem of $\Ndrp$ in which universal formulas $\forall X.A $ are all of the form $\forall X.\curc(A_{1},A_{2},A_{3})$.

\end{definition}

\begin{remark}
  The RP-fragment of $\Ld$ can be equivalently defined inductively as follows:
  $$    \Ldrp \ ::= \ X \ | \ A\supset B\ |\ \ \forall X. \EXP{X}{ k\in \C K   }{ \EXP{X}{j\in {g}^{-1}(k)  }{A_{f(j)}}} %\ | \ %A\wedge B \ | \ \bot \ |\
%             \forall X. A\\
             $$
             \noindent where in the last clause we assume $\lost{A_i}{i}{\C I}$ to be an $I$-indexed family of  $\Ldrp$-formulas and $\C I \stackrel{f}{\leftarrow}\C J \stackrel{g}{\to}\C K$ to be the diagram determining a polynomial connective. A quantified formula of $\Ldrp$ may therefore have  quantified formulas as proper subformulas, provided they are in turn universal polynomial formulas.
\end{remark}

\begin{remark}
It is clear that the restriction of the RP-translation to $\C L^{p}$ is into $\Ldrp$, and that if $\D D$ is an $\Nb$-derivation then  $\D D^*$ is an $\Ndrpc$-derivation. 

\end{remark}

\begin{remark}
 It is easy to check that the equational theory of $\Ndrp$ is well-defined, since $\Ldrp$ is closed under substitution. 
\end{remark}

%
%For the formulation of some results below it is useful to introduce another restriction on the second-order systems $\Nd$:
%\begin{definition}
%We say that a formula is \emph{nested sp-$X$} iff it is of the form $F_1\supset (F_2 \supset \ldots (F_k \supset X)\ldots)$, where  $F_1,\ldots, F_k$ are sp-$X$.
%We call   $\Ldrp$ the fragment of $\Ld$ in which all formulas of the form $\forall X.A$ are such that $A$ is nested sp-$X$, and 
% we let  $\Ndrp$ be the restriction of $\Nd$ to $\Ldrp$-formulas.
% It is easy to check that the equational theory of $\Ndrp$ is well-defined, since $\Ldrp$ is closed under substitution. 
%The system $\Ndrpc$ is the subsystem of $\Ndrp$ in which universal formulas $\forall X.A $ are all of the form $\forall X\curc(A_{1},A_{2},A_{3})$.
%
%\end{definition}
%

\begin{remark}
  The system $\Ndrp$ is the smallest fragment of $\Nd$ closed under $\supset$ and containing  the RP-translation of all polynomial connectives.

  Observe that  not only the translation of a formula $\dag^{f,g}\langle A_{i}\rangle$ is a universal polynomial formula, but for every universal polynomial formula  $A$ there is at least (but in general more than) one configuration $\C I \stackrel{f}{\leftarrow} \C J \stackrel{g}{\to} \C K$ such that $A=(\dag^{f,g}\langle A_{i}\rangle)^*$. To wit, let  $A= \forall X.A_{1}\supset \dots \supset A_{n}\supset X$, where $A_{i}=A_{i1}\supset \dots \supset A_{ik_{i}}\supset X$. It is enough to take  %$\C I \stackrel{I}{\leftarrow} \C J \stackrel{g}{\to} \C K$ where 
 $\C I=\C J= \langle \langle 1, 1\rangle, \dots, \langle 1,k_{1}\rangle, \dots, \langle n,1\rangle, \dots, \langle n, k_{n}\rangle\rangle$ (so that $A_{\langle i,j\rangle}=A_{ij}$),
 $\C K=\langle 1,\dots,n\rangle$, $f$ to be the identity function and $g=\pi_{1}^{2}:  \langle i, j\rangle\mapsto i$.%, so that $A$ encodes $\dag^{f,g}(A_{i})$.

  %The reason for calling such formulas polynomial is that these coincide with the translation of polynomial connectives. In fact, given any configuration $\C I \stackrel{f}{\leftarrow} \C J \stackrel{g}{\to} \C K$ and a $\C I$-indexed family of formulas $A_{i}$, the polynomial formula 
%$$
%\forall X. \langle    \langle   A_{f(j)}\rangle_{j\in g^{-1}(k)}\supset X\rangle_{k\in \C K}\supset X
%$$
%translates the formula $\dag^{f,g}\langle A_{i}\rangle$. Conversely, 
%any polynomial formula $A= \forall X.A_{1}\supset \dots \supset A_{n}\supset X$, where $A_{i}=A_{i1}\supset \dots \supset A_{ik_{i}}\supset X$ we can find (in a non-unique way) a configuration $\C I \stackrel{I}{\leftarrow} \C J \stackrel{g}{\to} \C K$ where 
% $\C I=\C J= \langle \langle 1, 1\rangle, \dots, \langle 1,k_{1}\rangle, \dots, \langle n,1\rangle, \dots, \langle n, k_{n}\rangle\rangle$ (so that $A_{\langle i,j\rangle}=A_{ij}$),
% $\C K=\langle 1,\dots,n\rangle$, $I$ is the identity function and $g=\pi_{1}^{2}:  \langle i, j\rangle\mapsto i$, so that $A$ encodes $\dag^{f,g}(A_{i})$.

 %Hence polynomial formulas allow for a simple encoding of generalized propositional connectives:
%This discussion should have convinced the reader that the polynomial formulas have a very close connection with the generalized connectives introduced in the previous subsection.  
%In fact, polynomial formulas provide a natural extension of the RP-translation to such connectives, as shown below:

 Given this %Moreover, by what we observed above,
 it is easily seen by induction that for any $A\in\Ldrp$ there is at least (but in general more than) one $B\in\C L^{p}$ such that $B^{*}=A$. Note however that not every $\Ndrp$-derivation is the image of some $\Np$-derivation under the RP-translation.
\end{remark}
%{\color{magenta}The RP-translation of derivations is not presented! Should we leave all to appendix or shall we put at least an example of how to translate $\dagger$-rules? }

%\begin{remark} All results claimed to hold for $\Ndrpc$ will be proved in appendix for the whole $\Ndrp$.
%\end{remark}

\begin{remark}\label{rem:spX}
In \cite{StudiaLogica} we alluded to the fact that the translation of polynomial connectives can be described through a class of formulas called \emph{nested sp-$X$}. There we used ``sp-$X$'' as short for ``strictly positive'' rather than ``strongly positive'', where strictly positive formulas are a slight generalization of strongly positive formulas in which $X$ may not occur at all, and nested sp-$X$ formulas roughly stand to universal polynomial as strictly positive formulas stand to strongly positive formulas. 

The class of polynomial formulas we consider here is thus a proper subset of the class of \emph{nested sp-$X$} formulas and it is already sufficient for the goal of RP-translating polynomial connectives. However, most of the results we prove for $\Ndrp$ can be extended to a similar (and slightly larger) system defined using nested sp-$X$ (more generally these systems are fragments of a more general system $\mathsf{\Lambda2}^{\kappa\leq 2}$ we are currently investigating, see \cite{PTCSL}).

\end{remark}

%
%
%
% $B_{1}\supset\dots \supset B_{k}\supset \forall X.A'$, where the $B_{i}$ do not contain $X$ and 
%  $A'= A_{1}\supset \dots \supset A_{p}\supset X$, where $A_{u}=A_{u1}\supset \dots \supset A_{up_{u}}\supset X$, so we can let  $\C I=\C J= \langle \langle 1, 1\rangle, \dots, \langle 1,p_{1}\rangle, \dots, \langle u,1\rangle, \dots, \langle u, p_{u}\rangle\rangle$, $\C K=\langle 1,\dots, u\rangle$, $f$ be the identity function and $g( \langle v, w\rangle)=v$.
%  Observe that this construction is by no means unique.
%  \PP{A small issue here}
%  
%%  =(\dagger(A_i))^*$ for some $\dagger^{(f,g)}$.
%\end{remark}

%{\color{magenta} I am not sure spX like this works for $\bullet$. Do we need this?}

\section{The $\varepsilon$-conversions}\label{sec:epsilon}
In this section we introduce some notational conventions and shortly recall some notions and results from the previous paper. These are based on the introduction of a notation for functors in natural deduction and of a class of conversions, called the $\varepsilon$-conversions, that express a naturality condition for natural deduction derivations. 
%While we restrict here to standard connectives, all definitions and results are proved in appendix in full generality.

\subsection{Weak expansion}\label{weakexp-sec}

%We shortly recall some technical notions from \cite{StudiaLogica} which allow to translate the concepts of functor and natural transformation into the language of natural deduction.

\begin{definition}\label{weakexp}
Given a second formula $A$ (whose structure can be describe as in  (\ref{Fformula}),  see page~\pageref{Fformula} above), the {\em  weak expansion  of $A$} is the derivation that one obtains  by repeatedly applying $\eta$-expansions to the derivation consisting only of the assumption of $A$ until  the minimal formula of the main branch is atomic, i.e.:% can be  schematically represented as follows:%, in particular it is the rightmost atom occurring in $A$
% .This derivation, which we call the weakly expanded normal form of $\D D$

\begingroup\makeatletter\def\f@size{9}\check@mathfonts
$$\AXC{$A$}\doubleLine\RightLabel{$\forall \text{E}$}\UIC{$F_{1}\supset \forall \langle Y_{2}\rangle(F_{2}\supset \dots \supset \forall \langle Y_{n}\rangle(F_{n}\supset\forall \langle Y_{n+1}\rangle ( X))\dots )$}
\AxiomC{$\stackrel{m_1}{F_{1}}$}
\RightLabel{${\supset}\text{E}$}
\BinaryInfC{$\forall \langle Y_{2}\rangle(F_{2}\supset \dots \supset \forall \langle Y_{n}\rangle(F_{n}\supset\forall \langle Y_{n+1}\rangle (X))\dots )$}
\noLine
\UIC{$\ddots$}
\noLine
\UIC{$\ $}
\noLine
\UIC{$\forall \langle Y_{n}\rangle(F_{n}\supset \forall \langle Y_{n+1} \rangle(X))$}\doubleLine\RightLabel{$\forall\text{E}$}\UIC{$F_{n}\supset \forall \langle Y_{n+1} \rangle(X)$}
\AXC{$\stackrel{m_n}{F_{n}}$}
\RightLabel{${\supset}\text{E}$}
\BIC{$\forall \langle Y_{n+1}\rangle (X)$}
\doubleLine
\RightLabel{${\forall}\text{E}$}
\UIC{$X$}
\RightLabel{${\forall}\text{I}$}
\doubleLine
\UIC{$\forall \langle Y_{n+1}\rangle (X)$}
\RightLabel{${\supset}\text{I}$ $(m_n)$}\UIC{$F_n\supset \forall \langle Y_{n+1} \rangle(X)$}\doubleLine\RightLabel{$\forall\text{I}$}\UIC{$\forall \langle Y_n\rangle (F_n\supset\forall \langle Y_{n+1}\rangle ( X))$}\noLine\UIC{$\vdots$}\RightLabel{${\supset}\text{I}$ $(m_1)$}\UIC{$F_{1}\supset \forall \langle Y_{2}\rangle(F_{2}\supset \dots \supset \forall \langle Y_{n}\rangle(F_{n}\supset \forall \langle Y_{n+1} \rangle(X))\dots )$}\doubleLine\RightLabel{$\forall\text{I}$ }\UIC{$A$}
\DisplayProof$$
\endgroup
where {\small$\AXC{$\forall \langle Y_i\rangle G$}\doubleLine\RightLabel{$\forall \text{E}$}\UIC{$G$}\DP$} and {\small$\AXC{$G$}\doubleLine\RightLabel{$\forall \text{I}$}\UIC{$\forall \langle Y_i \rangle G$}\DP$} indicate (possibly empty) sequences of applications of $\forall\text{E}$ and $\forall\text{I}$.

\end{definition}

We will indicate by $\rr{\TT{El}}%\D E
^{\vec m}_A
$ 
the ``first half'' of the derivation above, consisting of a chain of elimination rules depending on indices
%
% the two halves of such a derivation consisting of elimination rules and introduction rules respectively, so we will abbreviate the above derivation in the following way:
%\begingroup\makeatletter\def\f@size{10}\check@mathfonts
%$$
%\AXC{$A$}
%%\AXC{$\stackrel{m_1}{F_1}\ \ldots\ \stackrel{m_n}{F_n}$}
%\noLine\UIC{$\D E_A^{\vec m}$}\noLine\UIC{$X$}\noLine\UIC{$\phantom{(\vec m)}\ \D I_A \ (\vec m)$}\noLine\UIC{$A$}\DP 
%%\quad\qquad\qquad\AXC{$A$}
%%\noLine\UIC{$\D E_A$}\noLine\UIC{$X$}\noLine\UIC{$%\phantom{(m_1,\ldots,m_n)}\
%%  \D I_A %\ (m_1,\ldots,m_n)
%%  $}\noLine\UIC{$A$}\DP 
%$$
%\endgroup
 $\vec m=m_{1},\dots, m_{n}$ for the undischarged assumptions $\stackrel{m_{1}}{F_{1}},\dots, \stackrel{m_{n}}{F_{n}}$.
 Similarly, we  will indicate by $%\D I_{A}
 \rr{\TT{In}}(\vec m)$ 
 the ``second half'' of the derivation above, consisting of a chain of introduction rules discharging the assumptions with indices $\vec m$ (a more rigorous description of this construction is provided in Appendix \ref{cexp-appendix} through the notion of \emph{expansion pair}).

% $\vec m=m_{1},\dots, m_{n}$ in $\mathscr E^{\vec m}_{A}$ indicates that the conclusion $X$ depends, in addition to the assumption $A$, on the undischarged assumptions $\stackrel{m_{1}}{F_{1}},\dots, \stackrel{m_{n}}{F_{n}}$, and 
%$(\vec m)$ indicates that in $\mathscr I_{A}(\vec m)$ those assumptions are discharged (a more rigorous description of this construction is provided in Appendix \ref{cexp-appendix} through the notion of \emph{expansion pair}).

\begin{remark}The weak expansion  of $A$ might differ from the expanded normal form (also known as $\eta$-long normal form)  of the derivation consisting only of the assumption of $A$, since in the latter not only the minimal formula in the main branch, but the minimal formulas of all branches are atomic, see \citet[\S II.3.2.2.]{Prawitz1971}.
\end{remark}

%\begin{remark}
%Observe that the minimal formula of the weak expansion of $A$ is the atomic formula occurring in rightmost position in $A$.
%\end{remark}

\subsection{$C$-expansion}
%When $C$ is sp-$X$ and $\D D$ is a derivation of $B$ from  assumptions $A,\Delta$ (where $A$ is freely chosen among  the undischarged assumptions of $\D D$), the {\em $C$-expansion of a derivation $\D D$ (relative to $X$ on main assumption $A$)}
%is a derivation of $C\llbracket B/X\rrbracket$ from $C\llbracket A/X\rrbracket, \Delta$ defined as follows:

\begin{definition}\label{cexp}
  If $C=C_1\imp\ldots\imp C_n\imp X$ is sp-$X$ and $\mathscr{D}$ is a derivation of $B$ from undischarged assumptions $A,\Delta$ (where $A$ is freely chosen among  the undischarged assumptions of $\D D$), the {\em $C$-expansion of $\mathscr{D}$ relative to $X$ on main assumption $A$}, notation $\cexpD{X}{C^{\scaleto{X}{4pt}}}{\AXC{$A$}\noLine\UIC{$\mathscr{D}$}\noLine\UIC{$B$}\DP}$, is
  the derivation  of $C\ldbrack B/X\rdbrack$ from $C\ldbrack A/X\rdbrack, \Delta$ defined by induction on $n$ as follows: 
\begin{itemize}
\item If $n=0$, then $C=X$ and $\cexpD{X}{C^{\scaleto{X}{4pt}}}{\AXC{$A$}\noLine\UIC{$\mathscr{D}$}\noLine\UIC{$B$}\DP}$ is just $\D D$.
\item If $n\geq 1$ then $C=C_1\imp D$ where $X\not\in FV(C_1)$ and $D$ is sp-X. We define:

\begingroup\makeatletter\def\f@size{10}\check@mathfonts
\begin{lrbox}{\mypti}% Store prooftree in \mypti
\begin{varwidth}{\linewidth}
$\cexpD{X}{D^{\scaleto{X}{4pt}}}{\AXC{$A$}\noLine\UIC{$\mathscr{D}$}\noLine\UIC{$B$}\DP}$
\end{varwidth}
\end{lrbox}
$$\cexpD{X}{C^{\scaleto{X}{4pt}}}{\AXC{$A$}\noLine\UIC{$\mathscr{D}$}\noLine\UIC{$B$}\DP} = \def\defaultHypSeparation{\hskip .45em}\AxiomC{$C_1\imp D\llbracket A/X\rrbracket$}
\AxiomC{$\stackrel{n}{C_1}$}
\RightLabel{\footnotesize$\supset$E}
\BinaryInfC{\usebox{\mypti}}
\RightLabel{\footnotesize$\supset$I $(n)$}
\UIC{$C_1\imp D\llbracket B/X\rrbracket$}
\DisplayProof$$
%
% $$\left (\text{resp.~} \SF{X}{C}{\mathscr{D}} \equiv \def\defaultHypSeparation{\hskip .45em}\AxiomC{$\begin{matrix} \ \\ \ \\  F\supset\SF{X}{G}{B}\end{matrix}$}
% \AxiomC{$\stackrel{n}{C_1}$}
% \RightLabel{\footnotesize$\supset$E}
% \BinaryInfC{\usebox{\mypti}}
% \RightLabel{\footnotesize$\supset$I $(n)$}
% \UIC{$F\supset \SF{X}{G}{A}$}
% \DisplayProof\equiv
% \AxiomC{$\begin{matrix} \ \\ \ \\  \SF{X}{F\supset G}{B}\end{matrix}$}
% \AxiomC{$\stackrel{n}{\usebox{\mypto}}$}
% \RightLabel{\footnotesize$\supset$E}
% \BinaryInfC{\usebox{\mypti}}
% \RightLabel{\footnotesize$\supset$I $(n)$}
% \UIC{$\SF{X}{F\supset G}{A}$}
% \DisplayProof\right )$$
 \endgroup	
 \end{itemize}
 \end{definition}

%  $\cexpD{X}{(A_1\imp C')^X}{\AXC{$A$}\noLine\UIC{$\mathscr{D}$}\noLine\UIC{$B$}\DP}$

\begin{remark}\label{strongpx}
%We will make essential use of the notion of $C$-expansion of a derivation $\D D$ defined for sp-$X$ formulas $C$ \cite[see][sec.~3.2]{StudiaLogica}. 

%Actually, in \cite{StudiaLogica} we defined the notion of $C$-expansion for the more general class of  formulas, called {\em pn-$X$} formulas, but these will play no role in the present paper.

  In \citep[sec.~3.2]{StudiaLogica} we defined the notion of $C$-expansion for a broader class of formulas called pn-$X$ formulas. The restriction of the definition to the class of sp-$X$ formulas allows a straightforward reformulation of the definition using the notion of weak expansion that we  give in Appendix~\ref{cexp-appendix}.
\end{remark}

\begin{remark}
  In the functorial semantics of $\Nd$, what we call the $C$-expansion of a derivation $\D D$ is  just the result of applying the functor interpreting $C$ to the morphism interpreting $\D D$. As any $\Ld$-formula can be interpreted as a functor,   the notion of $C$-expansion can be extended to any $C\in \Ld$.% (but this will not be necessary for our discussion).
\end{remark}

%
%
%, that  we indicate with 
%$\cexpD{X}{C}{\AXC{$A$}\noLine\UIC{$\mathscr{D}$}\noLine\UIC{$B$}\DP}
%$
%and sometimes 
%
%
%$\cexp{X}{C}{\D D}$ thereby leaving both $X$ and $\Delta$ implicit,
%.  %For the exact definition of $C$-expansion see Appendix~\ref{cexp-appendix}.
\begin{remark}
  Whenever $X$ is clear from the context, with the notation for substitution introduced in  \cite[sec. 3.2]{StudiaLogica},  we write $C\llbracket B/X\rrbracket$ and  $C\llbracket A/X\rrbracket$  as $\SF{X}{C}{B}$ and $\SF{X}{C}{A}$, and leaving the main assumption $A$ implicit  we indicate the $C$-expansion of $\D D$ as %in either of the ways below:
$\SF{X}{C}{\mathscr{D}}%\qquad \qquad \cexpD{X}{C}{\AXC{$A$}\noLine\UIC{$\mathscr{D}$}\noLine\UIC{$B$}\DP}
$.
\end{remark}
%
%, so that:
%
%\begingroup\makeatletter\def\f@size{10}\check@mathfonts
%$$\SF{X}{C}{\mathscr{D}} = { \def\extraVskip{0pt}\AXC{$C\llbracket A/X\rrbracket$}\noLine\UIC{$\SF{X}{C}{\makebox[2em][c]{$\mathscr{D}$}}$}\noLine\UIC{$C\llbracket B/X\rrbracket$}\DP} = { \def\extraVskip{0pt}\AXC{$\SF{X}{C}{A}$}\noLine\UIC{$\SF{X}{C}{\mathscr{D}}$}\noLine\UIC{$\SF{X}{C}{B}$}\DP
%}$$
%\endgroup
%which we sometimes shorten further to:
%\begingroup\makeatletter\def\f@size{10}\check@mathfonts
%$$
%\cexpD{X}{C}{\AXC{$A$}\noLine\UIC{$\mathscr{D}$}\noLine\UIC{$B$}\DP}
%$$
%\endgroup

\begin{remark}\label{cexpweakexp}In the present paper, we will only be concerned with the $C$-expansion of derivations  of the form $\rr{\TT{El}}^{\vec m}_A$.
Observe that, by the definition of C-expansion, the undischarged assumptions of $\SF{X}{C}{{\rr{\TT{El}}}^{\vec m}_A}$ besides $\SF{X}{C}{A}$ are the same as the undischarged assumptions of $\rr{\TT{El}}^{\vec m}_A$ besides $A$, namely (see Definition~\ref{weakexp} above) $\stackrel{m_1}{F_1},\ldots, \stackrel{m_n}{F_n}$. 
\end{remark}

\subsection{The $\varepsilon$-conversions}

As recalled in the introduction, it is common to characterize identity of proofs using an equivalence induced by (the symmetric closure of) some reduction relation over derivations. 
The equivalence $\simeq_{\beta\eta}$ generated by the $\beta$- and $\eta$-conversions for $\Nd$ is standard (and recalled in Appendix \ref{betaetagamma}). Propositional connectives like $\vee$ require, in addition to $\beta$- and $\eta$-conversion rules, the so-called permutative conversions (which allow to permute an elimination rule upwards a $\vee$-elimination rule), that we call here $\gamma$-conversions, generating an equivalence relation $\simeq_{\beta\eta\gamma}$. 
Analogous conversions can be defined for all polynomial connectives, see Appendix \ref{betaetagamma}.

\begin{remark}\label{rem:gammapiu}
The equivalence relation $\simeq_{\beta\eta\gamma}$ can be further extended in two equivalent ways: either by replacing $\gamma$ by a stronger permutation $\gamma^{+}$ which allows the permutation of an arbitrary derivation upwards across an application of  $\dagger$E, or by replacing $\eta$ by a stronger $\eta^{+}$  which expresses in categorical terms the universality of the connective ($\eta^{+}$ in fact subsumes both $\gamma$ and $\gamma^{+}$). %It is the extended equivalence induced by  these extended  conversions that \PP{characterizes the identities found in a bi-cartesian closed category}.
\end{remark}

The equivalence-preservation problem of the RP-translation can be formulated as the failure of the implication below (see \cite{StudiaLogica}):
\begin{center}
$\D D_{1} \simeq_{\beta\eta\gamma}\D D_{2}  \qquad \Rightarrow \qquad 
\D D^{*}_{1} \simeq_{\beta\eta} \D D_{2}^{*}$
\end{center}

In \cite{StudiaLogica} (see Section 4.1) we showed that the implication above does hold when the equivalence we consider for $\Nd$ is the one induced by adding to $\beta$- and $\eta$-conversions a new class of conversions, called $\varepsilon$-conversions (for the case of $\circ$ these are  shown in Table \ref{epsi-conv}, for arbitrary universal polynomial formulas see Appendix~\ref{epsilonappendix}). 
%
%
%A solution  was obtained in categorial investigations of $\Nd$ by considering   equivalence relations on derivations stronger than the one induced by  $\beta$- and $\eta$- conversions. In \cite{StudiaLogica} we introduced  an  equation called $\varepsilon$ (see Section 4.1 of \cite{StudiaLogica}) expressing  in categorical terms a naturality condition for $\Nd$-derivations (this equation is an instance of a more general  equation investigated in e.g.~\cite{Bainbridge1990} and \cite{Girard1992} expressing a {\em di}naturality condition).
%The equivalence relation $\simeq_{\varepsilon}$ induced by $\beta$-, $\eta$-, and $\varepsilon$-equations strictly extends the one induced by  $\beta$- and $\eta$-equations only and  its adoption offers a solution to the equivalence-preservation problem in the following sense:
\begin{proposition}[\cite{StudiaLogica}]\label{eps-preservation}
For all $\Nv$-derivations $\D D_{1}$ and $\D D_{2}$, if $\D D_{1}\simeq_{\beta\eta\gamma} \D D_{2}$, then $\D D_{1}^{*} \simeq_{\beta\eta\varepsilon} \D D^{*}_{2}$.
\end{proposition}

Semantically, the $\varepsilon$-conversions express a naturality condition for 
$\Nd$-derivations. Thus, the results of our previous paper were a syntactic reformulation of some well-known properties which hold in parametric models of $\Nd$ (see \cite{Bainbridge1990} and \cite{Girard1992}).

\begin{remark} The proof of Proposition~\ref{eps-preservation} scales straightforwardly to the whole of $\Np$ (see Appendix~\ref{eps-preservation-lambda}), and it actually holds if one replaces $\eta$ (resp.~$\gamma$) with the more general $\eta^+$ (resp.~$\gamma^+$) (see Remark~\ref{rem:gammapiu} above, Remark~\ref{gammag-lambda} in Appendix~\ref{betaetagamma} and Section~\ref{gammag-subs} below). The class of $\varepsilon$-conversions needed to establish the analog of Proposition~\ref{eps-preservation} for $\Nb$ are  depicted in Table~\ref{epsi-conv}.
\end{remark}

%\begin{remark}We will refer to the left-to-right orientation of the $\varepsilon$-conversions as $\varepsilon$-permutation.
%\end{remark}

\begin{table}[!h!]

\rule{\textwidth}{.5pt}

\medskip
\begingroup\makeatletter\def\f@size{8}\check@mathfonts
\begin{lrbox}{\mypti}% Store prooftree in \mypti
\raisebox{2.5ex}[-0ex][0ex]{\begin{varwidth}{\linewidth}
$\cexpD{X}{F_1\supset F_2\supset  X}{\AXC{$A$}\noLine\UIC{$\mathscr{D}_2$}\noLine\UIC{$B$}\DP}$
\end{varwidth}}
\end{lrbox}
\begin{lrbox}{\mypto}% Store prooftree in \mypti
\raisebox{2.5ex}[-0ex][0ex]{\begin{varwidth}{\linewidth}
$\cexpD{X}{F_3\supset X}{\AXC{$A$}\noLine\UIC{$\mathscr{D}_2$}\noLine\UIC{$B$}\DP}$
\end{varwidth}}
\end{lrbox}
\adjustbox{center, scale=0.8}{$
%\begin{equation*} 
  % \begin{matrix}   
    \AXC{$\mathscr{D}_1$}
\noLine
\UnaryInfC{$\forall X. \curc(F_1,F_2,F_3)$}\UIC{\SF{X}{\curc(F_1,F_2,F_3)}{A}}
\AxiomC{$\SF{X}{F_1\supset F_2\supset  X}{A}$}
\RightLabel{\footnotesize$\supset$E}
\BinaryInfC{$\SF{X}{(F_3\supset X)\supset X )}{A}$}
\AXC{$\SF{X}{F_3\supset X}{A}$}
\RightLabel{\footnotesize$\supset$E}
\BIC{$A$}
\noLine
\UnaryInfC{$\mathscr{D}_2$}
\noLine
\UnaryInfC{$B$}
\DisplayProof\qquad\qquad\qquad\qquad 
%\end{equation*}
$}
                   \begin{equation*}\tag{$\varepsilon$}\label{epsilonew}%\text{
\ \ \rightsquigarrow_{\varepsilon} \ \ 
\end{equation*}
\adjustbox{center,scale=0.8}{$
%\begin{equation*}
         %          \def\defaultHypSeparation{\hskip .05in}
\qquad\qquad  \qquad\qquad\qquad\qquad\qquad\qquad\qquad\qquad\qquad\AXC{$\mathscr{D}_1$}
\noLine
\UnaryInfC{$\forall X. \curc(F_1,F_2,F_3)$}\RightLabel{\footnotesize$\forall $E}\UIC{$\SF{X}{\curc(F_1,F_2,F_3)}{B}$}\AxiomC{\usebox{\mypti}}
  \RightLabel{\footnotesize$\supset$E}
  \BIC{$\SF{X}{(F_3\supset X)\supset X}{B}$} \AxiomC{\usebox{\mypto}}
 \RightLabel{\footnotesize$\supset$E}
 \BinaryInfC{$B$}\DP
% \end{matrix}
%\end{equation*}
$}
%\end{equation*}%}

%}
\endgroup
% \smallskip
% $$\text{(for $F$ nested sp-$X$)}$$

\medskip
\rule{\textwidth}{.5pt}
\caption{$\varepsilon$-conversion}\label{epsi-conv}
\end{table}

\section{Predicative  translations via atomization}\label{sec-ferris}

In this section we show that  the RP-translation can be related to  the predicative translations by embedding the $\Ndrp$ fragment of $\Nd$  into the atomic fragment $\Nda$. % these %result (up to $\beta$-equivalence) from an atomization procedure for $\Nd_{RP}$ based on the $\varepsilon$-rule.
%This leads us to slightly reformulate the definitions from \cite{Ferreira2006,Ferreira2009}. 

\subsection{The FF- and ESF-translations}

As we recalled in the introduction, Ferreira and Ferreira \cite{Ferreira2006, Ferreira2009, Ferreira2013, Ferreira2017}, proposed  an alternative translation of $\NI$-derivations into $\Nd$ that we  call the  FF-translation. The FF-translation agrees with the RP-translation on how to translate formulas, but not on how to translate derivations. In particular, the FF-translation  does not use the full power of the $\forall\text{E}$ rule of $\Nd$, but rather it maps  derivations in $\NI$ into derivations in $\Nda$.

%, the fragment of $\Nd$ obtained by restricting the witnesses of the rule % witnesses of
%$\forall\text{E}$ to atomic formulas:
%%
%\begingroup\makeatletter\def\f@size{10}\check@mathfonts
% $$\AXC{$\forall X. A$}\RL{\footnotesize$\forall\text{E}$$_{at}$}\UIC{$A\llbracket Y/X\rrbracket$}\DP$$
%\endgroup
% \noindent
%in which the witnesses are restricted to atomic formulas (for uniformity,  we will refer to this system as $\Nda$).

The FF-translation exploits the property of \emph{instantiation overflow}:
\begin{definition}
  A formula $\forall X. A\in\Ld$ enjoys the instantiation overflow property iff for all formulas $B\in\Ld$
  \begingroup\makeatletter\def\f@size{10}\check@mathfonts
  $$\forall X. A \vdash_{\Nda} A\llbracket B/X\rrbracket $$
\endgroup  
\end{definition}
While Ferreira and Ferreira only address standard intuitionistic connectives, it is easily seen that the instantiation overflow property holds for all universal polynomial formulas $\forall X.A \in\Ldrp$  (and actually for many more, see \cite{Pistone2018}). The results of Ferreira and Ferreira thus scale to the whole system $\Np$. For simplicity we will focus here on the  connective $\bullet$ and on the fragments $\Nb$ and $\Ndrpc$ of $\Np$ and $\Ndrp$ respectively. However,  all definitions and results here presented scale to all polynomial connectives, as shown in Appendix.
%However, for readability we will focus here only on the fragment $\Nb$, postponing the details for the full system in  appendix.

%This property is enjoyed by a class of formulas (see \cite{Pistone2018} for an exact characterization) among which there are all  formulas of the form $\forall X. A$ that are the RP-translation of some formula in $\Lv$. %For all these formulas $A[B/X]$ follows from $\forall X. A$ for any formulas $B\in\Ld$.%: Each instance of the full rule $\forall E$ is derivable in $\Nda$ in which the premise is the RP-translation of a propositional formula. %Hence their variant of the RP-translation, that we will call \emph{FF-translation}, translates derivations in $\NI$ into derivations in $\Nda$, the fragment of $\Nd$ in which the rule $\forall\text{E}$ is restricted to atomic instantiations. 

%To better focus on (the internal structure of) derivations rather than on derivability, 

By reformulating Ferreira and Ferreira's insight, we define an embedding of $\Ndrpc$ into $\Nda$ that we call FF-atomization, and using it we define the FF-translation from $\Nb$ into $\Nda$ as the composition of the RP-translation and FF-atomization.

\begin{definition}[FF-atomization, FF-translation]\label{all-at}
If $\D D$ is an $\Ndrpc$-derivation,  the {\em FF-atomization}  of $\D D$, which we indicate as $\D D^{\downarrow}$,  is the $\Nda$-derivation defined by induction on $\D D$ as follows. We only consider the case in which the last rule $\D D$ is $\forall\text{E}$  with a non-atomic witness,  since all other rules
are translated in a trivial way. In this case observe that 
\begingroup\makeatletter\def\f@size{10}\check@mathfonts
$$\D D= {\small\AXC{$\mathscr{D}'$}\noLine\UIC{$\forall X. \curc(A,B,C)$}\RightLabel{$\forall\text{E}$}\UIC{$\curc(A,B,C)\llbracket F/X\rrbracket$}\DP}$$ 
\endgroup
%\noindent
%where $A\curlyvee B$ is short for $(A\supset X)\supset (B\supset X)\supset X$). %(analogous cases in which the premise of $\forall\text{E}$ is $(A\wedge B)^*$ or $\bot^*$ are left to the reader).

We define $\D D^{\downarrow}$ by a sub-induction on  $F$.

\begin{itemize}

\item If $F= F_1\supset F_2$ then 
 \begingroup\makeatletter\def\f@size{8}\check@mathfonts
\begin{lrbox}{\mypti}% Store prooftree in \mypti
\begin{varwidth}{\linewidth}
$\AXC{$\mathscr{D}'$}\noLine\UIC{$\forall X. \curc(A,B,C)$}\RightLabel{$\forall\text{E}$}\UIC{$\curc(A,B,C)\llbracket F_2/X\rrbracket$}\DP$
\end{varwidth}
\end{lrbox}

 $$%\!\!\!\!\!\!\!\!\!\!\!\!\!\!\!\!\!\!\!\!\!\!\!\!\!\!\!\!\!\!\!\!\!\!\!\!\!\!\!\!\!\!\!\!
 \mathscr{D}^{\downarrow}=\def\defaultHypSeparation{\hskip .05in}\AXC{\raisebox{4ex}{$\left(\usebox{\mypti}\right)^{\downarrow}$}}\AXC{$\stackrel{k_1}{A\supset B\supset (F_1\supset F_2)}$}\AXC{$\stackrel{o}{A}$}\RightLabel{$\mathord{\supset}\text{E}$}\BIC{$B\supset F_1\supset F_2$}\AXC{$\stackrel{o'}{B}$}\RightLabel{$\mathord{\supset}\text{E}$}\BIC{$F_1\supset F_2$}\AXC{$\stackrel{m}{F_1}$}\RightLabel{$\mathord{\supset}\text{E}$}\BIC{$F_2$}\RightLabel{$\mathord{\supset}\text{I}$ $(o')$}\UIC{$B\supset  F_2$}\RightLabel{$\mathord{\supset}\text{I}$ $(o)$}\UIC{$A\supset  B\supset F_2$}\RightLabel{$\mathord{\supset}\text{E}$}\insertBetweenHyps{\hspace{-2 em}}\BIC{$(C\supset F_2)\supset F_2$}\AXC{$\stackrel{k_2}{C\supset (F_1\supset F_2)}$}\AXC{$\stackrel{o}{C}$}\RightLabel{$\mathord{\supset}\text{E}$}\BIC{$F_1\supset F_2$}\AXC{$\stackrel{m}{F_1}$}\RightLabel{$\mathord{\supset}\text{E}$}\BIC{$F_2$}\RightLabel{$\mathord{\supset}\text{I}$ $(o)$}\UIC{$C\supset  F_2 $}\RightLabel{$\mathord{\supset}\text{E}$}\BIC{$F_2$}\RightLabel{$\mathord{\supset}\text{I}$ $(m)$}\UIC{$F_1\supset F_2$}\RightLabel{$\mathord{\supset}\text{I}$ $(k_2)$}\UIC{$(C\supset (F_1\supset F_2))\supset (F_1\supset F_2)$}\RightLabel{$\mathord{\supset}\text{I}$ $(k_1)$}\UIC{$\curc(A,B,C)\llbracket F_1\supset F_2/X\rrbracket$}\DP $$
\endgroup

  \item If $F=\forall Z.  F'$, the clause is analogous to the previous one (for a fully detailed definition, see Definition~\ref{at-lambda} in Appendix~\ref{ff-def-appendix})

\end{itemize}

\noindent
 If $\mathscr{D}$ is an $\Nb$-derivation, we call the $\Nda$-derivation $\mathscr{D}^{*\downarrow}$ the %Ferreira-Ferreira translation
 FF-translation of $\mathscr{D}$.% (i.e.~the FF-translation of $\D D$ is defined as  the FF-atomization of the  RP-translation $\D D^*$ of $\D D$).% (which we abbreviate to ).

\end{definition}

%
% Given the  FF-translation, it is immediate that $\forall X. \curc(A,B,C)$ enjoys instantiation overflow: %  FF-translation, and not the other way round:
% \begin{proposition}[Instantiation overflow for $\forall X. \curc(A,B,C)$] \label{io-or}
%  For any $A, B, C, D \in \Ld$,
%  $$\forall X.\curc(A,B,C) \vdash_{\Nda} \curc(A,B,C)\llbracket D/X\rrbracket$$
% \end{proposition}

% \begin{proof}  
% The proposition follows from the fact that for all $C \in \Ld$ $$\left(\AXC{$\forall X. ((A\supset X)\supset (B\supset X)\supset X )$}\RL{\small$\forall\text{E}$}\UIC{$(A\supset C)\supset (B\supset C)\supset C$}\DP\right)^{\downarrow}$$
% is an  $\Nda$-derivation.
% \end{proof}

\begin{remark}
Ferreira and Ferreira present their result in a different way, by using the inductive clauses of Definition~\ref{all-at} to give a direct proof of instantiation overflow for the universal formulas of the form $(A\vee B)^{*}$, and they refer to what we here called the FF-translation of $\D D$ as to ``the canonical translation of $\D D$ in $F_{at}$ provided by instantiation overflow'' \cite{Ferreira2017}. This difference in presentation will allow a more straightforward formulation of our results. 
\end{remark}

Whereas the FF-%and the ${\varepsilon}$-
translation is defined by combining the RP-translation from $\Nb$ into $\Ndrp$ with  the FF-atomization embedding  $\Ndrp$ into $\Nda$%. 
%
%
% M
, more recently Esp\'irito Santo and Ferreira  \cite{ESF19} introduced an alternative translation by ``directly'' defining an embedding $\sharp$ from $\NI$ into $\Nda$. We will refer to this translation as the \emph{ESF-translation}. Using the notation we introduced, the definition of $\sharp$ can be adapted to $\Nb$ (the definition actually scales to the whole of $\Np$, see Def.~\ref{reftrans} in Appendix~\ref{ff-def-appendix}), and the crucial case of the definition, that of an $\Nb$-derivation ending with an application of $\bullet$E runs as follows (assuming $\mathsf{at}(C^*)=Z$):

\noindent
\resizebox{\textwidth}{!}{
$\left( \def\defaultHypSeparation{\hskip .01in}\AXC{$\D D_1$}\noLine\UIC{$\bullet(A_1,A_2,A_3)$}\AXC{$[A_1][A_2]$}\noLine\UIC{$\D D_2$}
 \noLine\UIC{$C$}\AXC{$[A_3]$}\noLine\UIC{$\D D_3$}\noLine\UIC{$C$}\RL{$\bullet$E}\TIC{$C$}\DP \right)^{\sharp} = 
\def\defaultHypSeparation{\hskip .02in}\AXC{$\D D_1^{\sharp}$}
\noLine
\UIC{$\forall X. \curc(A_1,A_2,A_3)$}
\RightLabel{$\forall\text{E}_{at}$}
\UIC{$\curc(A_1,A_2,A_3)\llbracket Z/X\rrbracket$}
\AXC{$[\stackrel{n_1}{A_1^*}][\stackrel{n_2}{A_2^*}]$}\noLine\UIC{$\D D_2^{\sharp}$}\noLine\UIC{$C^*$}\noLine\UIC{$\rr{\TT{El}}^{\vec m}_{C^*}$}\noLine\UIC{$Z$}\RL{$\imp$I $(n_2)$}\UIC{$A_2^*\imp Z^*$}\RL{$\imp$I $(n_1)$}\UIC{$A_1^*\imp A_2^*\imp Z^*$}\BIC{$\SF{X}{(A_3^*\imp X)\imp X}{Z}$}\AXC{$\stackrel{m}{[A_3^*]}$}\noLine\UIC{$\D D_2^{\sharp}$}\noLine\UIC{$C^*$}\noLine\UIC{$\rr{\TT{El}}^{\vec m}_{C^*}$}\noLine\UIC{$Z$}\RL{$\imp$I $(m)$}\UIC{$A_3^*\imp Z^*$}\BIC{$Z^*$}\noLine\UIC{$\phantom{(\vec m)}
  \rr{\TT{In}}_{C^*}\
(\vec m)
  $}
\noLine\UIC{$C^*$}%\RightLabel{$\mathord{\supset}\text{I}$ $(k_2)$}\UIC{$A_3^*\supset C^*)\supset C^*$}\RightLabel{$\mathord{\supset}\text{I}$ $(k_1)$}\UIC{$\curc(A^*_1,A^*_2,A_3^*)\llbracket C^*/X\rrbracket$}
\DP$
}

%\medskip
%Esp\'irito Santo and Ferreira \cite{ESF19} show that the FF-translation of $\D D$ $\beta$-reduces to $\D D^\sharp$.

\medskip
What is remarkable about these two translations %(henceforth referred to as {\em atomic translations})
is not only that they show that one can translate $\Nb$ using only a very weak predicative fragment of $\Nd$, but also the fact that, unlike the RP-translation, they  preserve the equivalence induced by $\gamma$- and $\eta$-conversion. In fact, both translations  map  not only $\beta$-, but also $\eta$- and $\gamma$-reduction steps in $\NI$ onto chains of $\beta\eta$-reduction steps (respectively $\beta\eta$-equivalences)  in $\Nd$% (hence, a fortiori, $\eta$- and $\gamma$-equivalence)%(only for disjunction, see below Sec.~\ref{__})
: %Actually, it does also preserve the equivalence $=_{\gamma_{g}}$ generated by generalized permutations (see \cite{StudiaLogica}, p.5).
%However, the FF-translation cannot be used to construct an isomorphism between a formula and its translation (as explained later in remark \ref{ferreiraiso}). 
%
%
%\subsection{Two translations of $\NI$ into $\Nda$}% FF-translation and its alternative}
%
% -As mentioned before, Ferreira and Ferreira offered an alternative solution to the preservation problem consisting in replacing the RP-translation with %by showing that if two $\NI$-derivations are $\gamma$-equivalent or $\eta$-equivalent
%their FF-translation:%s are $\beta\eta$-equivalent $\Nda$-derivations
%. %In particular, they have shown that given their mapping,  $\vee$, $\wedge$ and $\bot$ connectives can be translated into $\beta\eta$ redexes of $NI^{2}$, by using a variant of Prawitz translation which embeds 
% as well.
%That is, if two $\Nv$-derivations are ${\eta}$-equivalent, then their FF-translations are $\beta\eta$-equivalent $\Nda$-derivations. %the schema translating disjunctive $\eta$-reduction \eqref{eta} in the variant of Russell-Prawitz translation presented in \cite{Ferreira2009} is derivable from $\beta$ and $\eta$ reduction for the system $\B F_{at}$.
%More precisely,
\begin{proposition}\label{ff-preservation}
  For all $\Nb$-derivations $\D D_{1},\D D_{2}$, if $\D D_{1}\rightsquigarrow_{\beta\eta\gamma} \D D_{2}$, then
  \begin{enumerate}
  \item $\D D_{1}^{*\downarrow} \rightsquigarrow_{\beta\eta} \D D^{*\downarrow}_{2}$.
%  \item $\D D_{1}^{*\downarrow^{\varepsilon}} \rightsquigarrow_{\beta\eta} \D D^{*\downarrow^{\varepsilon}}_{2}$.
  \item $\D D_{1}^{\sharp} \simeq_{\beta\eta} \D D^{\sharp}_{2}$.
  \end{enumerate}
\end{proposition}

\begin{proof}Point 1 of the proposition was proved for the whole of $\NI$ by Ferreira and Ferreira \cite{Ferreira2009} (for $\beta\gamma$)  and by Ferreira \cite{Ferreira2017} for $\eta$ (although in this case by  adding a primitive conjunction to $\Nd$, see below Section~\ref{gammag-subs}), and we claim that, at the cost of tedious computations, those proofs can be scaled to the whole of $\Np$. In Appendix~\ref{etagamma-appendix} we give a proof scaled to the whole of $\Np$ of point 2, %. and 3. %$\sharp$. %defined below
that  generalizes the analogous result for $\NI$ of %another
  % proof given by
  Esp\'irito Santo and Ferreira in \cite{ESF19}.%The proof follows the same pattern of the one of  Ferreira and Ferreira [\cite{Ferreira2013}] for $\Nv$.
\end{proof}

\subsection{The $\varepsilon$-translation}

We now show that the $\varepsilon$-rule can be used to clarify the relationship between the RP-translation and the FF- and ESF-translations. To see how, 
%
%One may wonder whether the relationship between the original RP-translation on the one hand and the two translations into $\Nda$ on the other hand can be made precise. %  an $\NI$-derivation $\D D$, its FF- and  FF$^{\varepsilon}$-translations are $\beta\eta$-equivalent. Thus , if two $\NI$-derivations $\D D_1$ and $\D D_2$ Second,  is   the alternative translation use the $\epsilon$-equations introduced in the previous paper to clarify the relationship between the RP-translation and the FF-translation (and its al
%%
% ^describe this connection in detail.
% In this section we give a precise answer to this question by
%To answer this question, %
%
% As the next definition  shows, %this is in fact the case:
we  define  an alternative atomization procedure yielding yet another translation of $\Nb$ into $\Nda$ (also this atomization  scales to the whole of $\Np$, see Definition~\ref{ateps-lambda} in Appendix~\ref{ff-def-appendix} :

%These results suggest the existence of a tight connection between the FF-translation and \eqref{epsilonew}.
%To highlight the connection between Proposition~\ref{eps-preservation} and Proposition~\ref{ff-preservation}, we first define an almost identical notion of atomization using the notion of $C$-expansion:%, and, in particular, that using \eqref{epsilonew} the Ferreira and Ferreira's instantiation overflow (and thus the the embedding $\downarrow$ from \Nd to \Nda) could be generalized to cover any $X$-safe sp-$X$ derivation in \Nd in which the premises of all applications of $\forall $E are quasi sp-$X$.

\begin{definition}[$\varepsilon$-atomization, $\varepsilon$-translation]\label{all-at-alt}%If $\D D$ is an RP-$\Nd$-derivation,  the {\em $\forall\text{E}$-atomization} (or simply atomization) of $\D D$, which we indicate as $\D D^{\downarrow}$,  is the $\Nda$-derivation defined by induction on $\D D$ as follows. We only consider the case in which the last rule $\D D$ is $\forall\text{E}$  since all other rules
%are translated in a trivial way. In this case observe that
The definition differs from Definition~\ref{all-at}  in the following respect: In case 
\begingroup\makeatletter\def\f@size{10}\check@mathfonts
$$\D D= \AXC{$\mathscr{D}'$}\noLine\UIC{$\forall X. \curc(A,B,C)$}\RightLabel{$\forall\text{E}$}\UIC{$\curc(A,B,C)\llbracket F/X\rrbracket$}\DP$$
\endgroup
where  $F$ is  not atomic and  $Z$ is the variable in rightmost position in $F$, using the notation introduced in Sections~\ref{weakexp-sec}--\ref{weakexp-sec} (see in particular Remark~\ref{cexpweakexp}), %
%
%, %and assuming that $\D D$ is $X$-safe
we define% (assuming
:
%{\small
\begingroup\makeatletter\def\f@size{10}\check@mathfonts
\begin{lrbox}{\mypte}% Store prooftree in \mypti
\begin{varwidth}{\linewidth}
$\SF{X}{A\supset B\imp  X}{\AXC{$F$}\noLine\UIC{$\D E^{\vec m}_{F}$}\noLine\UIC{$Z$}\DP}$
\end{varwidth}
\end{lrbox}

\begin{lrbox}{\mypta}% Store prooftree in \mypti
\begin{varwidth}{\linewidth}
\AXC{$\D D'^{\downarrow^{\varepsilon}}$}
\noLine
\UIC{$\forall X. \curc(A,B,C)$}
\RightLabel{$\forall\text{E}_{at}$}
\UIC{$\curc(A,B,C)\llbracket Z/X\rrbracket$}
\DP
\end{varwidth}
\end{lrbox}

\begin{lrbox}{\mypto}% Store prooftree in \mypti
\begin{varwidth}{\linewidth}
\AXC{$\ddots$}
\noLine
\UIC{$\forall \OV Y_{n}(F_{n}[Y_{B}/Y]\supset Y_{B})$}
\doubleLine
\RightLabel{$\forall_{X} E$}
\UIC{$F_{n}[Y_{B}/Y]\supset Y_{B}$}
\DP
\end{varwidth}
\end{lrbox}

\begin{lrbox}{\myptu}% Store prooftree in \mypti
\begin{varwidth}{\linewidth}
$\phantom{A}$

\vspace{-1.7cm}
$\stackrel{k_2}{\SF{X}{C\supset X}{\AXC{$F$}\noLine\UIC{$\D E^{\vec m}_{F}$}\noLine\UIC{$Z$}\DP}}$
\end{varwidth}
\end{lrbox}

% \adjustbox{scale=0.6,center}{
$$\D D^{\downarrow^{\varepsilon}}\qquad= \AXC{\usebox{\mypta}}
\AXC{$\stackrel{k_1}{\usebox{\mypte}}$}
\RightLabel{$\mathord{\supset}\text{E}$}
\BIC{$(C\supset Z)\supset Z$}
\AXC{\usebox{\myptu}}
\RightLabel{$\mathord{\supset}\text{E}$}
\BIC{$Z$}
\noLine\UIC{$\phantom{(\vec m)}
  \D I_F\
(\vec m)
  $}
\noLine\UIC{$F$}\RightLabel{$\mathord{\supset}\text{I}$ $(k_2)$}\UIC{$(C\supset F)\supset F$}\RightLabel{$\mathord{\supset}\text{I}$ $(k_1)$}\UIC{$\curc(A,B,C)\llbracket F/X\rrbracket$}
\DP
$$\endgroup
%}
%  where $F'$ is the formula $F_{1}\supset \forall \OV Y_{2}(F_{2}[Y_{B}/X]\supset \dots \supset \forall \OV Y_{n}(F_{n}[Y_{B}/X]\supset Y_{B})\dots )$.
If $\D D$ is an $\Nb$-derivation, we call the $\Nda$-derivation $\D D^{*\downarrow^{\varepsilon}}$ the \emph{${\varepsilon}$-translation of $\D D$}.
\end{definition}

The name ``$\varepsilon$-atomization'' is justified by the fact that, as can be easily verified in the case of $\circ$, the $\varepsilon$-atomization $\D D^{\downarrow^{\varepsilon}}$ of an $\Ndrpc$-derivation $\D D$ is the result of applying $\eta$-expansions and $\varepsilon$-conversions to $\D D$:

\begin{proposition}\label{rp-ffe-relation} If $\D D$ is an $\Ndrpc$-derivation, then   $\D D \simeq_{\eta\varepsilon} \mathscr{D}^{\downarrow^{\varepsilon}}$.

\end{proposition}

\begin{proof} 
See  Appendix \ref{rp-ffe-appendix}.
 \end{proof}

The relationship between the three predicative translations is very close. In fact, they yield $\beta$-equivalent derivations, with the $\varepsilon$-translation lying between the FF-translation and the ESF-translation, in the sense below:

% by eliminating some of the $\beta$-redexes introduced by the FF-atomization of the RP-translation of some derivation $\D D$ one obtains the $\varepsilon$-translation of $\D D$, and by further $\beta$-reducing it one obtains the ESF-translation of $\D D$%that is  the ${\varepsilon}$-translation  results from the  FF-translation , and the ESF-translation %(and by permorming some $\eta$-expansions):

%\begin{proposition}\label{ff-ffe-relation}If  $\D D$ is an $\Ndrpc$-derivation,  then $\D D^{\downarrow} \rightsquigarrow_{\beta%\eta
%  } \D D^{\downarrow^{\varepsilon}}$.
%  \end{proposition}

 % \begin{proof}% The proof is by induction on $\mathscr{D}$. If $\mathscr{D}$ ends with an application of either $\mathord{\supset}\text{I}$, $\mathord{\supset}\text{E}$, or $\forall\text{I}$, or $\forall\text{E}$ with an atomic witness then it is enough to apply the induction hypothesis to the immediate sub-derivations of $\mathscr{D}$. 

  % If $\mathscr{D}$ ends with an application of $\forall\text{E}$ with a non-atomic witness $F$ we reason by induction on $F$:
  % if $F=C\supset D$,   we assume that $Z$ is the rightmost atomic formula in $F$. The lengthy computation is postponed to
%See   Appendix~\ref{prova-app}.%;
  % if $F=\forall Z.  D$ the case is similar to the previous one.
%\end{proof}

%The same relationship holds between  the ${\varepsilon}$- and the ESF-translation:

\begin{proposition}\label{ffe-esf-relation}
For all $\Nb$-derivations $\D D$, $\D D^{*\downarrow}\rightsquigarrow_{\beta} \D D^{*\downarrow^{\varepsilon}}\rightsquigarrow_{\beta} \D D^{\sharp}$.
\end{proposition}
\begin{proof}See Appendix~\ref{espirito-appendix}.
\end{proof}

\begin{remark}
That $\D D^{*\downarrow}\rightsquigarrow_{\beta} \D D^{\sharp}$ 
was already known from \cite{ESF19}.
\end{remark}

% Thus, by putting together Proposition~\ref{ff-preservation} and \ref{ffe-esf-relation} we obtain a new proof of Esp\'irito Santo and Ferreira's result:

% \begin{corollary}
% For all $\Nb$-derivations $\D D$, $\D D^{*\downarrow}\rightsquigarrow_{\beta} \D D^{\sharp}$.
% \end{corollary}

By putting together Proposition~\ref{ff-preservation} and %any among Proposition~\ref{ff-ffe-relation} and
\ref{ffe-esf-relation} we can deduce that also the $\varepsilon$-translation preserves permutative conversions and $\eta$-conversions:% (but the latter only as equivalences, see Remark \ref{etaboh} below):

\begin{corollary}\label{corolla}
  For all $\Nb$-derivations $\D D_{1},\D D_{2}$, if $\D D_{1}\simeq_{\beta\eta\gamma} \D D_{2}$, then
 $\D D_{1}^{*\downarrow\varepsilon} \simeq_{\beta\eta} \D D^{*\downarrow^\varepsilon}_{2}$.
\end{corollary}

\begin{remark}\label{etaboh}
  The statement of Corollary \ref{corolla} cannot be expressed in terms of reduction, but only in terms of equivalence,
 due to the fact that neither $\D D_1\rightsquigarrow_{\eta}\D D_2$ nor $\D D_1\rightsquigarrow_{\gamma}\D D_2$  imply $\D D_1^{*\downarrow^{\varepsilon}}\rightsquigarrow_{\beta\eta} \D D_2^{*\downarrow^{\varepsilon}}$, but only $\D D_1^{*\downarrow^{\varepsilon}}\simeq_{\beta\eta} \D D_2^{*\downarrow^{\varepsilon}}$.
\end{remark}

By combining Proposition~\ref{rp-ffe-relation} with Proposition~\ref{ffe-esf-relation} we deduce that  the FF- and ESF-translations are $\beta\eta\varepsilon$-equivalent to the RP-translation:
\begin{corollary}\label{rp-ff-relation}
For all $\Nb$-derivations $\mathscr{D}$, $\mathscr{D}^* \simeq_{\beta\eta\epsilon} \mathscr{D}^{*\downarrow}$ and $\mathscr{D}^* \simeq_{\beta\eta\epsilon} \mathscr{D}^{\sharp}$.
\end{corollary}

The relationship between the four different translations is illustrated in Table \ref{unicfig}.%, in the case of two $\Np$-derivations
%$\D D_{1}\rightsquigarrow_{\beta\eta\gamma} \D D_{2}$.

\begin{table}[h!]
\begin{center}
  \rule{\textwidth}{.5pt}
%$\Np$-derivations
  For all $\Nb$ derivations $\D D_1$ and $\D D_2$ such that  $\D D_{1}\rightsquigarrow_{\beta\eta\gamma} \D D_{2}$:

  \medskip
%\adjustbox{scale=0.8}{
\begin{tikzcd}   & \D D_{1}^{*}  \ar[rotate=270,symbol=\simeq_{\eta\varepsilon}]{d}%[below right]{\eta\varepsilon}
&   & \\
\D D_{1}^{*\downarrow}   \ar[symbol=\rightsquigarrow_{\beta}]{r}{} \ar[rotate=270,symbol=\dsarrow_{\beta\eta}
]{d}&  \D D_{1}^{*\downarrow^{\varepsilon}} \ar[rotate=270,symbol=\simeq_{\beta\eta}]{d} \ar[symbol=\rightsquigarrow_{\beta}]{r}{} &   \D D_{1}^{\sharp}  \ar[rotate=270,symbol=\simeq_{\beta\eta}]{d}
 %[below right]{\beta\eta}
 & \\
 \D D_{2}^{*\downarrow}     \ar[symbol=\rightsquigarrow_{\beta}]{r}{}&  \D D_{2}^{*\downarrow^{\varepsilon}}   \ar[rotate=270,symbol=\simeq_{\eta\varepsilon}]{d}\ar[symbol=\rightsquigarrow_{\beta}]{r}{}&   \D D_{2}^{\sharp}  & \\
            & \D D_{2}^{*} &   & \\
\end{tikzcd}
%}
\rule{\textwidth}{.5pt}
\end{center}
\caption{\small Relationship among the RP-, FF-, ESF- and $\varepsilon$-translations. %of two %$\Np$-derivations
%$\D D_{1}\rightsquigarrow_{\beta\eta\gamma} \D D_{2}$.
}
\label{unicfig}
\end{table}

Summing up, all three predicative translations are equivalent, modulo $\beta\eta\varepsilon$-equivalence to the RP-translation and thus, semantically, they all interpret an $\Nb$-derivation using the same second-order proof (provided one understands second-order proofs as satisfying the naturality conditions expressed by $\varepsilon$-conversions). %If we accept 

\section{\mbox{RP-translation and  $\simeq_{\beta\eta\epsilon}$ vs  predicative translations and $\simeq_{\beta\eta}$}}%Predicative translations and identity of proofs}
\label{invertib-section}

%}\label{invertib-section}

%In this section we compare the semantic-oriented approach we followed, based on the $\varepsilon$-rule and its grounding in the categorical semantics of polymorphism, with the syntactic-oriented approach based on atomic polymorphism. 

%The investigations of identity of proofs in second-order logical system (corresponding to the study of program equivalence, in the computer science perspective) usually require a sophisticated semantic ideas and techniques (e.g. parametricity or dinaturality), resulting in the study of complicated and often undecidable mathematical structures.  

Given the results of the previous sections, it might be tempting to say that the approach based on atomic polymorphism might provide a fully syntactic alternative to categorical semantics and related techniques for the study of identity of proofs for propositional connectives.

In this section we argue that this is not entirely the case, by stressing two important limitations of the predicative translations. First, they do not preserve the whole equational theory of propositional connectives, and in particular the predicative translations of $\land,\lor$ do not yield products and co-products in $\Nda$. Hence, \emph{strictu sensu} the predicative translations do not  fully solve the equivalence-preservation problem. Second,  the $\forall$E rule restricted to atomic witnesses is too weak to  prove the logical equivalence between a connective and its second-order translation.
In Prawitz'  terminology \citep{Prawitz1965}, we show that the connective $\lor$ is not \emph{strongly} definable in F$_{at}$ but only \emph{weakly} definable.

Conversely, we highlight that the approach based on atomic polymorphism looks more apt to the syntactic study of proof reduction, since the rewriting theory induced by $\varepsilon$-conversions looks rather intricate.

\subsection{Predicative translations and generalized permutations}\label{gammag-subs}
% \PP{Da Systemare}
% \begin{remark}\label{rem-gammastar}
%The analogy between Proposition~\ref{ff-preservation} and Proposition~\ref{eps-preservation} may suggest that the atomic translations deliver a solution to the equivalence-preservation problem as well, by actually sharpening the one discussed in our previous paper in that (in the case of the FF- and ESF-translations) not only equivalence, but reduction is  preserved. This is however not the case.%, as Proposition~\ref{ff-preservation}  cannot be extended to the generalization of $\gamma$-conversion needed for interpreting  $\NI$ as a bicartesian closed category. %(As a consequence of this, although the proposition holds for disjunction, it only holds for conjunction when its elimination rule is in general form.)
%  For more details, see below Section~\ref{invertib-section}.
%\end{remark}

As is well known, in order to obtain a perfect match between the syntax of $\NI$ and the free bi-cartesian closed category, it is necessary to consider the following  stronger  permutative conversions in which every chunk of derivation (and not just those consisting of applications of elimination rules) can be permuted-up across an application of disjunction elimination:\footnote{These stronger  permutations were first formulated in natural deduction format  by \citet{See79} and have been recently discussed in the context of proof-theoretic semantics by \citet{Tra16,Tra18}.}\footnote{\label{gammapiufoot}The general form of the $\gamma^+$-conversions for an arbitrary polynomial connective (of which the one given here are a simple instance) is given in Appendix~\ref{betaetagamma}.}

\begingroup\makeatletter\def\f@size{10}\check@mathfonts
\begin{equation*}%\tag{$\gamma_g$}
  \label{gammaorG}
\AxiomC{$\D D$}
\noLine
\UnaryInfC{$A\lor B$}
\AxiomC{$\stackrel{n}{[A]}$}
\noLine
\UnaryInfC{$\D D_{1}$}
\noLine
\UnaryInfC{$C$}
\AxiomC{$\stackrel{m}{[B]}$}
\noLine
\UnaryInfC{$\D D_{2}$}
\noLine
\UnaryInfC{$C$}
\RightLabel{\footnotesize$\lor$E $(n,m)$}
\TrinaryInfC{$[C]$}
\noLine
\UnaryInfC{$\D D_3$}
\noLine
\UnaryInfC{$ D$}
\DisplayProof
\ \ \rightsquigarrow_{\gamma^+} \!\!\!
\AxiomC{$\D D$}
\noLine
\UnaryInfC{$A\lor B$}
\AxiomC{$\stackrel{n}{[A]}$}
\noLine
\UnaryInfC{$\D D_{1}$}
\noLine
\UnaryInfC{$[C]$}
\noLine
\UnaryInfC{$\D D_3$}
\noLine
\UnaryInfC{$ D$}
\AxiomC{$\stackrel{m}{[B]}$}
\noLine
\UnaryInfC{$\D D_{2}$}
\noLine
\UnaryInfC{$[C]$}
\noLine
\UnaryInfC{$\D D_3$}
\noLine
\UnaryInfC{$ D$}
\RightLabel{\footnotesize$\lor$E $(n,m)$}
\TrinaryInfC{$D$}
\DisplayProof
\end{equation*}
\endgroup

Proposition~\ref{eps-preservation} can be strengthened by replacing $\gamma$ with $\gamma^+$, so that the RP-translations of the left- and right-hand side of any instance of $\gamma^+$  are  $\beta\varepsilon$-equivalent  $\Nd$-derivations   (see Appendix~\ref{eps-preservation-lambda}).
%this follows from the fact that Proposition~\ref{eps-preservation} keeps on holding if one replaces $\gamma$ with  $\gamma_g$, see~Proposition~\ref{} in Appendix~\ref{}, and Corollary~\ref{rp-ff-relation}).

However, Proposition \ref{ff-preservation} ceases to hold as soon as one replaces $\gamma$ with $\gamma^+$ (even if one replaces reduction $\rightsquigarrow$ with equivalence $\simeq$ in point 1).  That is, although the predicative translations preserve the equivalences induced by  $\eta$- and $\gamma$-conversions, they do not preserve all equivalences needed to interpret $\NI$ as a bi-cartesian closed category.

To see this, it is enough to consider the instance of $\gamma^+$ in which  $A$, $B$, $C$ and $D$ are atoms, $\mathscr{D}$ (respectively $\mathscr{D}_1$, $\mathscr{D}_2$) consists only of the assumptions of $A\vee B$ (resp.~$C$),  and $\mathscr{D}_3$ is the derivation $\AXC{$C\supset  D$}\AXC{$C$}\RL{\footnotesize ${\supset}\text{E}$}\BIC{$D$}\DP$. As the reader can easily check, the predicative translations of the left and right-hand side of this instance of $\rightsquigarrow_{\gamma^+}$  are not $\beta\eta$-equivalent, just like their RP-translations.

That the predicative translations fail to preserve the stronger permutations provides an explanation for another puzzling aspect of the approach of Ferreira  and co-authors. As observed in \cite{Ferreira2017}, Proposition~\ref{ff-preservation} fails if $\Nv$ is extended to include the standard rules of conjunction. %:
%
% {
%   \small $$\AXC{$A$}\AXC{$B$}\BIC{$A\wedge B$}\DP\qquad\qquad \AXC{$A\wedge B$}\UIC{$A$}\DP \quad \AXC{$A\wedge B$}\UIC{$B$}\DP$$
 % 
% }
%
%\noindent
In particular, the predicative translations fail to preserve all instances of the $\eta$-equation for the standard conjunction rules. %:
%
% {\small
% $$\AXC{$\D D$}\noLine\UIC{$A\wedge B$}\DP \qquad\simeq_{\eta}\qquad \AXC{$\D D$}\noLine\UIC{$A\wedge B$}\UIC{$A$}\AXC{$\D D$}\noLine\UIC{$A\wedge B$}\UIC{$B$}\BIC{$A\wedge B$}\DP$$
% }
%
Nonetheless,  when the two standard elimination rules  $\wedge$E$_1$ and  $\wedge$E$_2$ %s
are replaced by the general elimination rule  $\wedge$E$_p$  %following the same pattern :
as in $\Np$ (see above Remark~\ref{ex:fpfsum} and Remark~\ref{gencon}),
Proposition~\ref{ff-preservation} applies to the equivalence relation induced by the $\beta$-, $\eta$- and $\gamma$-conversions for ``generalized'' conjunction (see  Appendix~\ref{betaetagamma} for a formulation of these conversions in $\lambda$-calculus notation and Appendix~\ref{etagamma-appendix} for a proof of Proposition~\ref{ff-preservation} scaled to $\Np$).

The standard elimination rules can be defined using the general elimination rule as follows:

{\small$$\AXC{$A\wedge B$}\AXC{$[A]$}\RL{$\wedge$E$_p$}\BIC{$A$}\DP \qquad \AXC{$A\wedge B$}\AXC{$[B]$}\RL{$\wedge$E$_p$}\BIC{$B$}\DP$$}

\noindent
and given these definitions, the  $\eta$-equations for the standard conjunction rules become the following:

 {\small
   $$\AXC{$\D D$}\noLine\UIC{$A\wedge B$}\DP \quad\simeq%_{\eta}
   \quad %\def\defaultHypSeparation{\hskip .05in}\AXC{$\D D$}\noLine\UIC{$A\wedge B$}\AXC{$\stackrel{m}{A}$}\AXC{$\stackrel{n}{B}$}\RL{\footnotesize$\wedge$I}\BIC{$A\wedge B$}\RL{\footnotesize$\wedge$E $(m,n)$}\BIC{$A\wedge B$}\DP \quad\simeq_{\gamma^+}\quad
   \AXC{$\D D$}\noLine\UIC{$A\wedge B$}\AXC{$\stackrel{m}{A}$}\RL{\footnotesize$\wedge$E$_p$ $(m)$}\BIC{$A$}\AXC{$\D D$}\noLine\UIC{$A\wedge B$}\AXC{$\stackrel{n}{B}$}\RL{\footnotesize$\wedge$E$_{p}$ $(n)$}\BIC{$B$}\RL{\footnotesize$\wedge$I}\BIC{$A\wedge B$}\DP%\BIC{$A\wedge B$}
   $$
}

The different behavior of the standard and the general elimination rules is explained by the fact that, as in the case of disjunction, the predicative translations of two $\gamma^+$-equivalent $\Np$-derivations need not be $\beta\eta$-equivalent $\Nda$-derivations, and that the \mbox{$\gamma^+$-conversion}  for conjunction:\footnote{The same remarks of footnote~\ref{gammapiufoot} apply.}

\begingroup\makeatletter\def\f@size{10}\check@mathfonts
\begin{equation*}%\tag{$\gamma_g$}
  \label{gammaorG}
\AxiomC{$\D D$}
\noLine
\UnaryInfC{$A\wedge B$}
\AxiomC{$\stackrel{n}{[A]}\stackrel{m}{[B]}$}
\noLine
\UnaryInfC{$\D D_{1}$}
\noLine
\UnaryInfC{$C$}
\RightLabel{\footnotesize$\wedge$E $(n,m)$}
\BinaryInfC{$[C]$}
\noLine
\UnaryInfC{$\D D_2$}
\noLine
\UnaryInfC{$ D$}
\DisplayProof
\ \ \rightsquigarrow_{\gamma^+} \!\!\!
\AxiomC{$\D D$}
\noLine
\UnaryInfC{$A\wedge B$}
\AxiomC{$\stackrel{n}{[A]}\stackrel{m}{[B]}$}
\noLine
\UnaryInfC{$\D D_{1}$}
\noLine
\UnaryInfC{$[C]$}
\noLine
\UnaryInfC{$\D D_2$}
\noLine
\UnaryInfC{$ D$}
\RightLabel{\footnotesize$\wedge$E $(n,m)$}
\BinaryInfC{$D$}
\DisplayProof
\end{equation*}
\endgroup

\noindent
is essential to define the  $\eta$-conversion for the standard elimination rules using the $\eta$-conversion for the general elimination rule (see Remark~\ref{etaconj} in Appendix~\ref{betaetagamma} for a proof).%.  the $\eta$-expansion expansions

\subsection{Extending $\Nd$ and $\Nda$ with primitive disjunction}
A second disadvantage of the predicative translations arises when one considers the extension $\Ndv$ of $\Nd$ with a primitive disjunction (for an early semantic investigation  of this system see \cite{Sol77,Gab81}),  %governed by the disjunction rules of $\NI$ (we call this system $\Ndv$, see Appedix~\ref{systemsappendix}),
in which the following holds:

%\footnote{Atomic system F has been investigated as part of  hierarchies of predicative fragments of System F, and in particular Kripke-style semantics  for which some predicative fragments of F (including the atomic fragment) are sound and complete have been developed by \citep{Sol77}. Some of the results presented in  Section~\ref{invertib-section} below will rely on underivability results that we will obtain using  Sobolev's semantics in Appendix~\ref{proofprop}.}

% \begin{fact}\label{rptranssp}If $\mathscr{D}$ is a $\Ndvsp$-derivation (and thus a fortiori if $\D D$ is an $\Nv$-derivation), then. As a consequence of Corollary~\ref{??} and of Fact~{??}, also the FF-translation $\D D^{*\downarrow}$ of $\D D$ is an $\Ndsp$-derivation.
%\end{fact}
%and that
%
%These advantages are encapsulated by the following two propositions.%
%
%is made precise by the following two lemmas:
\begin{proposition}\label{rpequiv} For all $A\in \Ldv$: %$\Ndv$ and
 $$A\dashv\vdash_{\Ndv} A^*$$
  \end{proposition}

  \begin{proof} See Appendix~\ref{proofprop}.
  \end{proof}

  That is, as soon as one extends $\Nd$ with a primitive disjunction% operator
  , one can show in the extended system  $\Ndv$  that every formula is interderivable with its  RP-translation.\footnote{Actually, $A$ and its   RP-translation are not just interderivable, but can be shown to be isomorphic modulo $\simeq_{\beta\eta\gamma^+\varepsilon}$ (see for details \cite{PTCSL}).}

In contrast to what happens in $\Ndv$,
  a propositional formula $A$ and its RP-translation  $A^*$  may fail to  %even
  be interderivable in the extension of $\Nda$ with a primitive disjunction%sign
  % (and thus, a fortiori, they are not isomorphic either)
  . By inspecting one direction of the proof of Proposition~\ref{rpequiv} it is clear that  $A\vdash_{\Ndva} A^*$. However, the inspection of the other direction clearly suggests that, at least in some cases%for some formulas $A\in \Ldv$
  , in order to  establish that $A^*\vdash_{\Ndva} A$  it is essential to apply $\forall$E with a non-atomic witness.  In particular,

\begin{proposition}\label{rpnequiv} $(Y\vee Z)^*\nvdash_{\Ndva} Y\vee Z$
\end{proposition}

\begin{proof}See Appendix~\ref{proofprop}.
\end{proof}

% This stands in sharp contrast with the fact that for any $\Ldv$ formula $A$, $A$ and $A^*$ are interderivable (actually isomorphic) in the extesion of $\Ndsp$ with a primitive disjunction sign. We dub these facts the {\em (in)comparability} of $A$ and $A^*$ with respect to $\Ndsp$ ($\Nda$).

\noindent In $\Ndva$ the RP-translation of $Y\vee Z$ is thus strictly weaker than $Y\vee Z$.

\begin{remark}\label{faith} An immediate consequence of Proposition~\ref{rpnequiv} concerns the faithfulness of the  RP-translation. We recall that a translation $(\cdot)^{\natural}$ from $\mathtt{N}$ to $\mathtt{N'}$ is faithful iff for all $\Gamma, A$, if $\Gamma^{\natural}\vdash_{\mathtt{N'}}A^{\natural}$ then $\Gamma\vdash_{\mathtt{N}}A$. Although the FF-translation is a faithful translation of $\Nv$ into $\Nda$ \citep{Ferreira2014}, it is \emph{not} a faithful translation of $\Ndva$ into $\Nda$ (since---keeping in mind that the FF- and the  RP-translations of formulas coincide---$((Y\vee Z)^*)^*= (Y\vee Z)^*$ and hence obviously $((Y\vee Z)^*)^*\vdash_{\Nda} (Y\vee Z)^*$), but $(Y\vee Z)^*\nvdash_{\Ndva} Y\vee Z$ by the above proposition). 

\end{remark}

\begin{remark}
In addition to the failure of faithfulness, also the disjunction property fails for the translation of disjunction in $\Nda$ in the following sense: there exist formulas $A,B$ of $\Ndva$ such that $\vdash_{\Nda}(A\vee B)^{*}$ holds but neither of $\vdash_{\Nda}A^{*}$ and $\vdash_{\Nda}B^{*}$ holds.

Let $A= \forall X\forall Y.X\supset Y$ and $B= \exists X.A\supset X$, where $\exists X.A$ is shorthand for $\forall Y.(\forall X.A\supset Y)\supset Y$. One can easily check that $\not\vdash_{\Nda}A$ and that for all variable $Z$, $\not\vdash_{\Nda}A\supset Z$, from which one can deduce that $\not\vdash_{\Nda}B$, and notice that $A^{*}=A$ and $B^{*}=B$. On the other hand one can show that $\vdash_{\Nda}(A\vee B)^{*}$ as follows:

{\small
$$
\AXC{$\stackrel{(p_{2})}{B\supset X}$}
\AXC{$\stackrel{(n)}{\forall X. (A\supset X)\supset Y }$}
\RL{$\forall$E}
\UIC{$(A\supset X)\supset Y$}
\AXC{$\stackrel{(p_{1})}{A\supset X}$}
\RL{$\supset$E}
\BIC{$Y$}
\RL{$\supset$I{\small$(n)$}}
\UIC{$(\forall X. (A\supset X)\supset Y )\supset Y$}
\RL{$\forall$I}
\UIC{$B$}
\RL{$\supset$E}
\BIC{$X$}
\doubleLine
\RL{$\supset$I{\small$(p_{1},p_{2})$}}
\UIC{$(A\supset X)\supset(B\supset X)\supset X$}
\RL{$\forall$I}
\UIC{$(A\vee B)^{*}$}
\DP
$$
}

\end{remark}

Another way of highlighting the difference between $(A\vee B)^*$ and $(A\vee B)$ in $\Ndva$  is by observing that, whereas the disjunction elimination rule warrants that  $A\vee B\vdash_{\Ndva} (A\supset C)\supset (B\supset C)\supset C$ for all $A,B, C$ in $\Ldv$, the same does not hold if one replaces $A\vee B$ with $\forall X. (A\curlyvee B)$:
\begin{proposition}\label{noio} There are  $\Ldv$-formulas $A, B, C$ such that
\begingroup\makeatletter\def\f@size{10}\check@mathfonts
$$\forall X. (A\curlyvee B)  \nvdash_{\Ndva} (A\curlyvee B)\llbracket C/X\rrbracket$$\endgroup
\end{proposition}

\begin{proof} See Appendix~\ref{proofprop}
\end{proof}

\noindent
That is, in contrast to what happens in $\Nda$%(see Proposition~\ref{io-or})
,  the instantiation overflow property fails for $\forall X. (A\curlyvee B)$ in  $\Ndva$.

Given Proposition~\ref{rpnequiv} the analog of Proposition~\ref{rpequiv} cannot hold for $\Ndva$. One may therefore wonder in which sense, if at all, can one say that disjunction is definable in $\Ndva$.  Following Prawitz \cite[p.~58]{Prawitz1965}, we can distinguish between \emph{strong and weak definability} of a connective. In Prawitz' terminology, disjunction is  weakly definable in $\Nda$ if and only if  there is a faithful translation of $\Nv$ into $\Nda$,  which is the case, since the FF-translation is faithful (see Remark~\ref{faith} above);  on the other hand, disjunction is strongly definable in $\Nda$ iff  for all $A,B\in \Lv$ there is an $\Ld$-formula $C$, such that $C\dashv\vdash_{\Ndva} A\vee B$.\footnote{Note that strong definability implies the existence of a faithful translation from $\Ndva$  to $\Nda$.} Not only Proposition~\ref{rpnequiv} shows that disjunction is not strongly defined by its RP-translation, but there is in fact no other formula strongly defining it in $\Nda$:%   by a formula $A$ translation of $\Nv$ into $\Nda$ by a formula iff  it

%  where weak definability amounts to the existence of a faithful translation, and strong definability means that 

% thus follows that $(A\vee B)^*$ weakly defines  $A\vee B$ in $\Nda$, but it does not strongly define it  (and it is plausible to conjecture that no other formula of $\Ld$ does).   

\begin{proposition}\label{undefdis}
$\vee$ is not strongly definable in $\Nda$.
\end{proposition}

\begin{proof}
See Appendix~\ref{proofprop}
\end{proof}

\subsection{Equivalence-preservation vs reduction-preservation}
Although our discussion so far showed that the RP-translation coupled with $\varepsilon$ yields stronger results concerning preservation of equivalence, when one restricts the attention to $\beta$, $\eta$ and the weaker $\gamma$ conversions the  FF-translation yield a stronger result in that they  preserve reduction and not merely equivalence. As observed, Proposition~\ref{eps-preservation} cannot be strengthened by replacing equivalence with reduction, at least not in the case of $\eta$ (as it is clear by inspecting the proof of  Proposition \ref{eps-preservation} in Appendix~\ref{eps-preservation-lambda}, see in particular Proposition~\ref{etag-epspres}).\footnote{These remarks do not exclude the possibility of devising further reductions in $\Nd$ that can allow to ``simulate'' those steps of $\varepsilon$-conversions and of $\eta$-expansions needed to formulate a version of Proposition~\ref{eps-preservation} using reduction in place of equivalence, as done by Esp\'irito Santo and Ferreira in \cite{ESFarx}.}

At the same time, the fact that the predicative translations do preserve  $\eta$- and $\gamma$-equivalences can be explained using Proposition~$\ref{rp-ffe-relation}$ and Corollary~\ref{rp-ff-relation}, which essentially show that the translations into $\Nda$ encapsulate those bits of $\eta$-expansions and %\PP{introdurre la terminologia prima da qualche parte}
of  $\varepsilon$-conversions needed to translate $\eta$-reductions and $\gamma$-permutations of $\Nb$.

Moreover, since  $\varepsilon$-conversions allow to translate not only $\gamma$ but also $\gamma^+$-permutation (see Remark~\ref{rem:gammapiu}), all  disadvantages of the fully extensional reduction theory for disjunction (induced by $\beta$ together with $\eta^+$ or, equivalently with  $\eta$ and $\gamma^+$) are ``imported'' into $\Nd$, in particular non-confluence and non-termination (see e.g.~\cite{Lindley2007} for a discussion).

As in the case of disjunction, these problems do not exclude the possibility of considering restricted forms of $\varepsilon$-conversion (essentially, those in which only elimination rules are permuted across an application of $\forall$E), and it is not implausible to thereby obtain a well-behaving reduction theory.% (see Esp\'irito Santo and Ferreira arxiv).  

It is however remarkable that the $\beta\eta\gamma^+$-equational theory has been recently shown to be the maximum consistent equational theory in $\NI$ \cite{Sche17} and this suggests the possibility of showing  that the $\beta\eta\varepsilon$-equational theory is the maximum equational theory in the fragment $\Ndrp$ of $\Nd$ (this line of research is currently being pursued by the authors  \cite{PTCSL}).   %Maybe a pointer to question of maximality, and also a phrase on the extensional/intensional distinction (since epsilon is somehow extensional, in the end).

Summing up, we can say that $\varepsilon$ is probably more useful for investigating proof-identity  rather than proof reduction, but that its connections with category theory provide an  explanation of how and why the  predicative translations work, by embedding the  syntactic results about them  into  a wider picture.

\section{Concluding remarks}\label{conc-section}

\subsection{Summary of the results}
In this paper we have shown how the category-theory-inspired framework introduced in our previous paper can be used to %sharpen the analysis of propositional connectives in System $F_{at}$ and to
clarify the relationship between the alternative translations proposed by Ferreira, Ferreira and Esp\'irito Santo and the original Russell-Prawitz translation, and to provide semantic insights on the proof-theoretic properties of the former.

Our approach consisted in focusing on an atomizing translation from a suitable fragment of $\Nd$ (in which universal formulas correspond to the translation of polynomial connectives) into $\Nda$, and showing that such atomizations are obtained by using the  $\varepsilon$-conversions.

This made it possible to show that the predicative translations produce derivations that are equivalent to the RP-translation modulo the $\varepsilon$-conversions (hence, in semantic terms, modulo parametricity), and that their proof-theoretic properties (the preservation of $\eta$- and permutative conversions) result from the fact that $\varepsilon$-conversions express a naturality condition for the proofs denoted by $\Nd$-derivations.

%
%
%In particular, we showed how the notion of C-expansion allows one to define in a straightforward manner an alternative mapping from $\NI$- to $\Nda$-derivations, the ${\varepsilon}$-translation. %This translation not only preserves $\eta-$ and $\gamma-$equivalence without the need to extend the equational theory of System $F_{at}$ of the $\varepsilon$-equation (as Ferreira and Ferreira's FF-translation), but it also preserves $\beta$-normality (as the original Russell-Prawitz translation, and thus we believe it deserves to be called the {\em canonical} translation of $\NI$ into $\Nda$.
%
% The relationship between the three translations has also been made fully precise.
%The ${\varepsilon}$-translation of an $\NI$-derivation $\D D$  can be seen as the result of applying some steps of $\beta$-reduction to the FF-translation of $\D D$, and by further applying some steps of $\beta$-reduction to it one obtains the refinement $\sharp$ of the FF-translation. On the other hand, the ${\varepsilon}$-translation can also be seen as the result of repeteadly applying several $\eta$-expansions followed by  one application of the $\varepsilon$-conversion to the RP-translation of $\D D$.

Finally, we highlighted the trade-off between the approach based on atomic polymorphism and the more semantics-inspired approach based on the $\varepsilon$-conversions, and we argued that the former is better adapted to  investigate proof reduction, while the second is more adapted for investigating identity of proofs.

\subsection{Related and further work}
Although a generalization of our $\varepsilon$-conversions (expressing a dinaturality condition, rather than just naturality) can be formulated for any formula $\forall X.A$ of $\mathcal{L}^2$, the tight connection between the $\varepsilon$-conversions and the phenomenon of instantiation overflow underlying the present paper seems to be limited to the class of universal polynomial formulas. %On the one hand, $\varepsilon$ can be generalized to ``dinatural'' equations applicable to arbitrary (and not just nested sp-$X$) formulas, that not necessarily satisfy investigation overflow.%
%On the other hand, 
In fact, in \cite{StudiaLogica} the $\varepsilon$-conversions were defined for a broader class of formulas, that of nested p-$X$ formulas, all enjoying the instantiation overflow property. However, it does not seem possible to define a uniform atomization procedure in the style of our $\varepsilon$-atomization  for the $\Nd$-derivations in which the premises of all applications of $\forall$E are nested p-$X$ (rather than universal polynomial) formulas. For an exact characterization of the class of $\Ld$-formulas enjoying  instantiation overflow, see \cite{Ferreira2016} and \cite{Pistone2018}.% it is shown that the. %In the second paper a class of formulas called \emph{generalized Russell-Prawitz formulas} (GRP) is introduced and
As to further directions of investigations, we observe that, %as we already briefly mentioned,
in the extension of System $F$ with primitive propositional connectives, not only every proposition and its Russell-Prawitz translation are interderivable,  but they are actually \emph{isomorphic} (modulo the equivalence relation induced by $\beta$-, $\eta$-, $\gamma^+$-, and $\varepsilon$-equations). From  a categorical  perspective, such isomorphisms belong to a more general class of isomorphisms induced by a proof-theoretic formulation of the \emph{Yoneda lemma}. This and further categorical  aspects underlying the present work are the object of current ongoing  work by the authors \cite{PTCSL}.

%\bibliographystyle{plainnat}
%\bibliography{BIBA}

%\input{completeness}

\appendix\section{The system $\Ndp$}\label{systemsappendix}

We indicate by $\C{TV}$ a countably infinite set of \emph{term variables} $x,y,z,\dots$.
The terms of the system $\Ndp$ are given by the following grammar (where, given a list $\C A=\langle a_1,\ldots a_n\rangle$, we indicate  an $\C A$-indexed list of terms $t_{a_1} \ldots t_{a_n}$ as $\lust{t_a}{a}{\C A}$ and, below,  simply with $\lost{t_a}{a}{\C A}$ whenever $\C A$ is clear from the context, see Remark~\ref{index-sets}):%

{  \small\centering$
\bb{t},\bb u := \bb x \in \C{TV}\mid\bb{ \lambda x.t} \mid \bb{tu}\mid \bb{\Lambda X.t }\mid  \bb{tB%^{\forall X.A}
}
%\mid\bb{( t,u)} \mid \bb{\pi_{i}^{A\wedge B}t} \mid \bb{\star}
\mid \bb{\iota_{\dagger }^{k}\lust{t_j}{j}{{g}^{-1}(k)}}  \mid \bb{\delta_{\dagger}(t, \lust{\lust{x_j}{j}{g^{-1}(k)}.s_{k}}{k}{K})} %\mid \bb{\xi^{C} t}
$

}

That is, for every connective $\dagger$ generated by $\C I \stackrel{f}{\leftarrow}\C J \stackrel{g}{\to}\C K$, we have a family of generalized injections $\iota_{\dagger}^{k}$ indexed by elements of $\C K$, and a generalized ``case'' constructor $\delta_{\dagger}$ with $|\C K|+1 $ arguments. %For readability we will  abbreviate injections and cases with $\bb{\iota_{\dagger}^{k}(t_j)}$ and $\bb{\delta^{\dagger(A_i)\Rightarrow C}(t, (x_j.s_{k}))}$ respectively and $\dagger^{(f,g)} (A_i)_{i\in\C I}$ with $\dagger (A_i)$. 

Let $\C A=\langle a_{1},\dots, a_{k}\rangle$ be a finite list. 
We will use the following abbreviated notation to indicate $\C A$-indexed sequences of $\lambda$-abstractions and applications:
if $\langle x_{a}\rangle_{a\in \C A}$ is an $\C A$-indexed list of variables, for any term $t$, we let
$\lambda \langle x_{a}\rangle_{a\in \C A}.t$ (often abbreviated as $\lambda \langle x_{a}\rangle.t$) indicate the term
$\lambda x_{a_{1}}\dots \lambda x_{a_{k}}.t$.
Similarly, if 
$\langle t_{a}\rangle$ is an $\C A$-indexed list of terms, for any term $u$, we let
$u\langle t_{a}\rangle_{a\in \C A}$ (often abbreviated as $u\langle t_{a}\rangle$) indicate the term $ut_{a_{1}}\dots t_{a_{k}}$.

As usual, by a \emph{typing context} we indicate a finite set of type declarations $x:A$ where all declared variables are distinct. We indicate typing contexts with $\Gamma,\Delta,\Sigma$.

%The type information associated to second-order instantiations and to the $\delta$-eliminators are needed to make it possible to define the different embeddings directly on the terms, rather than on  typing derivations. 
The typing rules for $\Ndp$ are the following:% given in Table~\ref{typr}. 

%\begin{table}[h]
%\rule{\textwidth}{.5pt}

\medskip
{\centering \small$\begin{matrix}
\qquad
\AXC{}\RL{Ax}
\UIC{$\Gamma,\bb x:A\vdash\bb x:A$}
\DP
\qquad
\AXC{$\Gamma,\bb x:A\vdash\bb t:B$}\RL{$\supset$I}
\UIC{$\Gamma\vdash \bb{\lambda x.t}:A\imp B$}
\DP
\qquad
\AXC{$\Gamma\vdash \bb t:A\imp B$}
\AXC{$\Gamma\vdash \bb u:A$}\RL{$\supset$E}
\BIC{$\Gamma\vdash \bb{tu}:B$}
\DP\\ \\
\AXC{$\Gamma\vdash\bb t:A$}
\RL{$\forall$I {\small(Proviso: $X\notin FV(\Gamma)$)}}
\UIC{$\Gamma\vdash \bb{\Lambda X.t}:\forall X.A$}
\DP
\qquad
\AXC{$\Gamma\vdash \bb t:\forall X.A$}\RL{$\forall$E}
\UIC{$\Gamma\vdash \bb{tB%^{\forall X.A}
  }  :A\llbracket B/X\rrbracket$}
\DP\\ \\
%\end{matrix}$$
%
%
%$$
%\begin{matrix}
% \AXC{$\Gamma\vdash\bb t:A$}
% \AXC{$\Gamma\vdash\bb u:B$}
% \BIC{$\Gamma\vdash\bb{ ( t,u)}:A\wedge B$}
% \DP \qquad
% \AXC{$\Gamma\vdash\bb t:A_{1}\wedge A_{2}$}
% \UIC{$\Gamma\vdash  \bb{\pi_{i}^{A_{1}\wedge A_2}t}:A_{i}$}
% \DP
% \qquad
% \AXC{}
% \UIC{$\Gamma\vdash  \bb{\star}:\top$}
% \DP
% \\
%  \qquad \\
\left\langle\AXC{$\lust{\Gamma\vdash\bb t_j:A_{f(j)}}{j}{g^{-1}(k)}$}
 \RL{$\dagger$I$_k$}\UIC{$\Gamma\vdash \bb{\iota_{\dagger}^{k}\lost{t_j}{
       }{}}:\dagger \lost{A_i}{}{}$}
\DP\right\rangle_{k\in \C K}
\qquad
\AXC{$\Gamma\vdash \bb t:\dagger\lost{A_{i}}{}{}$}
\AXC{$\left\langle \Gamma,\lust{\bb{y_j}:A_{f(j)}}{j}{g^{-1}(k)} \vdash\bb{ s_{k}}:C\right\rangle_{k\in\C K}$}\RL{$\dagger$E}
\BIC{$\Gamma\vdash \bb{\delta_{\dagger}(t,\lost{\lost{y_j}{}{}.s_{k}}{}{})}: C$}
\DP
% \qquad
% \AXC{$\Gamma\vdash \bb t:\bot$}
% \UIC{$\Gamma \vdash\bb{\xi^{C} t}: C$}
% \DP
%\\ \qquad \\
% \AXC{$\Gamma\vdash t:A[\mu X.A/X]$}
% \UIC{$\Gamma\vdash \inn_{A}t: \mu X.A$}
% \DP
% \qquad
% \AXC{$\Gamma\vdash t: A[B/X]\To B$}
% %\AXC{$\Gamma\vdash u:\mu  X.A$}
% \UIC{$\Gamma\vdash \ff_{A}(t):  \mu X.A\To B$}
% \DP
% \\ \qquad \\
% \AXC{$\Gamma\vdash t:\nu X.A$}
% \UIC{$\Gamma\vdash \outt_{A}t: A[\nu  X.A/X]$}
% \DP
% \qquad
% \AXC{$\Gamma\vdash t: B\To A[B/X]$}
% %\AXC{$\Gamma\vdash u:B$}
% \UIC{$\Gamma\vdash \uu_{A}(t): B\To \nu X.A$}
% \DP
\end{matrix}
$

}

% \rule{\textwidth}{.5pt}
% \caption{Typing rules}\label{typr}
% \end{table}

\medskip
We write $\Gamma\vdash_{\Ndp} t:A$ iff there is a derivation of $\Gamma\vdash t:A$ using the above rules. Similarly for $\Gamma\vdash_{\Nd} t:A$ and all other systems introduced in Sections~\ref{lansys} and \ref{rp-trans}.

A \emph{term context} (or simply context when no ambiguity with typing contexts arise, indicated as $\TT T, \TT U, \dots$) is a term with a distinguished variable that we indicate as the ``hole'' $[\ ]$. Observe that this is a  generalization of the standard notion of context, since we allow the hole $[\ ]$ to  occur zero, one or (finitely) many times in a context $ \TT T$.
For any context $\rr{\TT T}$, we let 
$\TT T[t]$ (resp. $\TT T[\TT U]$) indicate usual---i.e.~non variable-capturing---substitution (short n.v.c.-substitution) of the term $t$ (resp.~context $\rr{\TT U}$) for $[\ ]$ in~$\rr{\TT T}$, and $\rr{\TT T}\{t\}$ (respectively $\rr{\TT T}\{\rr{\TT U}\}$) indicate \emph{variable-capturing} substitution (short v.c.-substitution) of the term $t$ (resp.~context $\rr{\TT U}$) for $[\ ]$ in~$\rr{\TT T}$. 
Moreover, we let $\TT T: A\vdash^{\Gamma} B$ be a shorthand for $\Gamma, x:A\vdash \TT T[x]: B$.
Note that if  $\TT T: A\vdash^{\Gamma} B$  and $\Gamma\vdash t:A$, not only $\Gamma\vdash \TT T[t]:B$ but also 
$\Gamma\vdash \TT T\{t\}:B$ hold.

We will use the following special families of contexts:
\begin{itemize}
\item the \emph{principal contexts} (indicated as $\TT C, \TT D, \dots$)  are defined by the grammar below:
$$
\TT C:= [\ ]\mid \TT Cu \mid \TT C B \mid \bb{\delta_{\dagger}(\rr{\TT C},\lost{\lost{y_j}{}{}.s_{k}}{}{})}\mid 
 \lambda x.\TT C\mid \Lambda Y.\TT C \mid \iota_{\dagger}^{k}\langle t_1,\ldots, t_{l-1}, \TT C,t_{l+1},\ldots t_{|g^-1(k)|}\rangle 
$$
\item the \emph{elimination contexts} are defined by dropping the cases $\lambda x.\TT C$, $\Lambda Y.\TT C$ and\linebreak $\iota_{\dagger}^{k}\langle t_1,\ldots, t_{l-1}, \TT C,t_{l+1},\ldots t_{|g^-1(k)|}\rangle$ from the grammar above;

\item 
the \emph{introduction contexts} are defined by dropping the cases $\TT Cu$, $ \TT C B$ and $ \bb{\delta_{\dagger}(\rr{\TT E}\lost{\lost{y_j}{}{}.s_{k}}{}{})}$ from the grammar above.
\end{itemize}

Observe that all principal contexts are contexts in the standard sense (i.e.~they all contain exactly one occurrence of the hole). Note that if $\TT C, \TT D$ are  principal contexts, then $\TT C\{\TT D\}$ and $\TT C[\TT D]$ are principal contexts  as well. 

It is easily checked that if $\TT E$ is an elimination context, then for all term $t$, $\TT E\ho t$ cannot capture variables of $t$, whence $\TT E\ho t= \TT E[t]$.

%
%
%%  a context generated by the grammar below:
%%$$\rr{\TT E} :=  [\ ] \ |\ \rr{\TT E}t \ |\ \rr{\TT E}B%^{\forall X.A}
%%\ | \ \bb{\delta_{\dagger}(\rr{\TT E},\lost{\lost{y_j}{}{}.s_{k}}{}{})} $$
%An \emph{atomic elimination context} is a context generated by the grammar below:
%$$\rr{\TT E} :=  [\ ] \ |\ \rr{\TT E}t \ |\ \rr{\TT E}X \ (X\in \C V)%^{\forall X.A}
%\ | \ \bb{\delta_{\dagger}(\rr{\TT E},\lost{\lost{y_j}{}{}.s_{k}}{}{})} $$
%An \emph{introduction context} is a context generated by the grammar below:
%$$\rr{\TT I} :=  [\ ] \ |\ \lambda x. \rr{\TT I} \ |\ \Lambda X. \rr{\TT I}%^{\forall X.A}
%\ | \  \iota_{\dagger}^{k}\langle \TT I_{j}\rangle 
%$$
%
%We indicate contexts (i.e.~terms in which one occurrence of a variable is replaced by a hole $[\ ]$) 

\section{The standard equivalence on derivations}\label{betaetagamma}

%Let $\TT S$ be any of the systems introduced. 
%A \emph{well-typed equation over $\TT S$} is an expression of the form $\Gamma\vdash_{\TT S} t \simeq u: A$, such that $\Gamma\vdash_{\TT S} t:A$ and $\Gamma\vdash_{\TT S} u:A$ both hold. 
%A \emph{theory over $\TT S$} is a set $S$ of well-typed equations closed with respect to reflexivity, symmetry, transitivity and such that 
%for all context $\rr{\TT C}: (\Gamma\vdash_{\TT S} A)\To (\Delta\vdash_{\TT S} B)$, if 
%$\Gamma\vdash_{\TT S} t\simeq u: A \in S$, then $\Gamma\vdash_{\TT S} \rr{\TT C}[t]\simeq \rr{\TT C}[u]: B\in S$.
%
%

The rules of equivalence are the following:

% We consider the following classes of well-typed equations:% on terms in $\Ndp$ are given in Table~\ref{equiv-table}.

 \medskip
% \begin{table}[!h!]

%   \rule{\textwidth}{.5pt}
  
%\fbox{
%\begin{subfigure}{\textwidth}
%  \adjustbox{scale=0.8, center}{
{\small\centering  $
\begin{matrix}
  \AXC{$\Gamma, x:A\vdash t:B$}\AXC{$\Gamma\vdash s:A$}\RL{$\imp\beta$}\BIC{$\Gamma\vdash \bb{(\lambda x.t)s \ }\simeq%_{\beta}
    \bb{ \ t\llbracket s/x\rrbracket}:B$}\DP 
\qquad \AXC{$\Gamma\vdash \bb t:A\imp B$}\RL{$\imp\eta$ {\small ($x\notin FV{(t)}$)}}
\UIC{$\Gamma\vdash \bb{t }\simeq%_{\eta}
  \bb{ \lambda x.tx}:A\imp B$}\DP\\
\ \\
\AXC{$\Gamma \vdash t:A$}\RL{$\forall\beta$ {\small ($X\notin FV{(\Gamma)}$)} }\UIC{$\Gamma\vdash\bb{(\Lambda X.t)B \ } \simeq%_{\beta }
  \bb{ \ t\llbracket B/X\rrbracket }:A\llbracket B/X\rrbracket$}\DP
 \qquad
 \AXC{$\Gamma\vdash \bb t:\forall X.A$}\RL{$\forall\eta$ {\small ($X\notin FV{(t)}$)}}
 \UIC{$\Gamma\vdash \bb{t} \ \simeq%_{\eta}
   \ \bb{\Lambda X.tX%^{\forall X.A}
   }:\forall X.A$}
\DP\\
\\
%\end{matrix}
%$}
% \caption{$\beta$ and $\eta$-rules for $\Nd$.}
% \label{eqscherer0}
% \end{subfigure}
% }
%%
%
%
% \
%
%
%
% \
%
%
%
%
% \fbox{
% \begin{subfigure}{\textwidth}
% \adjustbox{scale=0.8, center}{$
% \begin{matrix}
% \Big (\bb{\pi_{i}^{A_1\wedge A_{2}}( t_{1},t_{2}) \ }\simeq_{\beta}\bb{ \ t_{i} } \Big)_{\scalerel*{i=1,2}{X}}\qquad
\left\langle \AXC{$%\lost{
\left\langle\Gamma\vdash\bb t_j:A_{f(j)}\right\rangle_{j\in g^{-1}(h)}$}
%}
\AXC{$\left\langle \Gamma,\lust{\bb{y_j}:A_{f(j)}}{j}{g^{-1}(k)} \vdash\bb{ s_{k}}:C\right\rangle_{k\in \C K}
    $}
    \RL{$\dagger\beta_h$}\BIC{$\Gamma\vdash \bb{\delta_{\dagger}(\iota_{\dagger}^{h}\lust{t_j}{j}{g^{-1}(h)}, \langle{\lust{y_j}{j}{g^{-1}(k)}.s_{k}}\rangle_{k\in \C K}) \ }\simeq%_{\beta}
    \bb{ \ s_{h}\llbracket\lost{t_j}{}{}/\lost{y_j}{}{}\rrbracket}:C $}\DP\right\rangle_{h\in \C K}
    %\right\rangle%_{\scalerel*{k\in \C K}{X}}
\\
%\qquad%\\
%\
%\\
% \AXC{$\Gamma\vdash \bb t: A\wedge B$}
% \UIC{$\Gamma\vdash \bb{
% t }\simeq_{\eta}\bb{  (\pi_{1}^{A\wedge B}t, \pi_{2}^{A\wedge B}t) }: A\wedge B$}
% \DP
% \qquad\qquad
% \qquad\AXC{$\Gamma\vdash \bb t:\top$}
% \UIC{$\Gamma\vdash \bb{t}\simeq_{\eta}\bb{ \star}: \top$}
% \DP 
%\\

\ \\
%\AXC{$\Gamma\vdash t:\dagger\lost{A_i}{}{}$}\AXC{$\rr{\TT C}:  (\Delta\vdash\dagger\lost{A_i}{}{}) \Rightarrow (\Sigma \vdash C)$}\RL{$\dagger\eta^{+}$}\BIC{$\Gamma,\Sigma\vdash \bb{\rr{\TT C}[t]}  \simeq%_{\eta^{+}}
%  \bb{\delta_{\dagger}(t, \langle{\lost{y_j}{}{}.\rr{\TT C}[\iota_{\dagger}^{k}\lost{y_j}{}{}]}\rangle_{k\in \C K}) }:C$}
%\DP
\AXC{$\Gamma\vdash t:\dagger\lost{A_i}{}{}$}\AXC{$\rr{\TT U}:  \dagger\lost{A_i}{}{} \vdash^{\Gamma} C$}\RL{$\dagger\eta^{+}$}\BIC{$\Gamma\vdash   \bb{\delta_{\dagger}(t, \langle{\lost{y_j}{}{}.\rr{\TT U}[\iota_{\dagger}^{k}\lost{y_j}{}{}]}\rangle_{k\in \C K}) }  \simeq%_{\eta^{+}}
\bb{\rr{\TT U}[t]}:C$}
\DP
% \qquad\qquad
% \AXC{$\Gamma\vdash \bb u:\bot$}
% \AXC{$\rr{\TT C}: \bot \vdash^{\Gamma}C$}
% \BIC{$\Gamma\vdash\bb{ \rr{\TT C}[u] }\simeq_{\eta} \bb{\xi^{C} u}: C$}
% \DP 
\end{matrix}
$

}

%\caption{$\beta$ and $\eta$-rules for $\wedge,\vee, \bot,\top$.}
%\label{eqscherer}
%\end{subfigure}
%}

% \rule{\textwidth}{.5pt}
% %\

% %\

% \caption{The standard $\beta$- and $\eta$- rules}\label{equiv-table}
% \end{table}

%{\color{magenta}A bit unclear here: when I say $t\simeq_{\eta} u$ does this involve $\eta_{g}$ or $\eta$? In Prop. 10 it seems $\eta_{g}$, in Prop. 17 it seems $\eta$, right?}

 \medskip\noindent
together with reflexivity, transitivity and symmetry and congruence rules.

\begin{remark}\label{gammag-lambda}
As in the case of disjunction (see, e.g.~\cite{See79} and \cite{Lindley2007}) the ``generalized'' rule $\eta^+$ for $\dagger$-connectives can be decomposed into a simple form of $\eta$-rule and a ``generalized'' permutation rule:

{\centering\small$\AXC{$\Gamma\vdash t:\dagger\lost{A_i}{}{}$}%\AXC{$\rr{\TT C}:  \dagger(A_i) \vdash^{\Delta} C$}
 \RL{$\dagger\eta$}\UIC{$\Gamma\vdash 
\bb{\delta_{\dagger}(t, \langle{\lost{y_j}{}{}.\iota_{\dagger}^{k}\lost{y_j}{}{} }\rangle_{k\in \C K}) }  \simeq \bb{t} :\dagger\lost{A_i}{}{}$}\DP
  $

\medskip
  $
  \AXC{$\Gamma\vdash \bb t:\dagger(A_{i})$}
\AXC{$\left\langle\Gamma,\langle\bb{y_j}:A_{f(j)}\rangle \vdash\bb{ s_{k}}:F\right\rangle_{k\in\C K}$}
\AXC{$\TT U: F\vdash^{\Gamma}G$}
%\AXC{$\rr{\TT C}: (\Delta\vdash F) \Rightarrow (\Gamma\vdash G)$}
\RL{$\dagger\gamma^{+}$}
\TIC{$\Gamma\vdash\rr{\TT U}[\bb{\delta_{\dagger}(t,\langle{\lost{y_j}{}{}.s_{k}}\rangle_{k\in \C K})}]\simeq%_{\gamma^{+}} 
\bb{\delta_{\dagger}(t,\langle{\lost{y_j}{}{}.\rr{\TT U}[s_{k}]}\rangle_{k\in \C K})}: G$}\DP
$

}

\medskip
The rule $\dagger\gamma$ expressing the standard permuting conversions used in establishing the subformula of normal derivations in $\NI$ is the special case of  $\dagger{\gamma^{+}}$  in which the context $\rr{\TT U}$ is an {elimination context}.

\end{remark}

We write 

{\small
$$\begin{array}{l|c|l}
    \Gamma\vdash_{\Ndp} t\simeq_{\beta} s:A &    & \mathord{\supset}\beta, \forall\beta,\dagger\beta_k\\
    \Gamma\vdash_{\Ndp} t\simeq_{\eta^{+}} s:A & \text{ iff there is a derivation of }  & \mathord{\supset}\eta, \forall\eta,\dagger\eta^+\\
\Gamma\vdash_{\Ndp} t\simeq_{\eta} s:A & \text{$\Gamma\vdash t\simeq s$ using all the typing } & \mathord{\supset}\eta, \forall\eta,\dagger\eta \\
\Gamma\vdash_{\Ndp} t\simeq_{\gamma^{+}} s:A & \text{ rules of $\Ndp$ together with} & \dagger\gamma^+ \\
\Gamma\vdash_{\Ndp} t\simeq_{\gamma} s:A & & \dagger\gamma\\
  \end{array}
  $$
}

\noindent
as well as reflexivity, transitivity, symmetry and the congruence rules. With  e.g.~$\Gamma\vdash_{\Ndp} t\simeq_{\beta\eta} s:A$ we indicate that there is a derivation using the rules underlying both $\Gamma\vdash_{\Ndp} t\simeq_{\beta} s:A$ and $\Gamma\vdash_{\Ndp} t\simeq_{\eta} s:A$, and similarly for other ``combined'' labels.

Similar notation will be used for sub-systems of $\Ndp$. If the derivations do not use  symmetry we write $\rightsquigarrow$ in place of $\simeq$.

\begin{remark}\label{rem:etacapturing}
The rules $\dagger \eta^{+}$ and $\dagger\gamma^{+}$ above are formulated using n.v.c-substitution of a term in the context $\TT U$. Observe that the variants $\dagger \eta^{+}_{\text{v.c.}}$ and $\dagger\gamma^{+}_{\text{v.c.}}$ of these rules with v.c.-substitution in place of n.v.c-substutituion are derivable from $\dagger \eta^{+}$ and $\dagger\gamma^{+}$ (\emph{together with the other conversion rules and symmetry}). The proof of the derivability of the v.c.-variants is by induction on the context $\TT U$, and the only critical cases are those in which $\TT U=\lambda x. \TT U'$, $\TT U=\Lambda X. \TT U'$ or ${\TT U}$ is a $\delta_{\dagger}$-term in which one or more of its immediate subterms in $\lost{s_k}{}{}$ are a context ${\TT U'}$. We give the details of the inductive case for $\dagger \gamma^{+}_{\text{v.c.}}$ in which 
% If  of $\dagger \eta^{+}_{\text{v.c.}}$  %by informally showing how to construct a derivation of % For example, if 
$\TT U=\lambda x. \TT U'$ %, where $\TT U'$ does not capture any variable, then from $\dagger \eta^{+}$ we deduce 
by informally showing how to construct a derivation of $\Gamma\vdash\rr{\TT U}\{\bb{\delta_{\dagger}(t,\langle{\lost{y_j}{}{}.s_{k}}\rangle%_{k\in \C K}
)}\}\simeq%_{\gamma^{+}} 
\bb{\delta_{\dagger}(t,\langle{\lost{y_j}{}{}.\rr{\TT U}\{s_{k}\}}\rangle%_{k\in \C K})
}: G$  as follows:

{\small$$\begin{array}{c@{}l@{}c@{}l@{}l@{}}
\lambda x.\rr{\TT U'}\{\bb{\delta_{\dagger}(t,\langle{\lost{y_j}{}{}.s_{k}}\rangle)}\} & \ \ \stackrel{\text{I.H.}}{\simeq}\ \   &  \lambda x.\delta_{\dagger}(t, \langle\lost{y_j}{}{}.\rr{\TT U'}\{s_k\}\rangle) & \ \ \simeq_{\supset\beta}\ \ &\\
 & \ \ \simeq_{\supset\beta}\ \  &  \lambda x.\bb{\delta_{\dagger}(t, \langle{\lost{y_j}{}{}.\big(\lambda x.\rr{\TT U'}\ho{s_k} \big )x}\rangle)} &    \ \ \simeq_{\dagger\gamma}\ \  &   \\
     & \ \ \simeq_{\dagger\gamma}\ \   & \lambda x.\bb{\delta_{\dagger}(t, \langle{\lost{y_j}{}{}.\lambda x.\rr{\TT U'}\ho{s_k}}\rangle) }x & \ \rightsquigarrow_{\supset\eta}\ \ &   \bb{\delta_{\dagger}(t, \langle{\lost{y_j}{}{}.\lambda x.\rr{\TT U'}\{s_k\}}\rangle) }  \end{array}$$}

\noindent The other critical  cases for $\dagger\gamma^+$ and those for $\dagger\eta^+$ are similar.
\end{remark}

\begin{remark}\label{etaconj}
Assuming  the standard elimination rules for $\wedge$ are defined as $\pi_1(t) := \delta_{\!\mathord{\wedge}\!}(t,y_1y_2 .y_1)$ and $\pi_2(t) := \delta_{\!\mathord{\wedge}\!}(t, y_1y_2.y_2)$ (and taking  $\iota_{\!\mathord{\wedge}\!}(t_1,t_2)$ as prefix notation for  the standard pairing), we can derive the standard $\eta$-rule for conjunction $ \iota_{\!\wedge\!} (\pi_1 (t), \pi_2 (t))  \simeq t$ as follows:

\vspace{-1ex}
{\small
$$\iota_{\!\wedge\!}(\delta_{\!\mathord{\wedge}\!}(t,y_1y_2.y_1), \delta_{\!\mathord{\wedge}\!}(t,y_3y_4.y_4) )  \ \ \rightsquigarrow_{\dagger\gamma^+}\ \ \qquad\qquad\qquad\qquad\qquad\qquad\qquad\qquad\qquad\qquad\qquad\qquad $$
$$\begin{array}{c@{}l@{}c@{}l@{}}
  & \ \ \rightsquigarrow_{\dagger\gamma^+}\ \   &  \delta_{\!\mathord{\wedge}\!}(t,y_1y_2. \iota_{\!\wedge\!}\big(y_1,  \delta_{\!\mathord{\wedge}\!}(t,y_3y_4.y_4)\big))   & \ \ \rightsquigarrow_{\dagger\gamma^+}\\
 & \ \ \rightsquigarrow_{\dagger\gamma^+}\ \  & \delta_{\!\mathord{\wedge}\!}(t,y_1y_2.\delta_{\!\mathord{\wedge}\!}(t,y_3y_4. \iota_{\!\wedge\!}\big(y_1,y_4\big)))  &    \ \ \simeq_{\dagger\eta}\\
     & \ \ \simeq_{\dagger\eta}\ \   & \delta_{\!\mathord{\wedge}\!}(\Big(\delta_{\!\wedge\!}(t,z_1z_2.\iota_{\!\wedge\!}(z_1,z_2)) \Big),y_1y_2.\delta_{\!\mathord{\wedge}\!}(\Big(\delta_{\!\wedge\!}(t,z_1z_2.\iota_{\!\wedge\!}(z_1,z_2)) \Big),y_3y_4. \iota_{\!\wedge\!}\big(y_1,y_4\big)))  & \ \rightsquigarrow_{\dagger\gamma^+}\\
&\ \ \rightsquigarrow_{\dagger\gamma^+} \ \ & \delta_{\!\wedge\!}(t,z_1z_2.\delta_{\!\mathord{\wedge}\!}(\Big( \iota_{\!\wedge\!}(z_1,z_2)  \Big),y_1y_2.\delta_{\!\mathord{\wedge}\!}(\Big(\iota_{\!\wedge\!}(z_1,z_2) \Big),y_3y_4. \iota_{\!\wedge\!}\big(y_1,y_4\big)))) & \ \ \rightsquigarrow_{\dagger\beta}\\
& \ \ \rightsquigarrow_{\dagger\beta}\ \  & \delta_{\!\wedge\!}(t,z_1z_2.\delta_{\!\mathord{\wedge}\!}(\Big( \iota_{\!\wedge\!}(z_1,z_2)  \Big), y_1y_2.\iota_{\!\wedge\!}(y_1,z_2)))  & \ \ \rightsquigarrow_{\dagger\beta}\\  
& \ \ \rightsquigarrow_{\dagger\beta}\ \  &  \delta_{\!\wedge\!}(t,z_1z_2.\iota_{\!\wedge\!}(z_1,z_2)  ) & \ \ \rightsquigarrow_{\dagger\eta}\end{array}$$
$$\qquad\qquad\qquad\qquad\qquad\qquad\qquad\qquad\qquad\qquad\qquad\qquad\qquad\qquad\qquad\qquad\qquad\qquad \ \ \rightsquigarrow_{\dagger\eta} \ \ t$$
}

\noindent
where the  three applications of $\dagger\gamma^+$ can be described as:
\begin{itemize}
\item ${\TT U}[\delta_{\!\mathord{\wedge}\!}(t,y_1y_2.y_1)]    \rightsquigarrow_{\gamma^+}\   \delta_{\!\mathord{\wedge}\!}(t,y_1y_2.{\TT U}[y_1])$ with ${\TT U}:= \iota_{\!\wedge\!}([\ ],   \delta_{\!\mathord{\wedge}\!}(t,y_3y_4.y_4) )$; 
\item ${\TT U}[\delta_{\!\mathord{\wedge}\!}(t,y_3y_4.y_4)]    \rightsquigarrow_{\gamma^+}\   \delta_{\!\mathord{\wedge}\!}(t,y_3y_4.{\TT U}[y_4])$ with ${\TT U}:= \iota_{\!\wedge\!}(y_1, [\ ])$; 
 \item ${\TT U}[\delta_{\!\mathord{\wedge}\!}(t,z_1z_2.\iota_{\!\wedge\!}(z_1,z_2))]    \rightsquigarrow_{\gamma^+}\   \delta_{\!\mathord{\wedge}\!}(t,z_1z_2.{\TT U}[\iota_{\!\wedge\!}(z_1,z_2)])$ with ${\TT U}:= \delta_{\!\mathord{\wedge}\!}([\ ],y_1y_2.\delta_{\!\mathord{\wedge}\!}([\ ],y_3y_4. \iota_{\!\wedge\!}\big(y_1,y_4\big)))$.

\end{itemize}
\end{remark}

%
%\begin{remark}
%The generalized $\eta$-rule for disjunction is sometimes presented in the more restricted form
%\begin{center}
%\AXC{$\Gamma\vdash t:\dagger\lost{A_i}{}{}$}\AXC{$\rr{\TT C}:  \dagger\lost{A_i}{}{} \vdash^{\Delta} C$}\RL{$\dagger\eta^{!}$}\BIC{$\Gamma,\Delta\vdash \bb{\rr{\TT C}[t]}  \simeq%_{\eta^{+}}
%  \bb{\delta_{\dagger}(t, \langle{\lost{y_j}{}{}.\rr{\TT C}[\iota_{\dagger}^{k}\lost{y_j}{}{}]}\rangle_{k\in \C K}) }:C$}
%\DP
%\end{center}
%
%The equivalence relation arising from $\dagger\eta^{!}$ is in fact equivalent to the one arising from $\dagger\eta^{+}$. For example, for $\lambda x.\rr{\TT C}: (x:D\vdash \dagger\lost{A_i}{}{}) \Rightarrow (\Delta \vdash  C)$, we can deduce
%\begin{center}
%$
%\delta_{\dagger}(t, \langle{\lost{y_j}{}{}.\lambda x.\rr{\TT C}[\iota_{\dagger}^{k}\lost{y_j}{}{}]}\rangle_{k\in \C K}) 
%\stackrel{\supset\eta}{\simeq}
%\lambda x.\Big ( \delta_{\dagger}(t, \langle{\lost{y_j}{}{}.\lambda x.\rr{\TT C}[\iota_{\dagger}^{k}\lost{y_j}{}{}]}\rangle_{k\in \C K}) 
%\Big ) x
%$
%\\
%$
%\stackrel{\gamma}{\simeq}
%\lambda  x.\delta_{\dagger}(t, \langle{\lost{y_j}{}{}.\Big (\lambda x.\rr{\TT C}[\iota_{\dagger}^{k}\lost{y_j}{}{}]}\Big )x\rangle_{k\in \C K} 
%\stackrel{\supset\beta}{\simeq}
%\lambda x.\delta_{\dagger}(t, \langle{\lost{y_j}{}{}.\rr{\TT C}[\iota_{\dagger}^{k}\lost{y_j}{}{}]}\rangle_{k\in \C K} 
%\stackrel{\eta^{!}}{\simeq}
%\lambda x.\rr{\TT C}[t]
%$
%\end{center} 
%where the rule $\gamma$ can be easily deduced from $\eta^{!}$.
%
%
%
%\end{remark}

\section{Weak expansion and $A$-expansion}\label{cexp-appendix}

%  Let $A\in \mathcal{L}^2$, i.e. for some $n$,  
%$A= \forall \vec{Y_1}. A_1 \imp \forall \vec{Y_2}. A_2 \imp \ldots \forall \vec{Y_n}. A_n \imp \forall \vec{Y_{n+1}}.Z $
%  where for all $i\leq n+1$,  $\forall \vec{Y_i} =\forall Y_{i1} \ldots \forall Y_{ik_i}$ for some $k_i\geq 1$ or it is empty.

%By induction on a formula $A\in \mathcal{L}^2$ we define the classes of elimination and introduction contexts for $A$. 

%  Fixed  $m\geq n$ variables $\vec{z}=z_1, \ldots, z_m$, let $\Sigma_A  = \{z_1:A_1,\ldots, z_n:A_n\}$.
%  We define two contexts $\bb{\rr{\TT{El}}_{A} }:(A\vdash^{\Sigma_A } Z)$ and $\bb{\rr{\TT{In}}_{A} }:(\Gamma,\Sigma_{A} \vdash Z \Rightarrow \Gamma\vdash A )$ as follows:

  \begin{definition}\label{ie-context} 
 For any $A\in \mathcal{L}^2$  we define a family of atomic elimination contexts $\mathsf{Elim}(A)$ and a family of introduction contexts $\mathsf{Intro}(A)$ by induction as follows:
    \begin{itemize}
      \item If $A=Z$, then $\mathsf{Elim}(A)=\mathsf{Intro}(A)=\{[\ ]\}$;
      
      % $\bb{\rr{\TT{El}}}_{A} =\bb{[\ ]}$  and $\bb{\rr{\TT{In}}}_{A} =\bb{ [\ ]}$
\item If $A= B\to C$, then $\mathsf{Elim}(A)=\{ \rr{ \TT E}[ [\  ]x]\mid x\text{ not free in }\rr{\TT E} \text{ and }\rr{\TT E}\in \mathsf{Elim}(C)\}$ and
$\mathsf{Intro}(A)=\{ \lambda  x.\rr{ \TT I}\mid \rr{\TT I}\in \mathsf{Intro}(C)\}$ 

%
%
%$\bb{\rr{\TT{El}}}_{A}^{z,\vec z}=\bb{\rr{\TT{El}}}_{C}  \bb{[ [\ ]z ]}$  and
%$\bb{\rr{\TT{In}}_{A}^{z,\vec z}}=\bb{ \lambda z. \rr{\TT{In}}_{C} }$
\item If $A=\forall Z.C$, then  
$\mathsf{Elim}(A)=\{ \rr{ \TT E}[ [\  ]X]\mid X \text{ not free in }\rr{\TT E} \text{ and }\rr{\TT E}\in \mathsf{Elim}(C)\}$ and
$\mathsf{Intro}(A)=\{ \Lambda  X.\rr{ \TT I}\mid \rr{\TT I}\in \mathsf{Intro}(C)\}$ 

%
%$\bb{\rr{\TT{El}}_{A} }=\bb{  \rr{\TT{El}}_{C} [ [\ ] Z%^A
%    ]}$ and $\bb{\rr{\TT{In}}_{A} }=\bb{ \Lambda Z. \bb{\rr{\TT{In}}_{C} }}$
\end{itemize}

An \emph{expansion pair for $A$} is a pair of contexts $(\rr{\TT{El}},\rr{ \TT{In}})$, where $\rr{\TT{El}}\in  \mathsf{Elim}(A), \rr{\TT{In}}\in \mathsf{Intro}(A)$ and $\rr{\TT{In}}\ho{\rr{\TT{El}}}\simeq_{\eta}[\ ]$.

\end{definition}

We list some useful and easily established facts about introduction and elimination contexts (recalling that for all elimination context $\TT E$ and term $t$, $\TT E\{t\}=\TT E[t]$):
%
%\begin{fact}
%For $A$ sp-$X$, either  $\bb{\Fun{X}{A}{\rr{\TT C}}}= \bb{\rr{\TT C}}$ or  $\bb{\Fun{X}{A}{\rr{\TT C}}}= \bb{\rr{\TT{In}}} [\bb{\rr{\TT C}}[\bb{\rr{\TT{El}}}]]$ for some expansion pair $(\rr{\TT{El}},\rr{ \TT{In}})$ for $A$.
%\end{fact}
\begin{fact}\label{etainel} If $\Gamma \vdash_{\Nd} t:A$ then for all expansion pairs $(\rr{\TT{El}},\rr{ \TT{In}})$ for $A$, 
 $\Gamma\vdash t \simeq_{\eta} \bb{\rr{\TT{In}}} \ho{\bb{\rr{\TT{El}}}[{t}]}:A$
\end{fact}
\begin{fact}\label{expa2} For all expansion pairs $(\rr{\TT{El}},\rr{ \TT{In}})$ for $A\supset B$ there exists $x\in \C{TV}$ such that
$\rr{\TT{El}}= \rr{\TT E}[[\ ]x]$ and $\rr{\TT{In}}=\lambda x.\rr{\TT I}$, where $(\rr{\TT E},\rr{\TT I})$ is an expansion pair for $B$.
Similarly, for all expansion pairs $(\rr{\TT{El}},\rr{ \TT{In}})$ for $\forall X.A$ there exists $X\in \C{V}$ such that
$\rr{\TT{El}}= \rr{\TT E}[[\ ]X]$ and $\rr{\TT{In}}=\Lambda X.\rr{\TT I}$, where $(\rr{\TT E},\rr{\TT I})$ is an expansion pair for $A$.

 %$\Gamma\vdash t \simeq_{\eta} \bb{\rr{\TT{In}}} [\bb{\rr{\TT{El}}} [t]]:A$
\end{fact}
\begin{fact}\label{alphaexpa}
If $(\rr{\TT E}_{1}, \rr{\TT I}_{1})$ and $(\rr{\TT E}_{2}, \rr{\TT I}_{2})$ are two expansion pairs for $A$, then for any 
$\rr{\TT U}: \at{A}\vdash^{\Gamma}\at{A}$ and term $\Gamma\vdash t:A$, if no variable free in $\rr{\TT U}$ and $t$ is bound in either $\rr{\TT I}_{1}$ or $\rr{\TT I}_{2}$, then $\rr{\TT I}_{1}\ho{\rr{\TT U}[ \rr{\TT E}_{1}[t]]} =
\rr{\TT I}_{2}\ho{\TT U[ \rr{\TT E}_{2}[{t}]]}$.
\end{fact}

In the following we will suppose fixed for any formula $A\in \C L^{2}$ an expansion pair $(\rr{\TT{El}}_{A}, \rr{\TT{In}}_{A})$.

\begin{definition}[Weak expansion and $A$-expansion]\label{cexp-lambda}
For all $A\in \C L^{2}$, the \emph{weak expansion} of $A$ is the context $\rr{\TT{In}}_{A}\ho{\rr{\TT{El}}_{A}}$. 
Moreover, if $A$ is sp-$X$, for all 
$\TT U: B\vdash^{\Gamma}C$, we let 
the \emph{$A$-expansion of $\rr{\TT U}$} be 
% $\bb{\Fun{X}{A}{\rr{\TT C}}$ where:
%\begin{itemize}
%\item if the rightmost variable of  $A$ is $X$,
 $\bb{\Fun{X}{A}{\rr{\TT U}}}=\rr{\TT{In}}_{A}\ho{\rr{\TT U}[\rr{\TT{El}}_{A}[x]]}$.
%\item otherwise  $\bb{\Fun{X}{A}{\rr{\TT C}}}=[\ ]$.

%\end{itemize}
%
%{\color{magenta} There is a problem here!}
%
%    \begin{itemize}
%\item       If $X\not\in FV(A)$ then $\bb{\Fun{X}{A}{\rr{\TT C}}}=[\ ]$
%      \item Otherwise, by induction on $A$:
%\begin{itemize}
%        % \item if $A= Y\neq X$, then $\bb{\Fun{X}{A}{\rr{\TT C}}}=[\ ]$;
%\item if $A=  X$, then $\bb{\Fun{X}{A}{\rr{\TT C}}}= \rr{\TT C}$;
%\item if $A= F\imp G$, then  $\bb{\Fun{X}{A}{\rr{\TT C}}=  \lambda y.(\Fun{X}{G}(\rr{\TT C})[ [\ ] %\Fun{X}{F}(\rr{\TT C})[y]
%y] )}$;
%\item if $A=\forall YG$, then $\bb{\Fun{X}{A}{\rr{\TT C}}= \Lambda Y.\Fun{X}{G}(\rr{\TT C}) [[\ ]Y%^{\forall Y.F[A/X]}
%    ]}$.
%\end{itemize}
%\end{itemize} 
\end{definition}%
  
Observe that by Fact \ref{alphaexpa}, the context $\Fun{X}{A}{\TT U}$ does not depend on the chosen expansion pair for $A$.
Moreover, it is easily checked that for all sp-$X$ formula $A\in \C L^{2}$, if  
$\rr{\TT U}: B \vdash^{\Gamma}C$ then $\Fun{X}{A}{\rr{\TT U}}: 
A\llbracket B/X\rrbracket \vdash^{\Gamma} A\llbracket C/X\rrbracket
$.
%(\Delta\vdash A[B/X]) \Rightarrow (\Sigma \vdash A[C/X]$.
%
% %$A\in \Pos_{X}$ and 
%$\bb{\rr{\TT C}}: B \vdash^{\Gamma} C$ then   $\bb{\Fun{X}{A}{\rr{\TT C}}}: A[B/X]\vdash^{\Gamma} A[C/X]$
%%\item  If  $A\in \Neg_{X}$ and $\bb{\rr{\TT C}}: B \vdash^{\Gamma} D$ then   $\bb{\Fun{X}{A}{\rr{\TT C}}}: A[D/X]\vdash^{\Gamma} A[B/X]$
%%\end{itemize}
%
%\PP{e se C mangia variabili?? Possiamo scrivere piu in generale:
%
%    If  %$A\in \Pos_{X}$ and 
%$\rr{\TT C}: (\Delta\vdash B) \Rightarrow (\Sigma \vdash C)$ then $\Fun{X}{A}{\rr{\TT C}}: (\Delta\vdash A[B/X]) \Rightarrow (\Sigma \vdash A[C/X]$
%
%In particular,  [come sopra]
%
%    Serve per il lemma F2}

%\end{fact}

 \begin{remark}In \cite{StudiaLogica} we defined the notion of $A$-expansion in a more direct manner for the broader class of strictly positive formulas (and not, as done here, strongly positive, see Remark~\ref{rem:spX}). The interested reader can easily check that the two definitions coincide in the case of strongly positive formulas.
 \end{remark}

\section{The RP-translation}\label{rp-appendix}

\begin{definition}[RP-translation]\label{rp-lambda}
Given the RP-translation of formulas (Definition~\ref{def:rp1} in Section~\ref{rp-trans}), for every term $u$ such that $\Gamma\vdash_{\Ndp} u:A$ we define a term $u^*$  as follows:

\medskip
{\centering \small\begin{tabular}{@{}l|l}
 % \hline\\
  \parbox{.3\textwidth}{
  $x^{*} =x$

  $(\lambda x.t)^* = \lambda x. t^*$

  $(ts)^* = t^*s^*$

  $(\Lambda X.t)^* = \Lambda X. t^*$

                    $(tB%^\forall X.A
                    )^* = t^*B^*%^{*\forall X.A^*}
                    $
  }
  &
     \parbox{.6\textwidth}{
%$\bb{ ( t,u)^{\rop}} = \bb{ \Lambda X.\lambda y.yt^{\rop}u^{\rop} }$

 %    $     \bb{   (\pi_{i}^{B_1\wedge B_2}t)^{\rop}} =\bb{  t^{\rop}B_i^{\rop(B_1\wedge B_2)^{\rop}}\lambda x_{1}.\lambda x_{2}.x_{i} }$
     
    $%\left\langle\ \
     \bb{ (\iota_{\dagger}^{k_{0}}\langle t_j\rangle )^{\rop} } =\bb{  \Lambda X.\lambda \langle x_{k}\rangle.x_{k_{0}}\lost{t_j^{\rop}}{j}{g^{-1}(k)}}%\ \ \right\rangle_{k\in \C K}
    $

     \medskip
    $\bb{ (\delta_{\dagger}%(A_i)\Rightarrow C}
    (t, \lost{\lost{y_j}{}{}.s_{k}}{}{}))^{\rop}}  =\bb{  t^{\rop}A^{\rop%(\dagger(A_i)^{\rop}
    } \lost{\lambda \lost{y_j}{j}{g^{-1}}.s_{k}^{\rop}}{k}{\C K}}$

     }
%\\\hline
                  \end{tabular}

                }
              \end{definition}

              \noindent
The following fact is easily checked by induction on a derivation of $\Gamma\vdash_{\Ndp} u:A$:                            \begin{fact}\label{rp-derpres}If $\Gamma\vdash_{\Ndp} u:A$, then    $\Gamma^*\vdash_{\Nd} u^*:A^*$ and if $\Gamma\vdash_{\Np} u:A$, then    $\Gamma^*\vdash_{\Ndrp} u^*:A^*$.
              \end{fact}
              %\medskip\noindent
\begin{remark}
The clauses for $\iota$- and $\delta$-terms generalize the standard translation of disjunction and conjunction constructors in System F (where  $\iota_{\wedge}(t_1,t_2)$ is prefix notation for  the standard pairing).

\medskip
{\centering\small $\begin{array}{ll} \bb{ (\iota_{\mathord{\vee} }^{u}t)^{\rop} } =\bb{  \Lambda X.\lambda x_1x_2.x_{u}t^{\rop}}\ ({u=1,2 })& \bb{ (\delta_{\vee}%^{(A_1\vee A_2)\Rightarrow C}
                                                                                                                                                                           (t, y.s_{1}, y.s_2))^{\rop}}  =\bb{  t^{\rop}A^{\rop%(A_1\vee A_2)^{\rop}
                                                                                                                                                                           } (\lambda y.s_{1}^{\rop}) (\lambda y.s_{2}^{\rop})}\\
                     \bb{ (\iota_{{\!}\mathord{{\wedge}}{\!}}(t_1, t_2))^{\rop} } =\bb{  \Lambda X.\lambda x.xt_1^{\rop}t_2^{\rop}}\ & \bb{ (\delta_{\wedge}%{A_1\wedge A_2)\Rightarrow C}
                                                                                                                                       (t, y_1y_2.s))^{\rop}}  =\bb{  t^{\rop}A^{\rop%(A_1\wedge A_2)^{\rop}
                                                                                                                                       } \lambda y_1y_2.s^{\rop}}
        \end{array}
        $

      }     
    \end{remark}

    \begin{remark}\label{rp-context} 
    It can be checked that if $\TT C: A\vdash^{\Gamma}B$ is a context in the language of $\Ndp$, then
    $\TT C^{*}: A^{*}\vdash^{\Gamma^{*}}B^{*}$ is a context in the language of $\Nd$. Moreover, 
     if $\rr{\TT E}$ is an elimination context then $\rr{\TT E}^*$ is also an elimination context. \end{remark}
%
%    
%     $\rr{\TT C}: (\Delta\vdash A) \Rightarrow (\Sigma \vdash B)$, an $\Nd$-context $\rr{\TT C}^*: (\Delta^*\vdash A^*) \Rightarrow_{\Nd} (\Sigma^* \vdash B^*)$. Observe that
     
    % \nopagebreak[4]

      \section{The $\varepsilon$-equation}\label{epsilonappendix}

%  Let $A\in RP\subset \mathcal{L}^2$ iff $A=\forall X.A_{1}\imp \dots \imp A_{k}\imp X$, where each  $A_i$ is $\SPos_X$.
  % \begin{itemize}
  % \item  $A_i= B_{i1}\imp \ldots \imp B_{im_i}\imp X$
   %   \item $X\nin FV({B_{ij}})$ 
  %   \end{itemize}

We write $\Gamma\vdash t\simeq_{\varepsilon}u:A$ iff there is a derivation of $\Gamma\vdash t\simeq u$ using the rules of congruence, symmetry, transitivity, reflexivity  and any instance of the rule below:
%
%The $\varepsilon$-theory is the smallest congruence generated by the following equations: %(\PP{CHECK}):

\medskip
{\centering\small
$
\AXC{$\Gamma\vdash t: (\dagger\langle A_i\rangle)^*$}
\AXC{$\left\langle \Gamma\vdash u_{k}: (\EXPO{X}{j\in g^{-1}(k)  }{A_{f(j)}^{\phantom{fj}*}})\llbracket C/X\rrbracket  \right\rangle_{k\in \C K} $}
\AXC{$\rr{\TT U}:C\Rightarrow^{\Gamma}  D$}
\RL{$\varepsilon$}\TrinaryInfC{$
\Gamma\vdash 
  \rr{\TT U}[ tC%^{(\dagger(A_i))^{^*}}
  \langle{u_k}\rangle_{k\in\C K}] \simeq
  tD%^{(\dagger(A_i))^{^*}}
  \left\langle
  \Fun{X}{(\EXPO{X}{j\in g^{-1}(k)}{A_{f(j)}^{\phantom{fj}*}})}{
  \rr{\TT U}}[u_{k}]\right\rangle_{k\in\C K}
: D$}
\DP
$

}

For the use of e.g.~$\simeq_{\eta\varepsilon}$ and $\rightsquigarrow_{\varepsilon}$ we adopt the same conventions introduced at the end of Appendix~\ref{betaetagamma}.

\begin{remark}\label{rem:epsicapturing}
  As for the rules $\dagger\eta^{+}$ and $\dagger\gamma^{+}$, we formulated the rule $\varepsilon$ using n.v.c.-substitution  of a term for the hole of  $\TT U$.  Similarly to the other cases (see Remark \ref{rem:etacapturing}), a variant of $\varepsilon$  with v.c.-substituion in place of n.v.c.-substitution  can be derived from the rule $\varepsilon$ along with \emph{$\beta$- and $\eta$-conversions and symmetry}.

\end{remark}

% \begin{remark}\label{strongpx2}
% In \cite{StudiaLogica} we defined the $\varepsilon$-equation  for the more general case in which the conclusion of $\D D_1$ is a nested p-$X$ formula, where a formula is nested p-$X$ if it is of the  form $\forall \vec{Y_{1}}(F_{1}\supset \forall \vec Y_{2}(F_{2}\supset \dots \supset \forall \vec Y_{n}(F_{n}\supset\forall \vec Y_{n+1} (X))\dots ))$
% % (the $\forall \OV Y_{i}$ represent possibly empty lists of quantifiers) and
% where $\forall \vec Y_1 F_{1},\dots, \forall \vec Y_1 \ldots \forall \vec Y_n F_{n}$ are formulas in which $X$ occurs only positively.% are p-$X$ formulas. As in the present paper we  exclusively deal with (nested) sp-$X$ formulas,  we choose to formulate $\varepsilon$ directly for this more restricted class of formulas.%  and address some remarks on these more general classes of formulas only in the conclusions.%, we will restrict to the simpler We will refer to such formulas as \emph{naturality formulas}. 
% %We don't know if these formulas can be related to translation of propositional connectives. 

% %%In the present paper we will  only use the  restricted form given in Table~\ref{epsi-conversions}.% (i.e.~the instance of the more general scheme in which $n=2$, $\forall \OV{Y_1}$ and $\forall \OV{Y_2}$ are empty, and $F_1$ and $F_2$ are $A\supset X$ and $B\supset X$ respectively).
% \end{remark}

\section{Proof of Proposition~\ref{eps-preservation}}\label{eps-preservation-lambda}
We will establish the stronger statement below.

\begin{proposition}\label{eps-preservation2}
For all $u,v$ if $\Gamma\vdash_{\Ndp} u\simeq_{\beta\eta^{+}}v:A$ then $\Gamma^*\vdash_{\Nd} u^{*}\simeq_{\beta\eta\varepsilon} v^{*}:A^*$.
\end{proposition}

We need the following two lemmas:

\begin{lemma}\label{subscontext}
For all $\rr{\TT U}: A\vdash_{\Ndp}^{\Gamma}B$ and $t$ such that $\Gamma\vdash_{\Ndp} t:A$, $(\rr{\TT U}[t])^*= \rr{\TT U}^*[t^*]$.
\end{lemma}

\begin{proof}
The lemma is easily established by  induction on $\rr{\TT U}$.
\end{proof}

\begin{lemma}\label{lemmacexp}
For all  sp-$X$ formulas $A\in\Ld$ such that $A=\EXP{X}{a\in \C A}{A_a}$, if $\rr{\TT U}:
B\vdash^{\Gamma}C$
 %(\Delta\vdash B) \Rightarrow (\Sigma \vdash C)$ 
 and $\Gamma,\lust{x_a:A_{a}}{a}{\C A}\vdash_{\Nd} t:B$, the following hold:  
   \begin{center}
  {\small
$ \Gamma\vdash \Fun{X}{(\EXP{X}{a\in \C A}{A_a})}{\rr{\TT U}}[\lambda \lust{x_a}{a}{\C A}. t] \rightsquigarrow_{\beta} \lambda \lust{x_a}{a}{\C A}.\rr{\TT U}[t]:\EXP{C}{a\in\C A}{A_a}$
}
\end{center}
Moreover, if no free variable of $\TT{El}_{\langle A_{a}\rangle_{a\in \mathcal A}\supset X}$ is bound in $\TT U$, then
   \begin{center}
  {\small
$ \Gamma\vdash \Fun{X}{(\EXP{X}{a\in \C A}{A_a})}{\rr{\TT U}}\{\lambda \lust{x_a}{a}{\C A}. t\} \rightsquigarrow_{\beta} \lambda \lust{x_a}{a}{\C A}.\rr{\TT U}\{t\}:\EXP{C}{a\in\C A}{A_a}$
}
\end{center}
\end{lemma}

\begin{proof}
The lemma is easily established by  induction on the length of the list $\C A$.
\end{proof}

%\begin{remark}\label{rem:graffone}
%Lemma \ref{lemmacexp} can be extended with no difficulty to the case in which 
%\end{remark}

\begin{proposition}\label{gammag-epspres}
For all $u,v$ if $\Gamma\vdash_{\Ndp} u\rightsquigarrow_{\gamma^{+}}v:A$ then $\Gamma^*\vdash_{\Nd} u^{*}\rightsquigarrow_{\beta\varepsilon} v^{*}:A^*$.
\end{proposition}

\begin{proof}
  By induction on the typing derivation $\D D$ of $\Gamma\vdash_{\Ndp} u\rightsquigarrow_{\gamma^{+}}v:A$. If $\D D$ ends with an application of either reflexivity, transitivity or of one of the congruence rules, it is enough to apply the induction hypothesis to the derivations of the premises, and then applying the reflexivity, transitivity or  the  congruence rules.

  If $\D D$ ends with an application of $\dagger\gamma^{+}$, then for some context $\rr{\TT U}:  B\vdash^{\Gamma}  A$, $ u=\rr{\TT U}[\bb{\delta_{\dagger}(t,\lost{\lost{y_j}{}{}.s_{k}}{})}]$ and $v= \bb{\delta_{\dagger}(t,\lost{\lost{y_j}{}{}.\rr{\TT U}[s_{k}]}{}{})}$, where $\Gamma\vdash_{\Ndp} \delta_{\dagger}(t,\lost{\lost{y_j}{}{}.s_{k}}{}{} ) : B$, $\Gamma\vdash_{\Ndp} t:\dagger\lost{A_i}{}{}$ for some~$\dagger$, and   $\Gamma, \langle y_j:A_{f(j)}\rangle\vdash_{\Ndp} s_k: B$ for every $k\in\C K$. Thus $\Gamma^*\vdash_{\Nd} t^*: \forall X.\EXPO{X}{k\in\C K}{\EXPO{X}{j\in g^{-1}(k)}{A_{f(j)}^*}}$, $\Gamma^*, \langle y_j:A^*_{f(j)}\rangle\vdash_{\Nd} s_k^*: B^*$ for every $k\in\C K$ and we can construct a derivation $\D D'$ of $\Gamma^*\vdash_{\Nd}   u^{^*} \rightsquigarrow_{\beta\varepsilon} 
  v^{^*}:A^*$ as shown below:

  \medskip
  {\centering\small
 $\bb{ u^{*}}  \stackrel{\text{Def.~\ref{rp-lambda}}+\text{Lemma~\ref{subscontext}}}{=}\bb{\rr{\TT U}^* \left[ (t^{*}B^*)
     \lost{\lambda \lost{y_j}{j}{g^{-1}(k)}.s_{k}^{*}}{k}{\C K}\right]}   %\stackrel{\text{Lemma~\ref{lemmett2}}}{
 \rightsquigarrow_{\varepsilon}%}
( t^{*} A^*)\lost{\Fun{X}{(\EXPO{X}{j\in g^{-1}(k)}{A_{f(j)}^*})}{\rr{\TT U}^*}[\lambda\lost{y_j}{}{}.s_{k}^{*}]}{}{} 
\stackrel{\text{Lemma~\ref{lemmacexp}}}{\rightsquigarrow_{\beta}} ( t^{*} A^*)\lost{\lambda\lost{y_j}{}{}.\rr{\TT U}^*[s_{k}^{*}]}{}{}  \stackrel{\text{Lemma.~\ref{subscontext}+ Def.~\ref{rp-lambda}}}{=} \bb{v^{^*}}
$

}

\end{proof}

\begin{proposition}\label{etag-epspres}
For all $u,v$ if $\Gamma\vdash_{\Ndp} u\simeq_{\eta}v:A$ then $\Gamma^*\vdash_{\Nd} u^{*}\simeq_{\beta\eta\varepsilon} v^{*}:A^*$.
\end{proposition}
\begin{proof}
 As in the proof above we can argue by induction on the derivation $\D D$ of $\Gamma\vdash_{\Ndp} u\simeq_{\eta}v:A$. 
 The only non-trivial case is when $\D D$ ends with an application of $\dagger\eta$. 
% Using Remark \ref{gammag-lambda} we can split $\dagger\eta^{+}$ as an application of $\dagger\eta$ and one of $\dagger\gamma^{+}$. As the latter can be treated using Proposition \ref{gammag-epspres}, it suffices to treat the case of $\dagger\eta$. 
 So suppose 
  $u= \bb{\delta_{\dagger}(v,\langle{\lost{y_j}{}{}.\iota_{\dagger}^{k}\langle y_{j}\rangle}\rangle_{k})}$. Hence $\Gamma\vdash_{\Ndp} v:\dagger\lost{A_i}{}{}$ for some~$\dagger$, and thus $\Gamma^*\vdash_{\Nd} v^*: \forall X.\EXPO{X}{k\in\C K}{\EXPO{X}{j\in g^{-1}(k)}{A_{f(j)}^*}}$), and we can construct a derivation $\D D'$ of $\Gamma^*\vdash_{\Nd}   u^{^*} \simeq_{\beta\eta\epsilon} 
  v^{^*}:A^*$ as shown below:

  \medskip
  {\centering\small
 $\bb{ v^{*}}  
 \simeq_{\eta}
 \Lambda X. \lambda \langle f_{k}\rangle_{k\in \C K}. v^{*}X  \ \langle  \lambda \langle y_{j}\rangle. f_{k}\langle y_{j}\rangle \rangle_{k\in \C K} =
 \rr{\TT C}\left  \{v^{*}X  \ \langle  \lambda \langle y_{j}\rangle. f_{k}\langle y_{j}\rangle \rangle_{k\in \C K}
\right\}
% \stackrel{\text{Def.~\ref{rp-lambda}}}{=}\bb{\rr{\TT C}^* \left[ (t^{*}A^*)
%     \lost{\lambda \lost{y_j}{j}{g^{-1}(k)}.s_{k}^{*}}{k}{\C K}\right]}   %\stackrel{\text{Lemma~\ref{lemmett2}}}{
 \stackrel{\text{Rem. \ref{rem:epsicapturing}}}{\simeq_{\beta\eta\varepsilon}}%}
 (v^{*} A^{*}) 
\left \langle{\Fun{X}{(\EXPO{X}{j\in g^{-1}(k)}{A_{f(j)}^*})}{\rr{\TT C}}\ho{\lambda\lost{y_j}{}{}.f_{k}\langle y_{j}\rangle}}\right\rangle_{k\in \C K} 
\stackrel{\text{Lemma~\ref{lemmacexp}}}{\rightsquigarrow_{\beta}} ( v^{*} A^*)\lost{\lambda\lost{y_j}{}{}.\rr{\TT C}\{f_{k}\langle y_{j}\rangle\}}{}{}  \stackrel{\text{Lemma.~\ref{subscontext}+ Def.~\ref{rp-lambda}}}{=} \bb{u^{^*}}
$

}

\noindent
where $\rr{\TT C}= \Lambda X.\lambda \langle f_{k}\rangle_{k\in \C K}.[\ ]$ and the variables $X, f_{k}$ are chosen so as not to occur free in the contexts $\TT{El}_{\langle A_{f(j)}^{*}\rangle \supset X}$.

%\

%\noindent
%where $\rr{\TT C}= \Lambda X.\lambda \langle f_{k}\rangle_{k\in \C K}.[\ ]:  (  \langle \langle A^*_{f(j)}\rangle \supset X\rangle_{k\in\C K}  \vdash X)  \Rightarrow(\vdash A^{*}) $.
\end{proof}

\begin{proof}[Proof of Proposition \ref{eps-preservation2}]
It suffices to check rule-by-rule that equivalences are preserved. 
The only non-trivial case is the one of  $\dagger\eta^+$, which is treated using Remark~\ref{gammag-lambda}, Proposition \ref{gammag-epspres} and Proposition \ref{etag-epspres}.
%The proposition 
%follows by the previous one together with Propositions 15, 13 and 12. \PP{Check??}
\end{proof}

\section{The three translations into $\Nda$}\label{ff-def-appendix}

\begin{definition}[Atomization]\label{at-lambda} For every term $u$ such that $\Gamma\vdash_{\Ndrp} u:A$ we define a term $u^\downarrow$ %such that $\Gamma\vdash_{\Nda} u^\downarrow:A$
  as follows (observe that when $u =tB$ then $\Gamma\vdash_{\Ndrp}t:\forall W.F$ and since $F$ is polynomial in $W$, by what we observed in Subsection \ref{rp-trans}, $F$ can be written as $  \EXPO{W}{ i\in \C K   }{ \EXPO{W}{j\in \OV{g}(k)  }{A_{f(j)}^{\phantom{fj}*}}}$, with the $A_{i}$ having no occurrence of $W$):

%  The mapping \PP{CHECK}$\downarrow:\Ndrp\mapsto\Nda$ defined as follows:

{\centering\small\begin{tabular}{@{}l|l}
  \parbox{.23\textwidth}{$x^{\downarrow} =x$

  $(\lambda x.t)^\downarrow = \lambda x. t^{\downarrow}$

  $(ts)^{\downarrow} = t^{\downarrow}s^{\downarrow}$

  $(\Lambda X.t)^\downarrow = \Lambda X. t^{\downarrow}$}&

\parbox{.75\textwidth}{
                                                           $\phantom{A}$

                                                           $(tY%^{(\dagger(A_i))^*}
                                                           )^{\downarrow} = t^{\downarrow}Y%^{(\dagger(A_i))^*}
                                                           $

                                                           $(t(C\imp D)%^{(\dagger(A_i))^*}
                                                           )^{\downarrow} = \lambda \lost{y_k}{k}{\C K}.\lambda x. (tD%^{\dagger(A_i)^*}
                                                           )^{\downarrow} \lost{\lambda \lost{z_j}{j}{g^{-1}(k)}. y_k \lost{z_j}{j}{g^{-1}(k)} x}{k}{\C K}$

                                                           $(t(\forall X.D)%^{(\dagger(A_i))^*}
                                                           )^{\downarrow} =\lambda \lost{y_k}{k}{\C K}.\Lambda X. (tD%^{\dagger(A_i)^*}
                                                           )^{\downarrow} \lost{\lambda \lost{z_j}{j}{g^{-1}(k)}. y_k \lost{z_j}{j}{g^{-1}(k)} X}{k}{\C K}$

                                                        }\\
                 \end{tabular}

               }
             \end{definition}

%\section{The FF$^{\varepsilon}$-translation}

\begin{definition}[$\varepsilon$-atomization]\label{ateps-lambda}
  For every term $u$ such that $\Gamma\vdash_{\Ndrp} u:A$ we define a term $u^{\downarrow^{\varepsilon}}$ %such that $\Gamma\vdash_{\Nda} u^{\downarrow^{\varepsilon}}:A$
  by replacing in Definition~\ref{at-lambda} the last two clauses for terms of the form $tB$ (i.e. with $B$ non-atomic) with the following single clause (assuming $Z=\at{B}$):

% \begin{itemize}
%   \item $x^{\downarrow^{\varepsilon}} =x$

% \item   $(\lambda x.t)^{\downarrow^{\varepsilon}} = \lambda x. t^{\downarrow^{\varepsilon}}$

%  \item  $(ts)^{\downarrow^{\varepsilon}} = t^{\downarrow^{\varepsilon}}s^{\downarrow^{\varepsilon}}$

%  \item   $(\Lambda X.t)^{\downarrow^{\varepsilon}} = \Lambda X. t^{\downarrow^{\varepsilon}}$

%  \item

\medskip
{\centering\small$(tB)^{\downarrow^{\varepsilon}}\ =\  \lambda \lost{y_k}{k}{\C K}. \bb{\rr{\TT{In}}_{B}}\left\{ t^{\downarrow^{\varepsilon}}Z%^{(\dagger(A_i))^*}
  \lost{\Fun{X}{(\EXPO{X}{j\in g^{-1}(k)  }{A_{f(j)}^{\phantom{fj}*}})}{\bb{\rr{\TT{El}}_B}}[y_k]}{k}{\C K}\right\} $

}
%\end{itemize}
\end{definition}
% \medskip
% The FF-translation preserves derivability and identity of proofs:
% \begin{tabular}{cc}
%   \parbox{.5\textwidth}{\begin{itemize}
%     \item If $\Gamma\vdash_{\NI} A$ then $\Gamma^{*\downarrow}\vdash_{\Nda}A^{*\downarrow}$
%   \end{itemize}}
% &
  % \parbox{.5\textwidth}{\begin{itemize}
% \item If $\mathscr{D}_1 \stackrel{\beta}{\simeq} \D D_2$ then $\D D_1^{*\downarrow} \stackrel{\beta}{\simeq} \D D_2^{*\downarrow}$
% \item {If $\mathscr{D}_1 \stackrel{\eta^*}{\simeq} \D D_2$ then $\D D_1^{*\downarrow} \stackrel{\beta\eta}{\simeq} \D D_2^{*\downarrow}$}
% \item {If $\mathscr{D}_1 \stackrel{\gamma^*}{\simeq} \D D_2$ then $\D D_1^{*\downarrow} \stackrel{\beta\eta}{\simeq} \D D_2^{*\downarrow}$}
%   \end{itemize}
%  } 
% \end{tabular}
It easy to  check by induction on the derivation of $\Gamma\vdash_{\Ndrp} u:A$ that the following holds:
                            \begin{fact}\label{ff-derpres}For $\natural\in\{\downarrow, \downarrow^{\varepsilon}\}$, if $\Gamma\vdash_{\Ndrp} u:A$, then    $\Gamma\vdash_{\Nda} u^\natural:A$.
              \end{fact}
              From this and Fact~\ref{rp-derpres}  it follows moreover the following:
              \begin{corollary}For $\natural\in\{\downarrow, \downarrow^{\varepsilon}\}$, if $\Gamma\vdash_{\Np} u:A$, then    $\Gamma^*\vdash_{\Nda} u^{*\natural}:A^*$.
              \end{corollary}
              
\begin{definition}[ESF-translation]\label{reftrans}
The ESF-translation is obtained by replacing in  Definition~\ref{rp-lambda} $\Ndp$ with $\Np$ and the last clause with the following (we assume here $\Gamma\vdash_{\Np} u:A$ and $Z=\at{A}$):

$$\bb{ (\delta_{\dagger}%(A_i)\Rightarrow C}
  (t, \lost{\lost{y_j}{}{}.s_{k}}{}{}))^{\sharp}}  =\bb{ \rr{\TT{In}}_{A^*}  \left \{  (t^{\sharp}Z)%^{(\dagger(A_i))^{\rop}}
  \lost{\lambda \lost{y_j}{j}{g^{-1}(k)}.\rr{\TT{El}}_{A^*} [s_{k}^{\sharp}]}{k}{\C K}\right\} }$$
\end{definition}

Also in this case, it  easily  checked by induction on the derivation of $\Gamma\vdash_{\Np} u:A$ that the following holds:
                            \begin{fact}\label{esf-derpres}If $\Gamma\vdash_{\Np} u:A$, then    $\Gamma^*\vdash_{\Nda} u^\sharp:A^*$.
              \end{fact}

The interested reader can check that the refined interpretation proposed by Esp\'irito Santo and Ferreira  in \cite{ESF19} coincides essentially %\PP{what is ``essentially'' here?}
with ours in the case of $\Nv$.

%\section{Proof of Proposition \ref{ff-ffe-relation}}\label{prova-app}
%We prove the following generalization of Proposition~\ref{ff-ffe-relation} to  $\Ndrp$:

  \begin{remark}\label{remhelp}
  All three translations $^{*\downarrow},^{*\downarrow^{\varepsilon}}, ^{\sharp}:\Np\to \Nda$ extend straightforwardly to contexts, similarly to what happens for the RP-translation (see Remark \ref{rp-context}).
  Moreover, for each of the three translations $^{\TT T}$ it can be checked by induction that 
   for all elimination context $\rr{\TT E}$ of $\Np$, $\rr{\TT E}^{\TT T}$ is a principal context and that 
      $\TT E^{\TT T}\{t^{\TT T}\}= \TT E^{\TT T}[t^{\TT T}]= (\TT{E}[t])^{\TT T}$.
%   
%
%    for all term $t$ we can always suppose that $\TT E^{\TT T}$ does not capture any free variable of $t$, so that 
%$\TT E^{\TT T}\{t\}= \TT E^{\TT T}[t]$. Finally, one can check by induction on an elimination context $\TT E$ that 
%  $\TT{E}^{\TT T}[t^{\TT T}]= (\TT{E}[t])^{\TT T}$.        
  \end{remark}

\section{Proof of Proposition~\ref{ff-preservation}}\label{etagamma-appendix}

We prove the following generalization to  $\Np$ of Point 2 of Proposition~\ref{ff-preservation}:

\begin{proposition}\label{ff-preservation2}
  For all $u, v$, if $\Gamma\vdash_{\Np} u\rightsquigarrow_{\beta\eta\gamma} v:A$ %: %(1) $\Gamma^*\vdash_{\Nda}u^{*\downarrow}\rightsquigarrow_{\beta\eta} v^{*\downarrow}$;\linebreak (2)
  % $\Gamma^*\vdash_{\Nda}u^{*\downarrow^{\varepsilon}}\rightsquigarrow_{\beta\eta} v^{*\downarrow^{\varepsilon}}$; (3)
then   $\Gamma^*\vdash_{\Nda}u^{\sharp}\simeq_{\beta\eta} v^{\sharp}:A^*$.
\end{proposition}

Let $\mathcal A$ be a finite list. 
By a \emph{$\mathcal A$-multicontext} $\rr{\TT M}$ we indicate a term containing exactly one occurrence of $|\mathcal A|$ distinct special variables noted as $[\ ]_{a}$, for each $a\in\mathcal A$. Given
a $\mathcal A$-multicontext $\TT M$ and a finite list of terms $\langle u_{a}\rangle_{a\in \mathcal A}$, we let $\TT M\{\langle u_{a}\rangle_{a\in \mathcal A}\}$ be the term obtained by simultaneous variable-capturing substitution of $[\ ]_{a}$ by $u_{a}$.
We let $\TT M: \langle \Delta_{a}\vdash A\rangle_{a\in\mathcal A} \Rightarrow( \Gamma\vdash B)$ indicate that for all $\mathcal A$-indexed list of terms  $\langle u_{a}\rangle_{a\in \mathcal A}$ such that $\Gamma, \Delta_{a}\vdash u_{a}:A$, $\Gamma\vdash\TT M\{\langle u_{a}\rangle_{a\in \mathcal A}\}:B$ (meaning that the free variables of $u_{a}$ in $\Delta_{a}$ are captured by $\TT M$).

% which might contain several occurrences of $[\ ]$.
% If $\rr{\TT C}$ is a context and $\rr{\TT U}$ is a multicontext, we indicate by $\rr{\TT U} [ \rr{\TT C}]$ the result of (variable-capturing) substitution of $\rr{\TT C}$ for any hole of $\rr{\TT U}$. Observe that 
% the variable captured at each hole might well differ. We let $\TT U: A\vdash^{\Gamma}B$ be a shorthand for 
% $\Gamma, x:A\vdash \TT U[x]: B$.
% 

%
%In order to prove the result above in a general way we use 

%We need the following lemma:

\begin{lemma}\label{lemmett1} 
Let $A,B$ be formulas of $\Nd$; let  
$u_{1},\dots, u_{|\mathcal A|}$ be terms such that $\Gamma, \Delta_{a}\vdash_{\Nda} u_{a}:A$, $\rr{\TT C}: A\vdash^{\Gamma}_{\Nda} B$ be a principal context of $\Nda$ and $\rr{\TT M}: 
\langle \Delta_{a}\vdash_{\Nda} X\rangle_{a\in \mathcal A}\Rightarrow  (\Gamma\vdash_{\Nda} X)$
%X\vdash^{\Gamma}_{\Nda}X$
 be a $\mathcal A$-multicontext such that $X$ does not occur free neither in $\Gamma$ nor in the $\Delta_{a}$.
Moreover, suppose that 
 %$\TT C$ binds no term or type variable of either $\TT M$ or $\TT{El}_{A}$, that 
 $\TT M$ binds no term or type variable of $\TT{El}_{A}$ and that $\TT{In}_{A}$ binds no term or type variable of $\rr{\TT M}$, $\TT C$ and of the $u_{a}$. Then
% 
%$\TT U$ binds no term or type variable of $
% $\rr{\TT{In}}_{A}$ binds no term variable of $\rr{\TT U}$ (but possibly some type variable) and binds no term nor type variable of $u$, and $\TT U$ binds no type and no term variable of $\TT{El}_{A}$, then 
%  and $\rr{\TT{In}}_{A}$ and $\rr{\TT U}$ have no bound variable in common, then 
$$
\rr{\TT C}\Big [ \rr{\TT{In}}_{A}\big\{  \rr{\TT{M}_{\at A}}\{ \langle\rr{\TT{El}}_{A}[u_{a}]\rangle_{a\in\mathcal A}\}\big \} \Big ] \simeq_{\beta} 
\rr{\TT{In}}_{B}\Big\{ \rr{\TT{M}_{\at{B}}} \big\{ \langle \rr{\TT{El}}_{B}[ \rr{\TT C}[ u_{a}] ] \rangle_{a\in \mathcal A} \big\} \Big \}
$$
where $\TT M_{Y}$ is shorthand for $\TT M\llbracket Y/X\rrbracket$.
%
% and expansion pair $(\rr{\TT E}, \rr{\TT I})$ for $A$, 
%$
%
% $(\bb{\rr{\TT{El}}_{B} [\rr{\TT E}]}, \rr{\TT I})$, where $\rr{\TT I}$ is the $\beta$-normal form of $\bb{\rr{\TT E}[\rr{\TT{In}}_{A} ]}$, is an expansion pair for $A$.
%
%\begin{itemize}
%\item[i.]
%For all $\rr{\TT E}: A\vdash^{\Gamma} B$, we have that: $\bb{\rr{\TT{El}}_{B} [\rr{\TT E}]} \in \mathsf{Elim}(A)$. %\bb{\rr{\TT{El}}_{A} }$ 
%\item[ii.]
%For all $\rr{\TT E}: A\vdash^{\Gamma} B$, we have that: $\bb{\rr{\TT E}[\rr{\TT{In}}_{A} ]} \rightsquigarrow_{\beta}  \rr{\TT I}\in \mathsf{Intro}(A)$. %\bb{\rr{\TT{In}}_{B} }$ 
%\item[iii.] Moreover, $(\bb{\rr{\TT{El}}_{B} [\rr{\TT E}]}, \bb{\rr{\TT E}[\rr{\TT{In}}_{A} ]})$ is an expansion pair for $A$.
%\end{itemize}
\end{lemma}
\begin{proof}
We argue by induction on the principal context $\TT C$:
\begin{itemize}
\item if $\TT C=[\ ]$ the claim is immediate;

\item if $\TT C= \TT C'v$, then $\TT C': A\vdash^{\Gamma}_{\Nda} C\supset B$ and $\Gamma\vdash_{\Nda } v:C$,  and noticing that $\at{B}=\at{C\supset B}$ we then have

  \medskip
  {\centering\small
  \begin{tabular}{c}
  $
\rr{\TT C}\Big [ \rr{\TT{In}}_{A}\big\{ \rr{\TT{M}_{\at{A}}} \{\langle\rr{\TT{El}}_{A}[u_{a}] \rangle_{a\in \mathcal A}\}\big \} \Big ]
= \rr{\TT C'}\Big [ \rr{\TT{In}}_{A}\big\{ \rr{\TT{M}_{\at{A}}} \{\langle\rr{\TT{El}}_{A}[u_{a}]\rangle_{a\in \mathcal A} \}\big \} \Big ] v $\\
$
\stackrel{\text{I.H.}}{\simeq_{\beta}}
\Big ( \rr{\TT{In}}_{C\supset B}\big\{ \rr{\TT{M}_{\at{B}}} \{\langle\rr{\TT{El}}_{C\supset B}[\TT C'[u_{a}]]\rangle_{a\in \mathcal A} \}\big \} \Big ) v 
=
\Big ( \lambda x.\rr{\TT{In}}_{ B}\big\{ \rr{\TT{M}_{\at{B}}} \{\langle\rr{\TT{El}}_{ B}[\TT C'[u_{a}]x]\rangle_{a\in \mathcal A} \}\big \} \Big ) v 
$ \\
$\rightsquigarrow_{\beta}
 \rr{\TT{In}}_{ B}\big\{ \rr{\TT{M}_{\at{B}}}\{\langle \rr{\TT{El}}_{ B}[\TT C'[u_{a}]v] \rangle_{a\in \mathcal A}\}\big \} 
 =
\rr{\TT{In}}_{ B}\big\{ \rr{\TT{M}_{\at{B}}}\{ \langle \rr{\TT{El}}_{ B}[\TT C[u_{a}]] \rangle_{a\in \mathcal A}\}\big \} 
$
\end{tabular}
}

%
%
%\item if $\TT C= v\TT C'$, then $\Gamma\vdash_{\Nda }v: C\supset B$ and 
% $\TT C': A\vdash^{\Gamma}_{\Nda} C$, and we then have 
%
%
%  \medskip
%  {\centering\small
%  \begin{tabular}{c}
%  $
%\rr{\TT C}\Big [ \rr{\TT{In}}_{A}\big[ \rr{\TT{U}} [\rr{\TT{El}}_{A}[u] ]\big ] \Big ]
%=v\Big ( \rr{\TT C'}\Big [ \rr{\TT{In}}_{A}\big[ \rr{\TT{U}} [\rr{\TT{El}}_{A}[u] ]\big ] \Big ] \Big)
%\stackrel{\text{I.H.}}{\simeq_{\beta}} 
%v
%\Big ( \rr{\TT{In}}_{C\supset B}\big[ \rr{\TT{U}} [\rr{\TT{El}}_{C\supset B}[\TT C'[u]] ]\big ] \Big ) 
%$\\
%$
%\rightsquigarrow_{\beta}
% \rr{\TT{In}}_{ B}\big[ \rr{\TT{U}} [\rr{\TT{El}}_{ B}[\TT C'[u]v] ]\big ] 
% =
%\rr{\TT{In}}_{ B}\big[ \rr{\TT{U}} [\rr{\TT{El}}_{ B}[\TT C[u]] ]\big ] 
%$
%\end{tabular}
%}

\item if $\TT C= \TT C'W$, then $\TT C': A\vdash^{\Gamma} \forall Y. B'$, $B=B'\llbracket W/Y\rrbracket$; 
by noticing that $\at{\forall Y.B'}\llbracket W/Y\rrbracket=\at{B}$ we then
  have 

  \medskip
  {\centering\small
  \begin{tabular}{c}
  $
\rr{\TT C}\Big [ \rr{\TT{In}}_{A}\big\{ \rr{\TT{M}_{\at{A}}}\{\langle\rr{\TT{El}}_{A}[u_{a}]\rangle_{a\in\mathcal A} \}\big \} \Big ]
= \rr{\TT C'}\Big [ \rr{\TT{In}}_{A}\big\{ \rr{\TT{M}_{\at{A}}} \{\langle \rr{\TT{El}}_{A}[u_{a}]\rangle_{a\in\mathcal A} \}\big \} \Big ]W
$ \\
$
\stackrel{\text{I.H.}}{\simeq_{\beta}}
\Big ( \rr{\TT{In}}_{\forall Y. B'}\big\{ \rr{\TT{M}_{\at{\forall Y.B'}}} \{\langle\rr{\TT{El}}_{\forall Y. B'}[\TT C'[u_{a}]]\rangle_{a\in\mathcal A} \}\big \} \Big ) W
$ \\
$=
\Big ( \Lambda Y.\rr{\TT{In}}_{ B'}\big\{ \rr{\TT{M}_{\at{\forall Y.B'}}} \{\langle\rr{\TT{El}}_{ B'}[\TT C'[u_{a}]Y]\rangle_{a\in\mathcal A} \}\big \} \Big ) W
\rightsquigarrow_{\beta}
 \rr{\TT{In}}_{ B}\big\{ \rr{\TT{M}_{\at{B}}} \{\langle \rr{\TT{El}}_{ B}[\TT C'[u_{a}]W]\rangle_{a\in\mathcal A} \}\big \} 
 $\\
 $
=
\rr{\TT{In}}_{ B}\big\{ \rr{\TT{M}_{\at{B}}} \{\langle \rr{\TT{El}}_{ B}[\TT C[u_{a}]] \rangle_{a\in\mathcal A}\}\big \} 
$
\end{tabular}
}

\item if $\TT C= \lambda w. \TT C'$, then $B=B_{1}\supset B_{2}$ and $\TT C': A\vdash^{\Gamma, w:B_{1}} B_{2}$, so by noticing that $\at{B}=\at{B_{2}}$ we have

  \medskip
  {\centering\small
  \begin{tabular}{c}
  $
\rr{\TT{In}}_{ B}\Big\{ \rr{\TT{M}_{\at{B}}} \{\langle\rr{\TT{El}}_{ B}[\TT C[u_{a}]] \rangle_{a\in\mathcal A}\}\Big \} 
=
\lambda w'.\rr{\TT{In}}_{ B_{2}}\big\{\rr{\TT{M}_{\at{B}}} \{\langle\rr{\TT{El}}_{ B_{2}}[(\lambda w.\TT C'[u_{a}])w'] \rangle_{a\in\mathcal A}\}\big \} 
$\\
$
\rightsquigarrow_{\beta}
\lambda w'.\rr{\TT{In}}_{ B_{2}}\Big\{ \rr{\TT{M}_{\at{B}}} \{\langle\rr{\TT{El}}_{ B_{2}}[ \TT C' \llbracket w'/w\rrbracket[u_{a}]] \rangle_{a\in\mathcal A}\}\Big\} 
%=
%\lambda w'.\rr{\TT{In}}_{ B_{2}}\Big\{ \rr{\TT{M}_{\at{B}}} \{\langle\rr{\TT{El}}_{ B_{2}}[ \TT C'\llbracket w'/w\rrbracket [u_{a}]] \rangle_{a\in\mathcal A}\}\Big\} 
%$\\
%$
\stackrel{\text{I.H.}}{\simeq_{\beta}}
\lambda w'.\rr{\TT C'}\llbracket w'/w\rrbracket\Big [ \rr{\TT{In}}_{A}\big\{\rr{\TT{M}_{\at{A}}} \{\langle\rr{\TT{El}}_{A}[u_{a}]\rangle_{a\in\mathcal A} \}\big \} \Big ]
$\\$
=
\TT C\Big[\rr{\TT{In}}_{A}\big\{ \rr{\TT{M}_{\at{A}}} \{\langle\rr{\TT{El}}_{A}[u_{a}]\rangle_{a\in\mathcal A} \}\big \} \Big ]
$
\end{tabular}
}
%
%\noindent
%where $s^{w'}$ is shorthand for $s\llbracket w'/w\rrbracket$.

\item if $\TT C= \Lambda W. \TT C'$, then $B=\forall W.B'$ and $\TT C': A\vdash^{\Gamma} B'$, so by noticing that $\at{B}=\at{B'}$ we have

  \medskip
  {\centering\small
  \begin{tabular}{c}
  $
\rr{\TT{In}}_{ B}\Big\{ \rr{\TT{M}_{\at{B}}} \{\langle \rr{\TT{El}}_{ B}[\TT C[u_{a}]]\rangle_{a\in\mathcal A} \}\Big\} 
=
\Lambda W'.\rr{\TT{In}}_{ B'}\Big\{ \rr{\TT{M}_{\at{B}}} \{\langle \rr{\TT{El}}_{ B'}[(\Lambda W.\TT C'[u_{a}])W'] \rangle_{a\in\mathcal A}\}\Big \} 
$\\
$
\rightsquigarrow_{\beta}
\Lambda W'.\rr{\TT{In}}_{ B'}\Big\{ \rr{\TT{M}_{\at{B}}} \{\langle \rr{\TT{El}}_{ B'}[\TT C'\llbracket W'/W\rrbracket[u_{a}]]
\rangle_{a\in\mathcal A} \}\Big \} 
%=
%\Lambda W'.\rr{\TT{In}}_{ B'}\Big\{ \rr{\TT{M}_{\at{B}}} \{\langle \rr{\TT{El}}_{ B'}[\TT C'\llbracket W'/W\rrbracket[u_{a}]]
%\rangle_{a\in\mathcal A} \}\Big \} 
%$
%\\
%$
\stackrel{\text{I.H.}}{\simeq_{\beta}}
\Lambda W'.\rr{\TT C'}\llbracket W'/W\rrbracket\Big [ \rr{\TT{In}}_{A}\big\{ \rr{\TT{M}_{\at{A}}} \{\langle \rr{\TT{El}}_{A}[u_{a}]\rangle_{a\in\mathcal A} \}\big \} \Big ]
$\\ $
=
\TT C\Big [ \rr{\TT{In}}_{A}\big\{ \rr{\TT{M}_{\at{A}}} \{\langle\rr{\TT{El}}_{A}[u_{a}]\rangle_{a\in\mathcal A} \}\big \} \Big ]
$
\end{tabular}
}
%
%\noindent
%where $s^{W'}$ is shorthand for $s\llbracket W'/W\rrbracket$.
\end{itemize}
\end{proof}

%
%We indicate by $(\rr{\TT{El}}^{\rr{\TT E}}_{A}, \rr{\TT{In}}^{\rr{\TT E}}_{A} )$ the expansion pair induced by the elimination context $\rr{\TT E}$.
%
%\begin{lemma}\label{lemmett3} $\bb{\lambda \lost{x_j}{}{}.\rr{\TT{El}}_{(\dagger\lost{A_j}{}{})^*}^{\lost{z_k}{}{}}[(\iota_{\dagger}^{k}\lost{x_j}{}{})^*]} \rightsquigarrow_{\beta\eta} z_k$ 
%\end{lemma}
%
%\begin{lemma}\label{lemmett4} $\bb{\rr{\TT{In}}_{(\dagger\lost{A_j}{}{})^*}^{\lost{z_k}{}{}}[tX \lost{z_k}{}{}]} \rightsquigarrow_{\beta\eta}  \bb{t}$ \PP{problema: $\rr{\TT{In}}$ deve essere $\Lambda X...$ e non $\Lambda Y.$}
%\end{lemma}
%
%
%\begin{lemma}
%For any elimination context $\rr{\TT E}$, 
%$(\delta_{\dagger}(t, \langle \langle y_{y}\rangle. \rr{\TT E}[s_{k}]\rangle))^{\sharp}=
%(\delta_{\dagger}(t, \langle \langle y_{y}\rangle. s_{k}\rangle))^{\sharp}$.
%
%\end{lemma}
%\begin{proof}
%
%
%\end{proof}

\begin{proposition}\label{ff-gamma-pres}For all $u,v$, if  $\Gamma\vdash_{\Np} u\rightsquigarrow_{\gamma}v:A$, then  %(1) $\Gamma^*\vdash_{\Nda} u^{*\downarrow^{\varepsilon}}\rightsquigarrow_{\beta} v^{*\downarrow^{\varepsilon}}:A^*$;  (2)
  $\Gamma^*\vdash_{\Nda} u^{\sharp}\simeq_{\beta} v^{\sharp}:A^*$.
\end{proposition}

\begin{proof}
  By induction on the typing derivation $\D D$ of $\Gamma\vdash_{\Np} u\rightsquigarrow_{\gamma}v:A$. If $\D D$ ends with an application of either reflexivity, transitivity or of one of the congruence rules, it is enough to apply the induction hypothesis to the derivations of the premises, and then applying the reflexivity, transitivity or  the  congruence rules.

  If $\D D$ ends with an application of $\dagger\gamma$, $ u=\rr{\TT E}[\bb{\delta_{\dagger}(t,\lost{\lost{y_j}{}{}.s_{k}}{}{})}]$ and $v= \bb{\delta_{\dagger}(t,\lost{\lost{y_j}{}{}.\rr{\TT E}[s_{k}]}{}{})}$ for some elimination context $\rr{\TT E}:B\vdash^{\Gamma} A$ (hence $\Gamma\vdash_{\Np} \bb{\delta_{\dagger}(t,\lost{\lost{y_j}{}{}.s_{k}}{}{})}: B$ and $\Gamma\vdash_{\Np} t:\dagger\lost{A_i}{}{}$ for some~$\dagger$).  We show how to construct a derivation $\D D'$ of %$\Gamma^*\vdash_{\Nda} u^{*\downarrow^{\varepsilon}}\rightsquigarrow_{\beta} v^{*\downarrow^{\varepsilon}}:A^*$ (respectively
  $\Gamma^*\vdash_{\Nda}   u^{\sharp} \simeq_{\beta} 
  v^{\sharp}:A^*$). 
  Let $Z=\at{A^{*}}$ and $Z'=\at{B^{*}}$.   By Remark~\ref{remhelp}
  we know that $\TT E^{\sharp}$ is a principal context and that $\TT E^{\sharp}\{t\}=\TT E^{\sharp}[t]$. Moreover, consider the $\mathcal K$-multicontext $\TT M: \langle\Delta_{k}\vdash X\rangle_{k\in \mathcal K } \Rightarrow( \Gamma\vdash X)$ 
% (with $X$ being any variable not occurring in $\Gamma$) 
given by 
$\TT M= (u^{\sharp}X) \big\langle  \lambda \langle y_{j}\rangle_{j\in g^{-1}(k)}. [\  ]_{k}\big \rangle_{k\in \mathcal K}$,
%\footnote{We are here implicitly identifying $\mathcal K$ with the set $\{1,\dots, |\mathcal K|\}$.}, 
where $\Delta_{k}= \langle y_{j}: A_{f(j)}\rangle_{j\in g^{-1}(k)}$ and $X$ is a fresh variable. 
The reader can check that all conditions of Lemma~\ref{lemmett1} do hold with $\TT E^{\sharp}$ in place of $\TT C$, $A^{*},B^{*}$ in place of $A,B$ and the $s_{k}^{\sharp}$ in place of the $u_{a}$. Then we have:

  \medskip

%\lambda \lost{y_k}{k}{\C K}. \bb{\rr{\TT{In}}_{B}}\left[t^{\downarrow^{\varepsilon}}Z%^{(\dagger(A_i))^*}
 % \lost{\Fun{X}{\left(\EXPO{X}{j\in g^{-1}(k)  }{A_{f(j)}^{\phantom{fj}*}}\right)}\!\!\!(\bb{\rr{\TT{El}}_B})[y_k]}{k}{\C K}\right]
  
%   {\centering\small
%  $\bb{ u^{*\downarrow^{\varepsilon}}}
%    \stackrel{\text{Def.~\ref{rp-lambda}$+$\ref{ateps-lambda}}}{=}
%    \bb{\rr{\TT E}^* \left[   \left(\lambda \lost{z_k}{k}{\C K}. \bb{\rr{\TT{In}}_{B}}\left[t^{*\downarrow^{\varepsilon}}Z%^{(\dagger(A_i))^*}
%   \lost{\Fun{X}{\left(\EXPO{X}{j\in g^{-1}(k)  }{A_{f(j)}^{\phantom{fj}*}}\right)}\!\!\!(\bb{\rr{\TT{El}}_B})[z_k]}{k}{\C K}\right]\right) \lost{\lambda\lost{y_j}{j}{g^{-1}(k)}.s_k^{*\downarrow^{\varepsilon}}}{k}{\C K} \right]} 
%      \stackrel{\text{Lemma~\ref{lemmett1}}}{\rightsquigarrow_{\beta}}
% %   \bb{\rr{\TT{In}}_{A^*}^{\rr{\TT E}}  [  (t^{\sharp}Z)
% %\lost{\lambda \lost{y_j}{j}{g^{-1}(k)}.\rr{\TT{El}}_{B^*} [s_{k}^{\sharp}]}{k}{\C K} ]} =
% $
% $
% \stackrel{\text{Lemma~\ref{lemmett1}}}{\rightsquigarrow_{\beta}}
%      \bb{\rr{\TT{In}}_{A^*} [  (t^{\sharp}Z)
% \lost{\lambda \lost{y_j}{j}{g^{-1}(k)}.\rr{\TT{El}}_{A^*} \left[\rr{\TT E}^*[s_{k}^{\sharp}]\right]}{k}{\C K} }  \stackrel{\text{Def.~\ref{reftrans}}}{=} \bb{v^{\sharp}}
% $

% }

%\bigskip
  {\centering\small
\begin{tabular}{c} $\bb{ u^{\sharp}}
   \stackrel{\text{Def.~\ref{reftrans}}}{=}
   \bb{\rr{\TT E}^\sharp \left[ \rr{\TT{In}}_{B^*}  \left\{  (u^{\sharp}Z')  \lost{\lambda \lost{y_j}{j}{g^{-1}(k)}.\rr{\TT{El}}_{B^*} [s_{k}^{\sharp}]}{k}{\C K}\right \} \right]} 
%\stackrel{\text{Rem.~\ref{remhelp}}}{=}
%   \bb{\rr{\TT E}^\sharp \left[ \rr{\TT{In}}_{B^*}  \left\{  (u^{\sharp}Z')  \lost{\lambda \lost{y_j}{j}{g^{-1}(k)}.\rr{\TT{El}}_{B^*} [s_{k}^{\sharp}]}{k}{\C K}\right \} \right] } 
    % \stackrel{\text{Lemma~\ref{lemmett1}}}{\simeq_{\beta}}
%   \bb{\rr{\TT{In}}_{A^*}^{\rr{\TT E}}  [  (t^{\sharp}Z)
%\lost{\lambda \lost{y_j}{j}{g^{-1}(k)}.\rr{\TT{El}}_{B^*} [s_{k}^{\sharp}]}{k}{\C K} ]} =
$\\
$
\stackrel{\text{Lemma~\ref{lemmett1}}}{\simeq_{\beta}}
     \rr{\TT{In}}_{A^*} \left \{  (u^{\sharp}Z)
\lost{\lambda \lost{y_j}{j}{g^{-1}(k)}.\rr{\TT{El}}_{A^*} \left[\rr{\TT E}^\sharp[s_{k}^{\sharp}]\right]}{k}{\C K}  \right \}\stackrel{\text{Def.~\ref{reftrans}}}{=} \bb{v^{\sharp}}
$
\end{tabular}%$\Large)$

}

%\PP{non mi tornano gli ultimi passaggi con 
%\medskip
%
%\PP{bisogna dire che cosa e'
%  $\rr{\TT E}^*$?? ci sono dei lemmini nascosti?}

\end{proof}

\begin{proposition}\label{ff-eta-pres}For all $u,v$ if $\Gamma\vdash_{\Np} u\rightsquigarrow_{\eta} v:A$ then $\Gamma^*\vdash_{\Nda} u^{\sharp}\rightsquigarrow_{\beta\eta} v^{\sharp}:A^*$.
\end{proposition}

\begin{proof}
  By induction on the typing derivation $\D D$ of $\Gamma\vdash_{\Np} u\rightsquigarrow_{\eta}v:A$. If $\D D$ ends with an application of either reflexivity, transitivity,  one of the congruence rules or $\imp\eta$  it is enough to apply the induction hypothesis to the derivations of the premises, and then applying reflexivity, transitivity, the  congruence rules and $\imp\eta$ or $\forall\eta$.

  If $\D D$ ends with an application of $\dagger\eta$,   $ u=\bb{\delta_{\dagger}(v,\lost{\lost{y_j}{}{}.\iota_{\dagger}^{k}\lost{y_j}{}{} }{}{})}$  (where %$\Gamma\vdash_{\Np} \bb{\delta_{\dagger}(t,\lost{\lost{y_j}{}{}.s_{k}}{})}: B$ and
  $\Gamma\vdash_{\Np} v:\dagger\lost{A_i}{}{}$ for some~$\dagger$), and we informally show how to construct a derivation $\D D'$ of $\Gamma^*\vdash_{\Nda}   u^{\sharp} \rightsquigarrow_{\beta\eta} 
  v^{\sharp}:A^*$. 
By letting $\rr{\TT C}= \Lambda X.\lambda \langle x_{h}\rangle_{h\in \C K}.[\ ]$ and 
$\rr{\TT{In}}_{A^*}$ be $\Lambda Z.\lambda \langle z_{h}\rangle_{h\in \C K}.[\ ]$, we have that:  
  
%  Let $Z$ be the rightmost variable in  $A^*(=(\dagger\lost{A_i}{}{})^*$ \PP{male perche Z e' vincolata, ovvero non ha propriamente un nome...}, we have that:

  \medskip
  {\centering\small
    $\bb{ u^{\sharp}} \stackrel{\text{Def.~\ref{reftrans}}}{=}
    \rr{\TT{In}}_{A^*} \left\{  (v^{\sharp}Z)%]
     \lost{\lambda \lost{y_j}{}{}.\rr{\TT{El}}_{A^*}\big [\rr{\TT C}\{x_{k}\lost{y_j}{}{}\} \big] }{}{}_{k\in \C K}
    \right\}   
\rightsquigarrow_{\beta}$ \\
\ \\
$\rightsquigarrow_{\beta}
    \rr{\TT{In}}_{A^*} \left\{  (v^{\sharp}Z)%]
     \lost{\lambda \lost{y_j}{}{}.z_{h}\lost{y_j}{}{}}{}{}_{h\in \C K}
    \right\}
    \rightsquigarrow_{\eta}
%
%$
%
%
%$
%\stackrel{\text{Lemma~\ref{lemmett3}}}{\rightsquigarrow_{\beta\eta}}
   \rr{\TT{In}}_{A^*} \left\{ (v^{\sharp}Z)%]
     \lost{z_{h} }{}{}_{h\in \C K}
   \right\}   \rightsquigarrow_{\eta}
  v^{\sharp}$

%   $\stackrel{\text{Lemma~\ref{lemmett4}}}{\rightsquigarrow_{\beta\eta}}    v^\sharp$
 
    % \bb{\rr{\TT{In}}_{A^*}  [  (t^{\sharp}Z)
% \lost{\lambda \lost{y_j}{j}{g^{-1}(k)}.\rr{\TT{El}}_{B^*} [s_{k}^{\sharp}]}{k}{\C K} ]} =
% $
% $\stackrel{\text{Lemma~\ref{lemmett1}}}{=} \bb{\rr{\TT{In}}_{A^*}  [  (t^{\sharp}Z)
% \lost{\lambda \lost{y_j}{j}{g^{-1}(k)}.\rr{\TT{El}}_{A^*} \left[\rr{\TT E}^*[s_{k}^{\sharp}]\right]}{k}{\C K} }  \stackrel{\text{Def.~\ref{reftrans}}}{=} \bb{v^{\sharp}}

}

\end{proof}

\begin{proof}[Proof of Proposition \ref{ff-preservation2}]
It suffices to check rule-by-rule that reductions are preserved. The only non-trivial cases are those of  $\gamma$ and $\eta$, which can be treated using Proposition \ref{ff-gamma-pres} and Proposition \ref{ff-eta-pres}.
%The proposition 
%follows by the previous one together with Propositions 15, 13 and 12. \PP{Check??}
\end{proof}

\section{Proof of Proposition~\ref{rp-ffe-relation}}\label{rp-ffe-appendix}

We prove the following generalization of Proposition~\ref{rp-ffe-relation} to  $\Ndrp$:

\begin{proposition}\label{rp-ffe-relation-proof} For all $u$ such that  $\Gamma\vdash_{\Ndrp}u:A$  we have that   $\Gamma\vdash_{\Ndrp}u \simeq_{\eta\varepsilon} u^{\downarrow^{\varepsilon}}:A$.
\end{proposition}

\begin{proof}
  By induction on the typing derivation $\D D$ of $\Gamma\vdash_{\Ndrp} u:A$. If $\D D$ ends with an application of either Ax, $\supset$I, $\supset$E, $\forall$I or one of $\forall$E with an atomic witness, (i.e.~$u$ is  either a variable or of the form $\lambda x.t$, $ts$, $\Lambda X.t$ or $tY%^{(\dagger(A_i))^*}
    $), then it is enough to use reflexivity or to  apply the induction hypothesis to the immediate sub-derivations of $\D D$ and use the congruence rules.

    Otherwise $u=tB%^{(\dagger(A_i))^*}
    $  with $B$ non atomic, where the same remarks as in Def.~\ref{at-lambda} apply to $t$ (in particular $\Gamma\vdash_{\Ndrp}t:\forall W.F$), and we informally show  how to construct a derivation $\D D'$ of $\Gamma\vdash_{\Nda} u \simeq_{\eta\varepsilon}
  u^{\downarrow^{\varepsilon}}:A$.
%  The proof is by induction on $u$. If $u$ is  either of the forms $\lambda x.t$, $ts$, $\Lambda X.t$ or $tY%^{\forall X.A}
%  $, then it is enough to apply the induction hypothesis to the immediate sub-terms of $t$. 
%
% 
%  If $u=tB%^{\forall X.A}
%  $  (where the same remarks as in Def.~\ref{at-lambda} apply on $t$, in particular $\Gamma\vdash_{\Ndrp}t:\forall W.F$, and where 
Assuming  $Z=\at{B}$ (so that $\rr{\TT{El}}_{B}: B\vdash^{\Gamma}Z$) we have that: %with $F=A_1\imp, \ldots\imp A_k\imp Z$ with all $A_i$ $SP_X$ we have that

\medskip
{  \centering\small$\begin{array}{r@{}c@{}l@{}l}
       tB%^{(\forall X.A)}
       &\stackrel{{\text{Fact~\ref{etainel}}}}{\simeq_{\eta}} &\bb{\rr{\TT{In}}_{F\llbracket B/W\rrbracket}} \{ \bb{\rr{\TT{El}}_{F\llbracket B/W\rrbracket}} [tB%^{(\forall X.A)}
                                                                   ]\}%\\
       
                                                                   % &
                                                                       \quad\stackrel{\text{Fact~\ref{expa2}}}{=}\quad \lambda\lost{y_k}{k}{\C K}.\bb{\rr{\TT{In}}_{B} } \left \{ \bb{\rr{\TT{El}}_{B} } [(tB)%^{(\forall X.A)}
                                                                       \lost{y_k}{k}{\C K}]\right\} &\\
       
 &      {\simeq_{\varepsilon}}&
                                   \lambda \lost{y_k}{k}{\C K}.\bb{\rr{\TT{In}}_{B} } \left \{ (tZ)%^{(\forall X.A)}
                                   \lost{\bb{\Fun{Z}{(\EXPO{Z}{j\in g^{-1}(k)  }{A_{f(j)}^{\phantom{fj}*}})}{\rr{\TT{El}}_{B} }}[y_k]}{k}{\C K}\right\} \quad\stackrel{\text{Def.~\ref{ateps-lambda}+I.H.}}{\simeq_{\eta\varepsilon}}\quad (tB)^{\downarrow^{\varepsilon}}&\\
     \end{array}
     $

}

\medskip
\noindent where in the penultimate step the rule $\varepsilon$ is applied to the term in braces by taking $\rr{\TT{El}}_{B}$ for $\TT U$. 
\end{proof}

\section{Proof of Proposition~\ref{ffe-esf-relation}}\label{espirito-appendix}
We prove the following generalization of Proposition~\ref{ffe-esf-relation} to  $\Np$:
% %\section{Refining the FF$^{\varepsilon}$-translation}
% Let $\D D= {\small\AXC{$\D D_1$}\noLine\UIC{$\bullet(A_1,A_2,A_3)$}\AXC{$[A_1][A_2]$}\noLine\UIC{$\D D_2$}
% \noLine\UIC{$C$}\AXC{$[A_3]$}\noLine\UIC{$\D D_3$}\noLine\UIC{$C$}\RL{$\bullet$E}\TIC{$C$}\DP}$ be an $\Nb$-derivation ending with an application of $\bullet$E. Whenever $C$ is not atomic, by first  RP-translating  and then $\forall$E-atomizing$^{\varepsilon}$ $\D D$ one necessarily introduces a new $\beta$-redex (namely, the lowest occurrence of $\curc(A_1,A_2,A_3)\llbracket C/X\rrbracket$ in the second derivation below):

% \noindent\resizebox{\textwidth}{!}{

\begin{proposition}For all $u$ such that $\Gamma\vdash_{\Np} u:A$:  (1) $\Gamma^*\vdash_{\Nda} u^{*\downarrow}\rightsquigarrow_{\beta} u^{*\downarrow^\varepsilon}:A^*$;\linebreak (2)  $\Gamma^*\vdash_{\Nda} u^{*\downarrow^\varepsilon}\rightsquigarrow_{\beta} u^{\sharp}:A^*$.
\end{proposition}

The proposition follows immediately from the following two:

\begin{proposition}\label{ff-ffe-relation-lambda}For all $u$  such that $\Gamma\vdash_{\Ndrp} u:A$ we have that:  $\Gamma\vdash_{\Nda} u^{\downarrow} \rightsquigarrow_{\beta}
  u^{\downarrow^{\varepsilon}}:A$.
  \end{proposition}

  \begin{proof} By induction on the typing derivation $\D D$ of $\Gamma\vdash_{\Ndrp} u:A$. If $\D D$ ends with an application of either Ax, $\supset$I, $\supset$E, $\forall$I or one of $\forall$E with an atomic witness, (i.e.~$u$ is  either a variable or of the form $\lambda x.t$, $ts$, $\Lambda X.t$ or $tY%^{(\dagger(A_i))^*}
    $), then it is enough to use reflexivity or to  apply the induction hypothesis to the immediate sub-derivations of $\D D$ and use the congruence rules.

    Otherwise $u=tB%^{(\dagger(A_i))^*}
    $  with $B$ non atomic, where the same remarks as in Def.~\ref{at-lambda} apply to $t$, and we informally show  how to construct a derivation $\D D'$ of $\Gamma\vdash_{\Nda} u^{\downarrow} \rightsquigarrow_{\beta}
  u^{\downarrow^{\varepsilon}}:A$ by induction on $B$.
    % \begin{itemize}\item

    If $B=C\supset D$,   assuming $Z=\at{B}$, we have: %The lengthy computation is postponed to
   % Appendix~\ref{prova-app}.%;

    \medskip
{  \small\centering
$\begin{array}{r@{}c@{}l}
   (tB)^\downarrow &\stackrel{\text{Def.~\ref{at-lambda}+I.H}}{\rightsquigarrow_{\beta}} &\lambda \lost{v_k}{k}{\C K}\lambda u. ((t D)^{\downarrow^{\varepsilon}}%^{(\dagger(A_i))^*}
                                                                                ) \lost{\lambda \lost{x_j}{j}{g^{-1}(k)}.v_k\lost{x_j}{j}{g^{-1}(k)}u}{k}{\C K}\\
                & \stackrel{\text{Def.~\ref{ateps-lambda}}}{=}& \lambda\lost{v_k}{k}{\C K}\lambda u. (\lambda \lost{y_k}{k}{\C K}. \bb{\rr{\TT{In}}}_{D}  \left \{(t^{\downarrow^{\varepsilon}} Z)%^{(\dagger(A_i))^*}
                                                                \lost{\Fun{Z}{(\EXPO{Z}{j\in g^{-1}(k)  }{A_{f(j)}^{\phantom{fj}*}})}{\bb{\rr{\TT{El}}}_{D} }[y_k]}{k}{\C K} \right\})%\\&&\qquad\qquad\qquad\qquad\qquad\qquad\qquad\qquad\qquad\qquad\qquad\qquad
                                                                                                                                                                                                                   \lost{\lambda \lost{x_j}{j}{g^{-1}(k)}.v_k\lost{x_j}{j}{g^{-1}(k)}u}{k}{\C K}\\
             &\rightsquigarrow_{\beta}&\lambda\lost{v_k}{k}{\C K}\lambda u. \bb{\rr{\TT{In}}}_{D}  \left\{(t^{\downarrow^{\varepsilon}} Z)\lost{\Fun{Z}{(\EXPO{Z}{j\in g^{-1}(k)  }{A_{f(j)}^{\phantom{fj}*}})}{\bb{\rr{\TT{El}}}_{D}}[\lambda \lost{x_j}{j}{g^{-1}(k)}.v_k\lost{x_j}{j}{g^{-1}(k)}u]}{k}{\C K} \right\}\\
                &\stackrel{\text{Def}.~\ref{cexp-lambda}}{=}&\lambda\lost{v_k}{k}{\C K}\lambda u. \bb{\rr{\TT{In}}}_{D}  \left\{(t^{\downarrow^{\varepsilon}} Z%^{(\dagger(A_i))^*}
                                                       ) 
\lost{\lambda \lost{w_j}{j}{g^{-1}(k)}.\bb{\rr{\TT{El}}}_{D} [(\lambda \lost{x_{j}}{j}{g^{-1}(k)}.v_k\lost{x_{j}}{j}{g^{-1}(k)}u)\lost{w_j}{j}{g^{-1}(k)}]}{k}{\C K}\right\}\\
                &\rightsquigarrow_{\beta}&\lambda\lost{v_k}{k}{\C K}\lambda u. \bb{\rr{\TT{In}}}_{D}  \left \{(t^{\downarrow^{\varepsilon}} Z%^{(\dagger(A_i))^*}
                                                     ) \lost{\lambda \lost{w_j}{j}{g^{-1}(k)}.\bb{\rr{\TT{El}}}_{D} [v_k\lost{w_j}{j}{g^{-1}(k)}u]}{k}{\C K}\right\}\\
                &\stackrel{\text{Fact}~\ref{alphaexpa}}{=}&\lambda\lost{v_k}{k}{\C K}. \bb{\rr{\TT{In}}}_{C\imp D}  \left \{(t^{\downarrow^{\varepsilon}} Z%^{(\dagger(A_i))^*}
                                                      ) \lost{\lambda \lost{w_j}{j}{g^{-1}(k)}.\bb{\rr{\TT{El}}}_{C\imp D} [v_k\lost{w_j}{j}{g^{-1}(k)}]}{k}{\C K}\right\}
    \\
    
           &\stackrel{\text{Def}.~\ref{cexp-lambda}}{=}&\lambda\lost{v_k}{k}{\C K}. \bb{\rr{\TT{In}}}_{C\imp D}  \left \{(t^{\downarrow^{\varepsilon}} Z%^{(\dagger(A_i))^*}
) \lost{\Fun{Z}{(\EXPO{Z}{j\in g^{-1}(k)  }{A_{f(j)}^{\phantom{fj}*}})}{\bb{\rr{\TT{El}}}_{C\imp D} }[v_k]}{k}{\C K}\right\}\qquad\stackrel{\text{Def.~\ref{ateps-lambda}}}{=}\qquad (tB)^{\downarrow^{\varepsilon}}
  \end{array}
$

}

\medskip

  % \item
One can argue in a similar way if $B=\forall U. D$ (just replace $C\imp D$ with $\forall U. D$ and $u$ with $U$).
%  \end{itemize}
\end{proof}

\begin{proposition}For all $u$ such that $\Gamma\vdash_{\Np} u:A$: $\Gamma^*\vdash_{\Nda} u^{*\downarrow^\varepsilon}\rightsquigarrow_{\beta} u^{\sharp}:A^*$.
\end{proposition}

\begin{proof}
  By induction on the typing derivation $\D D$ of $\Gamma\vdash_{\Np} u:A$. If $\D D$ ends with an application of either Ax, $\supset$I, $\supset$E or  $\dagger$I$_k$%or one of $\forall$E with an atomic witness
  , (i.e.~$u$ is  either a variable or of the form $\lambda x.t$, $ts$ or $\iota_{\dagger}^{k}\lost{t_j}{}{}$), then it is enough to use reflexivity or to apply the induction hypothesis to the immediate sub-derivations of $\D D$ and use the congruence rules.

  Otherwise $u=\delta_{\dagger}(t, \lost{\lost{y_j}{}{}.s_{k}}{}{})$ and hence (by Def.~\ref{rp-lambda}) $u^*=(t^{\rop}A^{\rop})\lost{\lambda \lost{y_j}{j}{g^{-1}(k)}.s_{k}^{\rop}}{k}{K}$, where the same remarks as in Def.~\ref{at-lambda} apply to $t^{\rop}$, and we informally show  how to construct a derivation $\D D'$ of $\Gamma\vdash_{\Nda}   u^{*\downarrow^{\varepsilon}} \rightsquigarrow_{\beta}
  u^{\sharp}:A$.

  If $A$ is atomic, i.e.~$A=Y$ we have that:

\medskip
{\centering\small
   $\begin{array}{c}\bb{ (\delta_{\dagger}%(A_i)\Rightarrow C}
       (t, \lost{\lost{y_j}{}{}.s_{k}}{}{}))^{\rop\downarrow^{\varepsilon}}}  \stackrel{\text{Def.~\ref{rp-lambda}}}{=}\bb{ ( (t^{\rop}Y)
      \lost{\lambda \lost{y_j}{j}{g^{-1}(k)}.s_{k}^{\rop}}{k}{K}}{}{})^{\downarrow^{\varepsilon}} \stackrel{\text{Def.~\ref{ateps-lambda}}}{=}      \bb{  (t^{\rop\downarrow^{\varepsilon}}Y)
      \lost{\lambda \lost{y_j}{j}{g^{-1}(k)}.s_{k}^{\rop\downarrow^{\varepsilon}}}{k}{K}}{}{} \stackrel{\text{I.H.}}{\rightsquigarrow_{\beta}}   \\[1ex]
    \bb{  (t^{\sharp}Y)
      \lost{\lambda \lost{y_j}{j}{g^{-1}(k)}.s_{k}^{\sharp}}{k}{K}}{}{} \stackrel{\text{Def.~\ref{ie-context}}}{=} \bb{\rr{\TT{In}}_{Y} }\left \{\bb{  (t^{\sharp}Y)
      \lost{\lambda \lost{y_j}{j}{g^{-1}(k)}.\bb{\rr{\TT{El}}_{Y} }[s_{k}^{\sharp}]}{k}{K}}{}{}\right\}\stackrel{\text{Def.~\ref{reftrans}}}{=} u^{\sharp}\\[1ex]
    \end{array}$
    
  }

  \medskip
 % The proof is by induction on $u$. If $u$ is  either of the forms $\lambda x.t$ or $ts$ or $\iota_{\dagger}^{k}\lost{t_j}{}{}$%, $\Lambda X.t$ or $tY$
 % , then it is enough to apply the induction hypothesis to the immediate sub-terms of $u$. 
%  If $u=\delta_{\dagger}
%(A_i)\Rightarrow C}
%      (t, \lost{\lost{y_j}{}{}.s_{k}}{}{})$, %, where the the same remarks as in Def.~\ref{at-lambda} apply to $t$, %). given that $\forall X.A=(\dagger^{f,g}(A_i))^*= \forall X. \EXP{X}{ i\in \C K   }{ \EXP{X}{j\in g^{-1}(k)  }{A_{f(j)}^{\phantom{fj}*}}}$ for some $\dagger^{f,g}$ and
%  assuming $Z$ to be the rightmost variable in $A$ we have that %with $F=A_1\imp, \ldots\imp A_k\imp Z$ with all $A_i$ $SP_X$ we have that

Otherwise, assuming $Z=\at{A^*}$ we have that: %with $F=A_1\imp, \ldots\imp A_k\imp Z$ with all $A_i$ $SP_X$ we have that

   \medskip
{\centering\small
   $\begin{array}{l}\bb{ (\delta_{\dagger}%(A_i)\Rightarrow C}
       (t, \lost{\lost{y_j}{}{}.s_{k}}{}{}))^{\rop\downarrow^{\varepsilon}}}  \stackrel{\text{Def.~\ref{rp-lambda}}}{=}\bb{ ( (t^{\rop}A^{\rop%(\dagger(A_i)^{\rop}
       })
       \lost{\lambda \lost{y_j}{j}{g^{-1}(k)}.s_{k}^{\rop}}{k}{K}}{}{})^{\downarrow^{\varepsilon}} \stackrel{\text{Def.~\ref{ateps-lambda}}}=\\[1ex]
       = \bb{\lambda \lost{z_k}{k}{\C K}. \bb{\rr{\TT{In}}_{A^*} }\left\{(t^{*\downarrow^{\varepsilon}}Z)%^{(\dagger(A_i))^*}
       \lost{\Fun{\!\!Z}{\left(\EXPO{Z}{j\in g^{-1}(k)  }{A_{f(j)}^{\phantom{fj}*}}\right)}{\bb{\rr{\TT{El}}_{A^*} }}[z_k]}{k}{\C K}\right\} \lost{\lambda \lost{y_j}{j}{g^{-1}(k)}.s_{k}^{\rop\downarrow^{\varepsilon}}}{k}{K}}\rightsquigarrow_{\beta}\\[2ex]
       \rightsquigarrow_{\beta} \bb{\rr{\TT{In}}_{A^*} \left\{(t^{*\downarrow^{\varepsilon}}Z)%^{(\dagger(A_i))^*}
       \lost{\Fun{\!\!Z}{\left(\EXPO{Z}{j\in g^{-1}(k)  }{A_{f(j)}^{\phantom{fj}*}}\right)}{\bb{\rr{\TT{El}}_{A^*} }}[\lambda \lost{y_j}{j}{g^{-1}(k)}.s_{k}^{\rop\downarrow^{\varepsilon}}]}{k}{\C K}\right\} }
    \stackrel{\text{Def.~\ref{cexp-lambda}}}{=}\\[2ex]
       = \bb{ \bb{\rr{\TT{In}}_{A^*} }\left\{(t^{*\downarrow^{\varepsilon}}Z)%^{(\dagger(A_i))^*}
       \lost{\lambda \lost{v_j}{j}{g^{-1}(k)}.\bb{\rr{\TT{El}}_{A^*} }[(\lambda \lost{y_j}{j}{g^{-1}(k)}.s_{k}^{\rop\downarrow^{\varepsilon}} )\lost{v_j}{j}{g^{-1}(k)}]}{k}{\C K}\right\} } \rightsquigarrow_{\beta}
 \\
       \rightsquigarrow_{\beta} \bb{ \bb{\rr{\TT{In}}_{A^*} }\left\{(t^{*\downarrow^{\varepsilon}}Z)%^{(\dagger(A_i))^*}
       \lost{\lambda \lost{y_j}{j}{g^{-1}(k)}.\bb{\rr{\TT{El}}_{A^*} }[s_{k}^{\rop\downarrow^{\varepsilon}}]}{k}{\C K}\right\} } \stackrel{\text{I.H.}}{\rightsquigarrow_\beta} u^{\sharp}
   \end{array}
  $

}

\end{proof}

\section{Proofs of Proposition \ref{rpequiv}, \ref{rpnequiv}, \ref{noio} and \ref{undefdis}}\label{proofprop}
We establish the following generalization of Proposition~\ref{rpequiv} to the whole of $\Ndp$.

\begin{proposition} For all $A\in \Ldp$, 
  $A\dashv\vdash_{\Ndp}A^*$.
\end{proposition}
\begin{proof}
We can construct by induction on $A$ term contexts $\rr{\TT C}:  A^* \vdash^{\emptyset}A$ and  $\rr{\TT D}: A \vdash^{\emptyset}A^*$. The only non trivial case is $A= \dagger\langle A_i\rangle$. In this case we have that:

{\small
$$\rr{\TT C}= [\ ]\dagger\langle A_i\rangle\lost{\lambda \lost{y_j}{j}{g^-1(k)}. \iota^k_{\dagger}\lost{\rr{\TT C}_j[y_j]}{j}{g^-1(k)}   }{k}{\C K}_{k\in \C K}$$
$$\rr{\TT D}= \Lambda X.\delta_{\dagger}([\ ], \lost{\lost{y_j}{j}{g^-1(k)}. \lambda\lust{x_h}{h}{\C K}. x_k \lost{\rr{\TT D}_j[y_j]}{j}{g^-1(k)}   }{k}{\C K}) $$
}

\noindent
where by induction hypothesis $\rr{\TT C}_j:A_{f(j)}^{\phantom{f}*}\vdash^{\emptyset}A_{f(j)}$ and $\rr{\TT D}_j:A_{f(j)}\vdash^{\emptyset}A_{f(j)}^{\phantom{f}*}$

% \begingroup\makeatletter\def\f@size{9}\check@mathfonts
% $$
% %Yon_{A,B}^{1} \ \equiv \ 
% \AXC{$(B\vee C)^{*}$}\RL{$\forall\text{E}$}
% \UIC{$(B^*\supset B\vee C)\supset (C^*\supset B\vee C)\supset B\vee C$}
% \AXC{$\stackrel{1}{[B^*]}$}\noLine\UIC{$\phantom{\text{I.H.}}\vdots\text{I.H.}$}\noLine\UIC{$B$}\RL{$\vee\text{I}$}
% \UIC{$B\vee C$}
% \RL{\small$\mathord{\supset}\text{I}$ $(1)$}
% \UIC{$B^*\supset B\vee C$}\RL{$\mathord{\supset}\text{E}$}
% \BIC{$(C^*\supset B\vee C)\supset B\vee C$}
% \AXC{$\stackrel{2}{[C^*]}$}\noLine\UIC{$\phantom{\text{I.H.}}\vdots\text{I.H.}$}\noLine\UIC{$C$}\RL{$\vee\text{I}$}
% \UIC{$B\vee C$}
% \RL{$\mathord{\supset}\text{I}$ $(2)$}
% \UIC{$C^*\supset B\vee C$}\RL{$\mathord{\supset}\text{E}$}
% \BIC{$B\vee C$}
% \DP
% $$

% $$
% %Yon_{A,B}^{2} \ \equiv \ 
% \AXC{$B\vee C$}
% \AXC{$\stackrel{1}{B^*\supset X}$}
% \AXC{$\stackrel{3}{[B]}$}\noLine\UIC{$\phantom{\text{I.H.}}\vdots\text{I.H.}$}\noLine\UIC{$B^*$}
% \BIC{$X$}
% \UIC{$(C^*\supset X)\supset X$}
% \RL{\small$\mathord{\supset}\text{I}$ $(1)$}
% \UIC{$(B^*\supset X)\supset (C^*\supset X)\supset X$}
% \UIC{$(B\vee C)^{*}$}
% \AXC{$\stackrel{2}{C^*\supset X}$}
% \AXC{$\stackrel{4}{[C]}$}\noLine\UIC{$\phantom{\text{I.H.}}\vdots\text{I.H.}$}\noLine\UIC{$C^*$}
% \BIC{$X$}
% \RL{\small$\mathord{\supset}\text{I}$ $(2)$}
% \UIC{$(C^*\supset X)\supset X$}
% \UIC{$(B^*\supset X)\supset (C^*\supset X)\supset X$}
% \UIC{$(B\vee C)^{*}$}\RL{$(3,4)$}
% \TrinaryInfC{$(B\vee C)^{*}$}
% \DP
% $$
% \endgroup

\end{proof}

To establish Proposition \ref{rpnequiv} we will exploit a sound and complete semantics for $\Ndva$ from \cite{Sol77}, that we briefly recall.

For any partially ordered set $\langle W, \leq\rangle$, let $ W^{\uparrow}=\{ a\subseteq W\mid \forall \alpha, \beta\in W (\alpha\in a \land \alpha\leq \beta \Rightarrow \beta\in a)\}$ the set of upward closed subsets of $W$.

A \emph{$\Ndva$-model} is a tuple $\mathcal M=\langle W, \leq, \bot, D, g\rangle$ such that $\langle W, \leq, \bot\rangle$ is a partially ordered set with a bottom element $\bot$, $D$ is a monotone function from $W$ to $\wp(W^{\uparrow})$ (i.e. $\alpha\leq \beta\Rightarrow D(\alpha)\subseteq D(\beta)$) and $g: \mathcal V \to \bigcup_{\alpha\in W}D(\alpha)$.

For any formula $A\in \C L^{\vee}$ and model $\mathcal M=\langle W, \leq, \bot, D, g\rangle$, for all $\alpha\in W$ such that $g(\C V)\subseteq D(\alpha)$, we let the relation $ \alpha \Vdash_{\mathcal M} A$ be defined by:
\begin{itemize}
\item $\alpha\Vdash_{\mathcal M} X$ iff $\alpha\in g(X)$
\item  $\alpha\Vdash_{\mathcal M} A\supset B$ iff for all $\beta\geq \alpha$, if $\beta\Vdash_{\mathcal M} A$ then $\beta\Vdash_{\mathcal M} B$ 
\item  $\alpha\Vdash_{\mathcal M} A\vee B$ iff $\alpha\Vdash_{\mathcal M} A$ or  $\alpha\Vdash_{\mathcal M} B$
\item $\alpha\Vdash_{\mathcal M} \forall X. A$ iff  for all $\beta\geq \alpha$ and $a\in D(\beta)$, $\beta\Vdash_{\mathcal M[X\mapsto a  ]} A$\end{itemize}  
 where $\mathcal M[X\mapsto a]$ is $\langle W, \leq, \bot, D, g[X\mapsto a]\rangle$, with $g[X\mapsto a]$ being like $g$ except that it sends $X$ to $a$. 

A model $\mathcal M=\langle W, \leq, \bot, D, g\rangle$ 
 is called \emph{regular} if  $g(\mathcal V)\subseteq D(\bot)$. %In a regular model $ \vDash_{\mathcal M}  A$ holds iff $\bot\Vdash_{\mathcal M}A$.
%For any regular model $\mathcal M$, 
%we let $ \vDash_{\mathcal M}  A$ if $\bot\Vdash_{\mathcal M}A$.
%%A model $\mathcal M$ is \emph{regular} if  $g(\mathcal V)\subseteq D(\bot)$. In a regular model $ \vDash_{\mathcal M}  A$ holds iff $\bot\Vdash_{\mathcal M}A$.
We  let $\Gamma \vDash  A$ if for any regular model $\mathcal M$, $\bot \Vdash_{\mathcal M}  \Gamma$ implies $ \bot\Vdash_{\mathcal M}  A$.

 \begin{proposition}[Soundness and completeness \cite{Sol77}]
$\Gamma\vdash_{\Ndva} A$ iff $\Gamma\vDash A$.

 \end{proposition}

We will now exhibit a countermodel to 
$(\forall X. Y\curlyvee Z) \supset  Y\vee Z$.

%We establish the non derivability of $Y\vee Z$ from $\forall X (Y\curlyvee Z)$ by semantic means:
%
%\begin{proposition} $\forall X (Y\curlyvee Z)\nvdash_{\Ndva} Y\vee Z$
%\end{proposition}
\begin{proof}[Proof of Proposition~\ref{rpnequiv}]
Let $\C M=\langle W, \leq, \bot, D, g\rangle$ be the regular model given by: 
\begin{itemize}
\item $  W=\{\bot,\alpha,\beta\}$, $\leq$ is reflexive, 
  $\bot\leq \alpha,\beta$; % and $\alpha,\beta\leq \delta$; % $\alpha,\beta\leq \gamma$;
\item  $D(\bot)=D(\alpha)=D(\beta)=\{ \{\alpha\}, \{\beta\}   \}$; 
%$D(\delta)=D(\alpha)\cup D(\beta)$; %, $D(\gamma)=D(\alpha)\cup D(\beta)$;
  
\item $g(Y)=\{\alpha\}$, $g(Z)=\{\beta\}$ and $g(W)$ is chosen arbitrarily in $D(\bot)$ for $W\neq Y,Z$. %, $g(X)=\{ \{\alpha,\bot\}, \{\beta,\bot\}\}$ for all variable $X\neq Y,Z$; 
\end{itemize}
  
%  Consider the model $\mathcal{M}=\langle W,\leq, w_{1}, D, v\rangle$, where   $v(Y)=\{w_2\}$,  $v(Z)=\{w_3\}$ and $v(X)=\emptyset$ for all $X\neq Y,Z$. %We can depict the $\mathcal{M}$ as follow:

%Let $\mathtt{id}$ be the identity function on the set of variables.
%We say that $A$   is true at a world $w\in W$ in a  model \mbox{$\mathcal{M}=\langle \langle W,\leq\rangle, v\rangle$} (notation $w \vDash_{\mathcal{M}}A$) iff $w,\sigma \Vdash_{\mathcal{M}}A$ for all $\sigma$.

We have $\alpha,\beta\Vdash_{\mathcal M}Y\vee Z$ but
$\bot\not\Vdash_{\mathcal M} Y\vee Z$.
Given $a\in D(\bot)=D(\alpha)=D(\beta)$, we have the following facts (where $\alpha \Vdash_\mathcal{M} A_1,\ldots, A_n$ is short for $\alpha \Vdash_\mathcal{M} A_1$ \emph{and}, \ldots, \emph{and} $\alpha \Vdash_\mathcal{M} A_1$, and similarly for $\nVdash$): 

\begin{itemize}
\item if $a=\{\alpha\}$, then %
\parbox[c]{.5\textwidth}{	\begin{itemize}
	\item $\alpha\Vdash_{\mathcal M[X\mapsto a] } Y\supset X, Z\supset X, X$ 
	\item $\beta\Vdash_{\mathcal M[X\mapsto a] } Y\supset X$ 
	\item $\beta\not\Vdash_{\mathcal M[X\mapsto a] }Z\supset X, X$
	\end{itemize}}
	
\item if $a=\{\beta\}$, then  %
\parbox[c]{.5\textwidth}{	
	\begin{itemize}
	\item $\alpha\Vdash_{\mathcal M[X\mapsto a] } Z\supset X, $ 
	\item $\alpha \not\Vdash_{\mathcal M[X\mapsto a] } Y\supset X, X$ 
	\item $\beta\Vdash_{\mathcal M[X\mapsto a] }Y\supset X, Z\supset X, X$
	\end{itemize}}

\end{itemize}
From these facts we deduce in turn:
\begin{itemize}
\item if $a=\{\alpha\}$, then 
	\begin{itemize}
	\item $\bot\Vdash_{\mathcal M[X\mapsto a] } Y\supset X$ (since $\bot \not \Vdash_{\mathcal M[X\mapsto a] }Y$, $\beta \not \Vdash_{\mathcal M[X\mapsto a] }Y$ and $\alpha \Vdash_{\mathcal M[X\mapsto a] } Y,X$)
	\item $\bot\not\Vdash_{\mathcal M[X\mapsto a] } Z\supset X, X$  (since $\beta\Vdash_{\mathcal M[X\mapsto a] }Z$ but $\beta \not \Vdash_{\mathcal M[X\mapsto a] } X$)
	\end{itemize}
	
\item if $a=\{\beta\}$, then 
	\begin{itemize}
	\item $\bot\Vdash_{\mathcal M[X\mapsto a] } Z\supset X$  (since $\bot \not \Vdash_{\mathcal M[X\mapsto a] }Z$, $\alpha \not \Vdash_{\mathcal M[X\mapsto a] }Z$ and $\beta \Vdash_{\mathcal M[X\mapsto a] } Z,X$)

	\item $\bot\not\Vdash_{\mathcal M[X\mapsto a] } Y\supset X, X$   (since $\alpha\Vdash_{\mathcal M[X\mapsto a] }Y$ but $\alpha \not \Vdash_{\mathcal M[X\mapsto a] } X$)

	\end{itemize}

\end{itemize}

We now deduce that 
\begin{itemize}
\item if $a=\{\alpha\}$, then $\bot\Vdash_{\mathcal M[X\mapsto a] } Y\curlyvee Z$, since the only $\gamma\geq \bot$ such that $\gamma \Vdash_{\mathcal M[X\mapsto a] } Y\supset X$ and for all $\gamma'\geq \gamma$, $\gamma' \Vdash_{\mathcal M[X\mapsto a] } Z\supset X$ is $\alpha$ and $\alpha  \Vdash_{\mathcal M[X\mapsto a] } X$;

\item if $a=\{\beta\}$, then $\bot\Vdash_{\mathcal M[X\mapsto a] } Y\curlyvee Z$, since the only $\gamma\geq \bot$ such that $\gamma \Vdash_{\mathcal M[X\mapsto a] } Y\supset X$ and for all $\gamma'\geq \gamma$, $\gamma' \Vdash_{\mathcal M[X\mapsto a] } Z\supset X$ is $\beta$ and $\beta  \Vdash_{\mathcal M[X\mapsto a] } X$.

\end{itemize}

We deduce then $\bot \Vdash_{\mathcal M} \forall X.Y\curlyvee Z$, and since $\bot \not \Vdash_{\mathcal M}Y\vee Z$, we conclude  
$\bot \not \Vdash_{\mathcal M}(\forall X.Y\curlyvee Z)\supset Y\vee Z$, hence by the definition of $\vDash$ we have that $\nvDash (\forall X.Y\curlyvee Z)\supset Y\vee Z$.
%
%We conclude then $\vDash_{\mathcal M}\forall X (Y\curlyvee Z)$ and $\not\vDash_{\mathcal M}Y\vee Z$.

\end{proof}

\begin{proof}[Proof of Proposition~\ref{noio}]
Take $A= Y$, $B= Z$ and $C = Y\vee Z$. If there were an $\Ndva$-context $\rr{\TT{IO}}: \forall X. Y\curlyvee Z  \vdash^{\emptyset} (Y\curlyvee Z) \llbracket Y\vee Z/X\rrbracket$, then one could derive $Y\vee Z$ from $(Y\vee Z)^*$  by  replacing $[\ ]\dagger\langle A_i\rangle$ with $\rr{\TT{IO}}$  in the context $\rr{\TT{C}}$ of the proof of  proposition \ref{rpequiv}, thereby contradicting Proposition~\ref{rpnequiv}.
\end{proof}

\begin{proof}[Proof of Proposition~\ref{undefdis}]

%Let us fix a variable $X$ and 
%%, and let $\NI[X]$ be the restriction of $\NI$ to formulas in which the only free variable is $X$.
%%Moreover, 
%for any formula $A\in \Ldv$, let $(A)_{X}$ be the formula obtained from $A$ by deleting all quantifiers and replacing all propositional variable by $X$. A simple inductive argument shows then that if $\Gamma\vdash_{\Ndva}A$, then  $(\Gamma_{X})\vdash_{\NI}(A)_{X}$. In particular, if $A=B[Z/Y]$ and $\Gamma\vdash_{\Ndva} A$ is the conclusion of a $\forall$-elimination rule of premiss $\forall YB$, then since $(A)_{X}= (B[Z/Y])_{X}= (\forall YB)_{X}$, we can conclude by the induction hypothesis  that $(\Gamma_{X})\vdash_{\NI}(A)_{X}$.
%
For all $\Gamma \subset \Ld$ such that $\Gamma\vdash_{\Ndva} A\vee B$, either $\Gamma\vdash_{\Ndva}A$ or $\Gamma\vdash_{\Ndva }B$. This can be established following Prawitz' proof of the generalized disjunction property of $\NI$ (see Corollary 6 \cite[p.~56]{Prawitz1965}). We suppose (given the normalization theorem for $\Ndva$) the existence of a normal derivation of $\Gamma\vdash_{\Ndva} A\vee B$ and we reason as Prawitz. The only additional case to consider is that in which  $A\vee B$ is the conclusion of a (sequence of) applications of $\forall$E, but due to the atomic restriction this case can be treated as  the other cases of $\NI$.

% If this is the case then either the premiss $\forall \vec X(A\vee B)$ is conclusion of a $\bot$-elimination (in which case we can deduce $\Gamma\vdash_{\Ndva} \bot$ and thus $\Gamma\vdash_{\Ndva} A$, or $\forall \vec X(A\vee B)$ is conclusion of any other  $\sup$
%forced to occur in $\Gamma$, against the hypothesis. 
%
Now suppose $\vee$ is strongly definable in $\Ndva$ and let $C\dashv\vdash_{\Nda}(
Y
 \vee
 Z
  )$. Then $Y\vdash_{\Ndva} C$, $Z\vdash_{\Ndva} C$ and 
 $C\vdash_{\Ndva} Y \vee Z$ hold; by the above we deduce then that either $C\vdash_{\Ndva} Y$ or  $C\vdash_{\Ndva} Z$ holds,  and so we conclude that either $Y\vdash_{\Ndva}  Z$ or $Z \vdash_{\Ndva} Y$, which is impossible.

% 
% 
% 
%  Let $W\neq Y,Z$; Since $Y\vee Z \vdash W )\supset (Z\supset Y\vee Z) \supset Y\vee Z$ we can conclude
%

\end{proof}

\end{document}